\documentclass[10pt,english, a4paper]{article}
\usepackage[english]{babel}
\usepackage[utf8]{inputenc}
\usepackage[T1]{fontenc}
\usepackage{authblk}
\usepackage[table]{xcolor}
\usepackage[a4paper,top=1.75cm,bottom=1.75cm,left=1.75cm,right=1.75cm,marginparwidth=1.75cm]{geometry}

\usepackage{csquotes}

\usepackage{amsmath,empheq}
\usepackage{graphicx}
\usepackage{amsmath,mathtools}
\usepackage{upgreek}
\usepackage{amssymb}
\usepackage{amsfonts}
\usepackage{amsthm}
\usepackage{enumerate}
\usepackage{mathrsfs}
\usepackage{fontenc}
\usepackage{mathtools}
\usepackage{dsfont}
\usepackage{physics}
\usepackage{bm}

\usepackage{multirow}

\usepackage{enumitem}
\newlist{steps}{enumerate}{1}
\setlist[steps, 1]{label = \underline{Step \arabic*}:}

\usepackage{siunitx}
\usepackage{physics}
\usepackage{dsfont}
\usepackage{booktabs}

\usepackage{lipsum}
\usepackage{caption}
\usepackage{subcaption}
\graphicspath{{Figs/}}
\usepackage[ruled]{algorithm2e}
\usepackage{setspace}

\newcommand\X[1]{{\ensuremath{\mathbf{X}}}_{#1}}
\newcommand\Xhat[1]{{\ensuremath{\widehat{\mathbf{X}}}}_{#1}}
\newcommand\DX[2]{\ensuremath{\mathbf{D}_{#1}\mathbf{X}_{#2}}}
\newcommand\DXhat[2]{\ensuremath{\mathbf{D}_{#1}\widehat{\mathbf{X}}_{#2}}}
\newcommand\DXcheck[2]{\ensuremath{\mathbf{D}_{#1}\widecheck{\mathbf{X}}_{#2}}}
\newcommand\fD{\ensuremath{f^D}}

\newcommand\DZprojection[2]{\ensuremath{\overline{DZ}_{#1}^{#2}}}

\newcommand\Expectation[2][]{\ensuremath{\mathbb{E}_{#1}\left[#2\right]}}
\newcommand\Hess[1]{{\ensuremath{\text{Hess}_{x}}#1}}

\newcommand\CondExpnx[2]{\ensuremath{\mathbb{E}_{#1}^x\left[#2\right]}}
\newcommand\CondExpn[2]{\ensuremath{\mathbb{E}_{#1}\left[#2\right]}}
\newcommand\CondExp[2]{\ensuremath{\mathbb{E}\left[#1\vert #2\right]}}

\DeclareFontFamily{U}{mathx}{\hyphenchar\font45}
\DeclareFontShape{U}{mathx}{m}{n}{
	<5> <6> <7> <8> <9> <10>
	<10.95> <12> <14.4> <17.28> <20.74> <24.88>
	mathx10
}{}
\DeclareSymbolFont{mathx}{U}{mathx}{m}{n}
\DeclareFontSubstitution{U}{mathx}{m}{n}
\DeclareMathAccent{\widecheck}{0}{mathx}{"71}
\DeclareMathAccent{\wideparen}{0}{mathx}{"75}

\DeclareMathOperator*{\argmin}{arg\,min}

\numberwithin{equation}{section}

\usepackage{thmtools}
\newtheoremstyle{break}{\topsep}{\topsep}{\itshape}{}{\bfseries}{}{\newline}{}
\theoremstyle{break}
\newtheorem{theorem}{Theorem}[section]

\newtheorem{lemma}{Lemma}[section]
\newtheorem{remark}{Remark}[section]

\newtheorem{assumption}{Assumption}[section]

\usepackage{calligra}
\DeclareMathAlphabet{\mathcalligra}{T1}{calligra}{m}{n}

\usepackage[mathcal, mathscr]{eucal}
\providecommand{\keywords}[1]{\small\textbf{Keywords: } #1}

\usepackage[backend=biber, style=numeric,giveninits, doi=true, url=false, eprint=true, isbn=false]{biblatex}
\AtEveryBibitem{\clearfield{issn, publisher}}
\AtEveryCitekey{\clearfield{issn, publisher}}
\usepackage[colorlinks=true, allcolors=blue, hypertexnames=false]{hyperref}
\usepackage[nameinlink]{cleveref}

\title{The One Step Malliavin scheme: new discretization of BSDEs implemented with deep learning regressions}

\author[1]{Balint Negyesi\thanks{Corresponding author, \href{mailto:B.Negyesi@tudelft.nl}{B.Negyesi@tudelft.nl}}}
\author[2]{Kristoffer Andersson}
\author[3]{Cornelis W. Oosterlee}
\affil[1]{\small Delft Institute of Applied Mathematics (DIAM), Delft University of Technology}
\affil[2]{Research Group of Scientific Computing, Centrum Wiskunde \& Informatica}
\affil[3]{Mathematical Institute, Utrecht University}

\begin{document}
\maketitle
\sloppy

\begin{abstract}
A novel discretization is presented for forward-backward stochastic differential equations (FBSDE) with differentiable coefficients, simultaneously solving the BSDE and its Malliavin sensitivity problem.
The control process is estimated by the corresponding linear BSDE driving the trajectories of the Malliavin derivatives of the solution pair, which implies the need to provide accurate $\Gamma$ estimates. The approximation is based on a merged formulation given by the Feynman-Kac formulae and the Malliavin chain rule. The continuous time dynamics is discretized with a theta-scheme. In order to allow for an efficient numerical solution of the arising semi-discrete conditional expectations in possibly high-dimensions, it is fundamental that the chosen approach admits to differentiable estimates. Two fully-implementable schemes are considered: the BCOS method as a reference in the one-dimensional framework and neural network Monte Carlo regressions in case of high-dimensional problems, similarly to the recently emerging class of Deep BSDE methods \cite{han_solving_2018, hure_deep_2020}.
An error analysis is carried out to show $\mathds{L}^2$ convergence of order $1/2$, under standard Lipschitz assumptions and additive noise in the forward diffusion.
Numerical experiments are provided for a range of different semi- and quasi-linear equations up to $50$ dimensions, demonstrating that the proposed scheme yields a significant improvement in the control estimations.
\end{abstract}

\tableofcontents

\section{Introduction}\label{sec:intro}

In this paper, we are concerned with the numerical solution of a system of forward-backward stochastic differential equations (FBSDE) where the randomness in the backward equation (BSDE) is driven by a forward stochastic differential equation (SDE). These systems are written in the general form
\begin{subequations}\label{eq:fbsde} 
	\begin{align}	
		X_t &= x_0 + \int_0^t \mu(s, X_s)\mathrm{d}s + \int_0^t \sigma(s, X_s)\mathrm{d}W_s,\label{eq:fbsde:sde}\\
		Y_t &= g(X_T) + \int_t^T f(s, X_s, Y_s, Z_s)\mathrm{d}s - \int_t^T Z_s\mathrm{d}W_s\label{eq:fbsde:bsde},
	\end{align}
\end{subequations}
where $\{W_t\}_{0\leq t\leq T}$ is a $d$-dimensional Brownian motion and $\mu:[0, T]\times \mathds{R}^{d\times 1}\to\mathds{R}^{d\times 1}$, $\sigma:[0, T]\times \mathds{R}^{d\times 1}\to \mathds{R}^{d\times d}$, $g:\mathds{R}^{d\times 1}\to\mathds{R}^q$ and $f:[0, T]\times \mathds{R}^{d\times 1}\times \mathds{R}^q\times \mathds{R}^{q\times d}\to\mathds{R}^q$ are all deterministic mappings of time and space, with some fixed $T>0$. Adhering to the stochastic control terminology, we often refer to $Z$ as the \emph{control process}. We shall work under the standard well-posedness assumptions of Pardoux and Peng \cite{pardoux_backward_1992}, which require Lipschitz continuity of the corresponding coefficients in order to ensure the existence of a unique solution pair $\{(Y_t, Z_t)\}_{0\leq t\leq T}$ adapted to the augmented natural filtration. The main motivation to study FBSDE systems lies in their connection with parabolic, second-order partial differential equations (PDE), generalizing the well-known Feynman-Kac relations to non-linear settings. Indeed, considering the quasi-linear, parabolic terminal problem
\begin{align}\label{eq:parabolic_pde}
	\begin{split}
		\partial_t u(t, x) + \frac{1}{2}\Tr{\sigma\sigma^T(t, x) \Hess{u}(t, x)} + \braket{\mu(t, x)}{\nabla_x u(t, x)} + f(t, x, u, \nabla_x u(t, x)\sigma(t, x)) &= 0\\
		u(T, x) &= g(x),
	\end{split}
\end{align}
the Markov solution to \autoref{eq:fbsde} coincides with the solution of \autoref{eq:parabolic_pde} in an almost sure sense, provided by the \emph{non-linear Feynman-Kac} relations
\begin{align}\label{thm:feynman-kac}
	Y_t = u(t, X_t),\qquad Z_t = \nabla_x u(t, X_t) \sigma(t, X_t).
\end{align}
Consequently, the BSDE formulation provides a stochastic representation to the simultaneous solution of a parabolic problem and its gradient, which is an advantageous feature for several applications in stochastic control and finance, where sensitivities play a fundamental role.
These relations can be extended to \emph{viscosity solutions} in case \autoref{eq:parabolic_pde} does not admit to a classical solution -- see \cite{pardoux_backward_1992}. Moreover, it is known -- see \cite{pardoux_backward_1992, el_karoui_backward_1997, hu_malliavin_2011, mastrolia_malliavin_2017} -- that under suitable regularity assumptions the solution pair of the backward equation is differentiable in the Malliavin sense \cite{nualart_malliavin_2006}, and the Malliavin derivatives $\{(D_sY_t, D_sZ_t)\}_{0\leq s, t\leq T}$ satisfy a linear BSDE themselves, where the $Z$ process admits to a continuous modification provided by $Z_t=D_tY_t$.

From a numerical standpoint, the main challenge in solving BSDEs stems from the approximation of conditional expectations. Indeed, a discretization of the backward equation in \autoref{eq:fbsde:bsde} yields a sequence of recursively nested conditional expectations at each point in the discretized time window. Over the years, several methods have been proposed to tackle the solution of the FBSDE system using: PDE methods in \cite{ma_solving_1994}; forward Picard iterations in \cite{bender_forward_2007}; quantization techniques in \cite{bally_quantization_2003}; chaos expansion formulas in \cite{briand_simulation_2014}; Fourier cosine expansions in \cite{ruijter_fourier_2015, ruijter_numerical_2016} and regression Monte Carlo approaches in \cite{gobet_regression-based_2005, bouchard_discrete-time_2004, bender_least-squares_2012}. These methods have shown great results in low-dimensional settings, however, the majority of them suffers from the curse of dimensionality, meaning that their computational complexity scales exponentially in the number of dimensions. Although, regression Monte Carlo methods have been successfully proven to overcome this burden, they are difficult to apply beyond $d=10$ dimensions due to the necessity of a finite regression basis. The primary challenge in the numerical solution of BSDEs is related to the approximation of the $Z$ process. In particular, the standard backward Euler discretization results in a conditional expectation estimate of $Z$ which scales inverse proportionally with the step size of the time discretization -- see \cite{bouchard_discrete-time_2004}. This phenomenon poses a significant amount of difficulty in least-squares Monte Carlo frameworks, as the corresponding regression targets have diverging conditional variances in the continuous limit.

Recently, the field has received renewed attention due to the pioneering paper of Han et al. \cite{han_solving_2018}, in which they reformulate the backward discretization in a forward fashion, parametrize the control process of the solution by deep neural networks and train the resulting sequence of networks in a global optimization given by the terminal condition of \autoref{eq:fbsde:bsde}. Their method has enjoyed various modifications and extensions, see, e.g., \cite{fujii_asymptotic_2019, beck_machine_2019}. In particular, Huré et al. in \cite{hure_deep_2020} proposed an alternative where the optimization of the sequence of neural networks is done in a backward recursive manner, similarly to classical regression Monte Carlo approaches.
We refer to the class of these deep learning based formulations as \emph{Deep BSDE} methods.
Although such Deep BSDE solvers have shown remarkable empirical results in solving high-dimensional problems, they struggle to solve the whole FBSDE system in \autoref{eq:fbsde:bsde} and are merely focused on the PDE problem. In particular, the approach of \cite{han_solving_2018} solely captures the solution pair at $t=0$; whereas the extension of \cite{hure_deep_2020} gives good approximations at future time steps, its accuracy in the $Z$ part of the solution is significantly worse. 
The total approximation errors of such Deep BSDE methods have been investigated in \cite{han_convergence_2020, hure_deep_2020, germain_approximation_2021}. The results in \cite{han_convergence_2020} provide a \emph{posteriori estimate} driven by the error in the terminal condition, whereas the analyses in \cite{hure_deep_2020, germain_approximation_2021} show that due to the universal approximation theorem (UAT) of deep neural networks, the total approximation error of neural network parametrizations is consistent with the discretization in terms of regression biases.

The main motivation behind the present paper roots in the observations above. In order to provide more accurate solutions for the $Z$ process, we exploit the aforementioned relation between the Malliavin derivative of $Y$ and the control process by solving the linear BSDE driving the trajectories of $DY$. Hence, we are faced with the solution of one scalar-valued BSDE and one $d$-dimensional BSDE at each point in time. This raises the need for a new discrete scheme, which we call the \emph{One Step Malliavin (OSM)} scheme. The discretization of the linear BSDE of the Malliavin derivatives is based on a merged formulation of the Feynman-Kac formulae in \autoref{thm:feynman-kac} and the chain rule formula of Malliavin calculus \cite{nualart_malliavin_2006}. As we shall see, the resulting discrete time approximation of the $Z$ process possesses the same order of conditional variance as the ones of the $Y$ process, making the scheme significantly more attractive in a regression Monte Carlo framework compared to classical Euler discretizations. On the other hand, our formulation carries an extra layer of difficulty, in that we are forced to approximate the \textit{"the $Z$ of the $Z$, i.e. $\Gamma$ processes"} \cite[Pg.1184]{gobet_adaptive_2017} in the Malliavin BSDE which are, in light of \autoref{thm:feynman-kac}, related to the Hessian matrix of the solution of the corresponding parabolic problem \autoref{eq:parabolic_pde}. In this regard, our setting shares similarities with \emph{second-order backward SDEs (2BSDEs)} \cite{cheridito_second-order_2007} and fully non-linear problems \cite{fahim_probabilistic_2011}.
We analyze the discrete time approximation errors and show that under certain assumptions the new scheme has the same $\mathds{L}^2$ convergence rate of order $1/2$ as the backward Euler scheme of BSDEs \cite{bouchard_discrete-time_2004}.

Two fully-implementable approaches are investigated to solve the resulting discretization. First, we provide an extension to the BCOS method \cite{ruijter_fourier_2015} and approximate solutions to one-dimensional problems by Fourier cosine expansions. Ultimately, the presence of $\Gamma$ estimates induces $d^2$ many additional conditional expectations to be approximated at each point in time, which makes the OSM scheme less tractable for classical Monte Carlo parametrizations when $d$ is large. Thereafter, inspired by the encouraging results of Deep BSDE methods in case of high-dimensional equations, we propose a neural network least-squares Monte Carlo approach similar to the one of \cite{hure_deep_2020}, where the $Y$, $Z$ and $\Gamma$ processes are parametrized by fully-connected, feedforward deep neural networks. Subsequently, parameters of these networks are optimized in a recursive fashion, backwards over time, where at each time step two distinct gradient descent optimizations are performed, minimizing losses corresponding to the aforementioned discretization.  Motivated by the UAT property of neural networks in Sobolev spaces, similarly to \cite{hure_deep_2020}, we consider two variants of the latter approach: one in which the $\Gamma$ process is parametrized by a matrix-valued deep neural network; and one in which the $\Gamma$ process is approximated as the Jacobian of the parametrization of the $Z$ process, inspired by \autoref{thm:feynman-kac}. The total approximation error is investigated similarly to \cite{germain_approximation_2021, hure_deep_2020} and shown to be consistent with the discretization under the assumption of perfectly converging gradient descent iterations. We demonstrate the accuracy and robustness of our problem formulation with numerical experiments. In particular, using BCOS as a benchmark method for one-dimensional problems, we empirically assess the regression errors induced by gradient descent. We provide examples up to $d=50$ dimensions.

The rest of the paper is organized as follows. In \autoref{sec:bsdes_and_malliavin} we provide the necessary theoretical foundations, followed by \autoref{sec:discrete_scheme} where the new discrete scheme is formulated. In \autoref{sec:error_analysis} a discrete time approximation error analysis is given, bounding the total discretization error of the proposed scheme. \Cref{sec:regression} is concerned with the implementation of the discretization scheme, giving two fully-implementable approaches for the arising conditional expectations. First, the BCOS method \cite{ruijter_fourier_2015} is extended in case of one-dimensional problems, then a Deep BSDE \cite{han_solving_2018, hure_deep_2020} approach is formulated for high-dimensional equations. A complete regression error analysis is provided, building on the universal approximation properties of neural networks.
Our analysis is concluded by numerical experiments presented in \autoref{sec:numerical_experiments}, which confirm the theoretical results and showcase great accuracy over a wide range of different problems.

\section{Backward stochastic differential equations and Malliavin calculus}\label{sec:bsdes_and_malliavin}
In the following section we introduce the notions of BSDEs and Malliavin calculus used throughout the paper.
\subsection{Preliminaries}
Let us fix $0\leq T<\infty$ and $d, q, n, k\in \mathds{N}^+$. We are concerned with a filtered probability space $\left(\Omega, \mathcal{F}, \mathds{P}, \{\mathcal{F}\}_{0\leq t\leq T}\right)$, where $\mathcal{F}=\mathcal{F}_T$ and $\{\mathcal{F}\}_{0\leq t\leq T}$ is the natural filtration generated by a $d$-dimensional Brownian motion $\{W_t\}_{0\leq t\leq T}$ augmented by $\mathds{P}$-null sets of $\Omega$. In what follows, all equalities concerning $\mathcal{F}_t$-measurable random variables are meant in the $\mathds{P}$-a.s. sense and all expectations -- unless otherwise stated -- are meant under $\mathds{P}$. Throughout the whole paper we rely on the following notations
\begin{itemize}
	\item $\abs{x}\coloneqq \Tr{x^Tx}$ for the Frobenius norm of any $x\in \mathds{R}^{q\times d}$. In case of scalar and vector inputs this coincides with the standard Euclidean norm. Additionally, we put $\braket{x}{y}$ for the Euclidean inner product of $x, y\in\mathds{R}^d$.
	\item $\mathds{S}^p(\mathds{R}^{q\times d})$ for the space of continuous and progressively measurable stochastic processes $Y:\Omega\times [0, T]\to \mathds{R}^{q\times d}$ such that $\Expectation{\sup_{0\leq t\leq T}\abs{Y}^p}<\infty$.
	\item \sloppy $\mathds{H}^p(\mathds{R}^{q\times d})$ for the space of progressively measurable stochastic processes $Z: \Omega \times [0, T]\to \mathds{R}^{q\times d}$ such that $\Expectation{\left(\int_0^T \abs{Z_t}^2\mathrm{d}t\right)^{p/2}}<\infty$.
	\item $\mathds{L}^p_{\mathcal{F}_{t}}(\mathds{R}^{q\times d})$ for the space of $\mathcal{F}_t$-measurable random variables $\xi: \Omega\to \mathds{R}^{q\times d}$ such that $\Expectation{\abs{\xi}^p}<\infty$.
	\item $L^2([0, T]; \mathds{R}^q)$ for the Hilbert space of deterministic functions $h:[0, T]\to \mathds{R}^q$ such that $\int_0^T \abs{h(t)}^2\mathrm{d}t<\infty$. Additionally, we denote its inner product by $\braket{h}{g}_{L^{2}}\coloneqq \int_0^T h(t)g(t)\mathrm{d}t$.
	\item $\nabla_x f\coloneqq \left(\pdv{f}{x_1}, \dots, \pdv{f}{x_d}\right)$ for the gradient of a scalar-valued, multivariate function $(t, x, y, z)\mapsto f(t, x, y, z)$ with respect to $x\in\mathds{R}^d$, defined as a row vector, and analogously for $\nabla_y f, \nabla_z f$. Similarly, we denote the Jacobian matrix of a vector-valued function $\psi:\mathds{R}^d\to\mathds{R}^q$ by $\nabla_x \psi\in\mathds{R}^{q\times d}$. For notational convenience, we set the Jacobian matrix of row and column vector-valued functions in the same fashion.
	\item $C_b^k(\mathds{R}^d; \mathds{R}^q), C_p^k(\mathds{R}^d; \mathds{R}^q)$ for the set of $k$-times continuously differentiable functions $\varphi: \mathds{R}^d\to\mathds{R}^q$ such that all partial derivatives up to order $k$ are bounded or have polynomial growth, respectively.
	\item $\CondExpn{n}{\Phi}\coloneqq \CondExp{\Phi}{\mathcal{F}_{t_{n}}}$ for conditional expectations with respect to the natural filtration, given a time partition $0=t_0<t_1<\dots<t_N=T$. We occasionally use the notation $\CondExpnx{n}{\Phi}\coloneqq\CondExp{\Phi}{X_{t_{n}} = x}$ when the filtration is generated by a Markov process $X$.
	\item $\mathbf{1}_{q, d}, \mathbf{0}_{q, d}$ for $\mathds{R}^{q\times d}$ matrices full of ones and zeros, respectively.
\end{itemize}
By slight abuse of notation we put $\mathds{S}^p(\mathds{R})\coloneqq \mathds{S}^p(\mathds{R}^{1\times 1})$, $\mathds{H}^p(\mathds{R}^{d})\coloneqq \mathds{H}^p(\mathds{R}^{1\times d})$, $\mathbf{1}_{d}\coloneqq \mathbf{1}_{1, d}$ and $\mathbf{0}_d\coloneqq \mathbf{0}_{1\times d}$.

We recall the most important notions of Malliavin differentiability and refer to \cite{nualart_malliavin_2006} for a more detailed account on the subject. Consider the space of random processes $W(h)\coloneqq \int_0^T h(t)\mathrm{d}W_t$ with $h\in L^2([0, T]; \mathds{R}^n)$. Let us now define the subspace $\mathcal{R}\subseteq\mathds{L}^2_{\mathcal{F}_{T}}$ of smooth, scalar-valued random variables which are of the form $\Phi=\varphi(W(h_1), \dots, W(h_d))$ with some $\varphi\in C_p^\infty(\mathds{R}^d; \mathds{R})$. The Malliavin derivative of $\Phi$ is then defined as the $\mathds{R}^{1\times n}$-valued stochastic process $D_s\Phi\coloneqq \sum_{i=1}^d \partial_i \varphi(W(h_1), \dots, W(h_d))h_i(s)$. The derivative operator can be extended to the closure of $\mathcal{R}$ with respect to the norm
\begin{align*}
	\norm{\Phi}_{\mathds{D}^{1, p}}\coloneqq \left(\Expectation{\abs{\Phi}^p + \left(\int_0^T \abs{D_s\Phi}^2\mathrm{d}s\right)^{p/2}}\right)^{1/p},
\end{align*}
see \cite[Prop.1.2.1]{nualart_malliavin_2006}.
We denote this closure as the space of Malliavin differentiable, $\mathds{R}$-valued random variables by $D^{1, p}(\mathds{R})$.
For the space of vector-valued $\Phi=(\Phi_1, \dots, \Phi_q)$ Malliavin differentiable random variables, we put $\Phi\in\mathds{D}^{1, p}(\mathds{R}^q)$ when $\Phi_i\in \mathds{D}^{1, p}(\mathds{R})$ for each $i=1, \dots, q$. The Malliavin derivative $D_s\Phi\in \mathds{R}^{q\times n}$ is then the matrix-valued stochastic process whose $i$'th row is $D_s\Phi_i$.
The final result which extends the chain rule of elementary calculus to the Malliavin differentiation operator is fundamental for the present paper, essentially enabling the formulation of the upcoming discrete scheme.
\begin{lemma}[Malliavin chain rule lemma]\label{lemma:malliavin_chain_rule}
	Let $\psi\in C_b^1(\mathds{R}^d; \mathds{R}^q)$ and fix $p\geq 1$. Consider $F\in \mathds{D}^{1, p}(\mathds{R}^d)$. Then $\psi(F)\in \mathds{D}^{1, p}(\mathds{R}^q)$, furthermore for each $0\leq s\leq T$
	\begin{align}
		D_s\psi(F) = \nabla_x \psi(F)D_sF.
	\end{align}
\end{lemma}
The lemma can be relaxed to the case where $\psi$ is only Lipschitz continuous -- see \cite[Prop.1.2.4]{nualart_malliavin_2006}.

\subsection{Backward stochastic differential equations}

We first provide the necessary theoretical foundations for the well-posedness of the underlying FBSDE system in \autoref{eq:fbsde} guaranteeing the existence of a unique solution triple. Given the stronger assumptions later required for their Malliavin differentiability, we restrict the presentation to standard Lipschitz assumptions. For a more general exposure we refer to \cite{chassagneux_numerical_2016} and the references therein.

It is well-known -- see, e.g., \cite{karatzas_brownian_1998} -- that the SDE in \autoref{eq:fbsde:sde} admits to a unique strong solution $\{X_t\}_{0\leq t\leq T}\in \mathds{S}^p(\mathds{R}^{d\times 1})$ whenever $x_0\in\mathds{L}^p_{\mathcal{F}_{0}}(\mathds{R}^{d\times 1})$ and $\mu, \sigma$ are Lipschitz continuous in the spatial variable, i.e.
\begin{align}
	\abs{\mu(t, x_1) - \mu(t, x_2)} + \abs{\sigma(t, x_1) - \sigma(t, x_2)}\leq L_{\mu, \sigma} \abs{x_1 - x_2}
\end{align}
for all $t\in[0, T]$, $x_1,x_2\in\mathds{R}^{d\times 1}$, with some $L_{\mu, \sigma}>0$. Additionally, the solution $\{X_t\}_{0\leq t\leq T}$ satisfies the following estimates for all $p\geq 1$
\begin{align}\label{eq:sec2:sup_holder:x}
	\Expectation{\sup_{0\leq t\leq T} \abs{X_t}^p}\leq C_p,\qquad \Expectation{\abs{X_t-X_s}^p}\leq C_p\abs{t-s}^{p/2},
\end{align}
with constant $C_p$ only depending on $p, T, d$. In case of the Arithmetic Brownian Motion (ABM) with constant $\mu$ and $\sigma$, \autoref{eq:fbsde:sde} admits to the unique solution $X_t=\mu t + \sigma W_t$. In particular, the Malliavin chain rule formula in \autoref{lemma:malliavin_chain_rule} implies that $D_sX_t=\mathds{1}_{s\leq t}\sigma$.

The well-posedness of the backward equation in \autoref{eq:fbsde:bsde} is guaranteed by the Lipschitz continuity of the driver, on top of the polynomial growth of the terminal condition
\begin{align}
	\abs{f(t, x, y_1, z_1) - f(t, x, y_2, z_2)}\leq L_{f, g}\left(\abs{y_1-y_2} + \abs{z_1-z_2}\right),\qquad \abs{f(t, x, y, z)} + \abs{g(x)}\leq L_{f, g}\left(1+\abs{x}^p\right),
\end{align}
for any $t\in[0, T]$, $y_1, y_2\in \mathds{R}^q$, $z_1, z_2\in \mathds{R}^{q\times d}$, with some $L_{f, g}>0$ and $p\geq 2$. These conditions, combined with the ones for the SDEs above, imply the existence of a unique solution pair $Y\in \mathds{S}^p(\mathds{R}^{q})$, $Z\in\mathds{H}^p(\mathds{R}^{q\times d})$ satisfying \autoref{eq:fbsde:bsde}. Let us now fix $q=1$ and restrict the further analysis to scalar-valued backward equations.
Thereafter, under the aforementioned conditions, the FBSDE system in \autoref{eq:fbsde} admits to a unique solution triple $\{(X_t, Y_t, Z_t)\}_{0\leq t\leq T}\in \mathds{S}^p(\mathds{R}^{d\times 1})\times \mathds{S}^p(\mathds{R})\times \mathds{H}^p(\mathds{R}^{1\times d})$.

\subsection{Malliavin differentiable FBSDE systems}
This paper is focused on a special class of FBSDE systems such that the solution triple $\{(X_t, Y_t, Z_t)\}_{0\leq t\leq T}$ is differentiable in the Malliavin sense. The Malliavin differentiability of the forward equation is guaranteed by the following theorem due to Nualart in \cite[Thm.2.2.1]{nualart_malliavin_2006}.
\begin{lemma}[Malliavin differentiability of SDEs, \cite{nualart_malliavin_2006}]\label{thm:malliavin:sde}
	Let $x_0\in\mathds{L}^p_{\mathcal{F}_{0}}(\mathds{R}^{d\times 1})$, $\mu\in C_b^{0, 1}([0, T]\times \mathds{R}^{d\times 1}; \mathds{R}^{d\times 1})$, $\sigma\in C_b^{0, 1}([0, T]\times \mathds{R}^{d\times 1}; \mathds{R}^{d\times d})$ and $\mu(t, 0), \sigma(t, 0)$ be uniformly bounded for all $0\leq t\leq T$. Put $\{X_t\}_{0\leq t\leq T}$ for the unique solution of \autoref{eq:fbsde:sde}. Then for all $t\in[0, T]$ $X_t\in\mathds{D}^{1, p}(\mathds{R}^{d\times 1})$ and there exists a continuous modification of its Malliavin derivative $\{D_sX_t\}_{0\leq s, t\leq T}\in \mathds{S}^p(\mathds{R}^{d\times d})$ which satisfies the linear SDE
	\begin{align}\label{eq:sec2:malliavin_sde}
		D_sX_t = \mathds{1}_{s\leq t}\left\{\sigma(s, X_s) + \int_s^t \nabla_x \mu(r, X_r)D_sX_r\mathrm{d}r + \int_s^t \nabla_x\sigma(r, X_r)D_sX_r\mathrm{d}W_r.\right\},
	\end{align}
	where $\nabla_x \sigma$ denotes a $\mathds{R}^{d\times d\times d}$-valued tensor with $\left[\nabla_x \sigma\right]_{ijk}=\partial_k \left[\sigma\right]_{ij}$.
	Furthermore, there exists a constant $C_p$, such that
	\begin{align}\label{mean-squared continuity:dx}
		\sup_{s\in[0, T]}\Expectation{\sup_{t\in[s, T]} \abs{D_sX_t}^p}\leq C_p, \qquad\Expectation{\abs{D_sX_r - D_sX_t}^p}\leq C_p\abs{r-t}^{p/2}.
	\end{align}
\end{lemma}
The main implication of the proposition above is that under relatively mild assumptions on the bounded continuous differentiability of the coefficients in \autoref{eq:fbsde:sde}, the Malliavin derivative of the solution satisfies a linear SDE, where the random coefficients depend on the solution of the SDE itself. Intriguingly, a similar assertion can be made about the solution pair of the backward equation in \autoref{eq:fbsde:bsde}, which -- on top of establishing their Malliavin differentiability -- also creates a connection between the Malliavin derivative $DY$ and the control process. This is stated by the following theorem originally from Pardoux and Peng in \cite{pardoux_backward_1992}, which we state under the loosened conditions of El Karoui et al. \cite[Prop.5.9]{el_karoui_backward_1997}.
\begin{theorem}[Malliavin differentiability of BSDEs, \cite{el_karoui_backward_1997}]\label{thm:malliavin:bsde}
	Let the coefficients of \autoref{eq:fbsde:sde} satisfy the conditions of \autoref{thm:malliavin:sde}
	and assume $f\in C_b^{0, 1, 1, 1}([0, T]\times \mathds{R}^{d\times 1}, \mathds{R}, \mathds{R}^{1\times d}; \mathds{R})$, $g\in C_b^1(\mathds{R}^{d\times 1}; \mathds{R})$.
	Fix $p\geq 2$.
	Put $\{(Y_t, Z_t)\}_{0\leq t\leq T}$ for the unique solution pair of \autoref{eq:fbsde:bsde}. Then for all $t\in[0, T]$ $Y_t\in\mathds{D}^{1, 2}(\mathds{R})$, $Z_t\in\mathds{D}^{1, 2}(\mathds{R}^{1\times d})$ and there exist modifications of their Malliavin derivatives $\{D_sY_t\}_{0\leq s,t\leq T}\in \mathds{S}^p(\mathds{R}^{1\times d})$, $\{D_sZ_t\}_{0\leq, s,t\leq T}\in \mathds{H}^p(\mathds{R}^{d\times d})$ which satisfy the following linear BSDE
	\begin{align}\label{eq:sec2:malliavin_bsde}
		\begin{split}
			D_sY_t &= \begin{aligned}[t]
				&\nabla_x g(X_T)D_sX_T\\
				&+ \int_t^T \nabla_x f(r, X_r, Y_r, Z_r)D_sX_r + \nabla_y f(r, X_r, Y_r, Z_r)D_sY_r + \nabla_z f(r, X_r, Y_r, Z_r)D_sZ_r\mathrm{d}r\\
				&-\int_t^T D_sZ_r\mathrm{d}W_r,\quad 0\leq s\leq t\leq T,
			\end{aligned}\\
			D_sY_t &=\mathbf{0}_d,\quad D_sZ_t = \mathbf{0}_{d, d},\quad 0\leq t< s\leq T.
		\end{split}
	\end{align}
	Furthermore, there exists a continuous modification of the control process such that $Z_t=D_tY_t$ almost surely for all $0\leq t\leq T$.
\end{theorem}
We emphasize the linearity of \autoref{eq:sec2:malliavin_bsde} and remark that the corresponding random coefficients of the linear equation depend on the solution of \autoref{eq:fbsde}. Henceforth, in light of \autoref{thm:malliavin:sde} and \autoref{thm:malliavin:bsde}, we define $\{D_sX_t\}_{0\leq s,t\leq T}$ and $\{Z_t\}_{0\leq t\leq T}$ as the versions of the corresponding Malliavin derivatives satisfying \autoref{eq:sec2:malliavin_sde} and \autoref{eq:sec2:malliavin_bsde}, respectively. For the rest of the paper, in order to ease the presentation, we introduce the notations $\X{t}\coloneqq (X_t, Y_t, Z_t)$, $\DX{s}{t}\coloneqq (D_sX_t, D_sY_t, D_sZ_t)$ and $\fD(t, \X{t}, \DX{s}{t})\coloneqq \nabla_x f(t, \X{t})D_sX_t+\nabla_y f(t, \X{t})D_sY_t + \nabla_z f(t, \X{t})D_sZ_t$ for all $0\leq s, t\leq T$.

\paragraph{Path regularity and Hölder continuity.} For $\{X_t\}_{0\leq t\leq T}\in \mathds{S}^p(\mathds{R}^{d\times 1})$ we have that the solution of the forward SDE is a continuous $\mathds{R}^{d\times 1}$-valued random process which is bounded in the supremum norm. Similar statements can be made about its Malliavin derivative $\{D_sX_t\}_{0\leq s,t\leq T}$. In particular, the Hölder regularity estimates in \autoref{eq:sec2:sup_holder:x} and \autoref{mean-squared continuity:dx} ensure that the corresponding processes are not just continuous but also have a modification admitting to $\alpha$-Hölder continuous trajectories of order $\alpha\in (0, 1/2)$ provided by the Kolmogorov-Chentsov theorem -- see, e.g., \cite{karatzas_brownian_1998}. Since the $1/2$-Hölder regularity of $(Y, Z)$ plays a crucial role in the convergence analysis of the discrete scheme -- see \autoref{thm:osm:discretization} in particular --, we elaborate on the conditions under which the continuous parts of the solutions to \autoref{eq:fbsde:bsde} and \autoref{eq:sec2:malliavin_bsde} admit to similar estimates. 
Indeed, one can show that if the solutions $(Y, Z)\in\mathds{S}^p(\mathds{R})\times \mathds{H}^p(\mathds{R}^{d\times 1})$ of \autoref{eq:fbsde:bsde} satisfy the condition $\sup_{0\leq t\leq T}\Expectation{\abs{Z_t}^p}<\infty$ then there exists a constant $C_p$ such that
\begin{align}
	\Expectation{\abs{Y_t-Y_s}^p}\leq C_p\abs{t-s}^{p/2},
\end{align}
see \cite[Corollary 2.7]{hu_malliavin_2011}.
In particular, the $Y$ process admits to an $\alpha$-Hölder continuous modification of order $\alpha\in (0, 1/2-1/p)$. Under the conditions of \autoref{thm:malliavin:bsde}, this is naturally guaranteed, and for $p=2$ it implies the \emph{mean-squared continuity} of the $Y$ process. Moreover, the $Z$ process admits to a continuous modification solving \autoref{eq:sec2:malliavin_bsde}, which guarantees $Z\in\mathds{S}^p(\mathds{R}^{1\times d})$ and, in particular, boundedness in the supremum norm. Under stronger assumptions one can also establish a similar path regularity result of the control process. Imkeller and Dos Reis in \cite[Thm.5.5]{imkeller_path_2010} show that with additional conditions, essentially requiring second-order bounded differentiability of the corresponding coefficients $\mu, \sigma, f$ and $g$, the following also holds for all $p\geq 2$
\begin{align}\label{eq:sec2:sup_holder:z}
	\Expectation{\abs{Z_t-Z_s}^p}\leq C_p \abs{t-s}^{p/2}.
\end{align}
Hu et al. prove a similar result in \cite[Thm.2.6]{hu_malliavin_2011} under slightly different assumptions in the general non Markovian framework. We omit the explicit presentation of the necessary conditions for \autoref{eq:sec2:sup_holder:z} to hold, nevertheless emphasize that \autoref{ass:error_analysis} of the convergence analysis in \autoref{sec:error_analysis} ensures the path regularity of the $Z$ process and in particular implies mean-squared continuous trajectories.

\section{The discrete scheme}\label{sec:discrete_scheme}
In the following section the proposed discretization scheme is introduced. The objective of the discretization is to simultaneously solve the pair of FBSDE systems given by \autoref{eq:fbsde} and the FBSDE system of its Malliavin derivatives provided by \autoref{thm:malliavin:sde} and \autoref{thm:malliavin:bsde}. Therefore, we are concerned with the solution to the following pair of FBSDE systems
\begin{subequations}\label{eq:fbsde_fbsde}
	\begin{align}
		X_t &= x_0 + \int_0^t \mu(r, X_r)\mathrm{d}r + \int_0^t \sigma(r, X_r)\mathrm{d}W_r,\label{eq:fbsde_fbsde:sde}\\
		Y_t &= g(X_T) + \int_t^T f(r, \X{r})\mathrm{d}r - \int_t^T Z_r\mathrm{d}W_r,\label{eq:fbsde_fbsde:bsde}\\
		D_sX_t &= \begin{aligned}[t]
			\mathds{1}_{s\leq t}\left[\sigma(s, X_s) + \int_s^t \nabla_x \mu(r, X_r)D_sX_r\mathrm{d}r + \int_s^t \nabla_x\sigma(r, X_r)D_sX_r\mathrm{d}W_r\right],\label{eq:fbsde_fbsde:malliavin_sde}
		\end{aligned}\\
		D_sY_t &= \mathds{1}_{s\leq t}\left[\nabla_x g(X_T)D_sX_T + \int_t^T \fD(r, \X{r}, \DX{s}{r})\mathrm{d}r - \int_t^T D_sZ_r\mathrm{d}W_r\right]\label{eq:fbsde_fbsde:malliavin_bsde}.
	\end{align}
\end{subequations}
The solution is a pair of triples of stochastic processes $\{(X_t, Y_t, Z_t)\}_{0\leq t\leq T}$ and $\{(D_sX_t, D_sY_t, D_sZ_t)\}_{0\leq s,t\leq T}$ such that \autoref{eq:fbsde_fbsde} holds $\mathds{P}$ almost surely. Consider a discrete time partition $\pi^N\coloneqq \{t_0, \dots, t_N\}$ with $0=t_0<t_1<\dots<t_N=T$ and set $\Delta W_n\coloneqq W_{t_{n+1}} - W_{t_{n}}$, $\Delta t_n \coloneqq t_{n+1} - t_n$, $\abs{\pi}\coloneqq \max_{0\leq n\leq N-1} t_{n+1}-t_n$. We denote the discrete time approximations by $(X_n^\pi, Y_n^\pi, Z_n^\pi)$ and $(D_nX_m^\pi, D_nY_{m}^\pi, D_nZ_{m}^\pi)$ for each $0\leq n, m \leq N$.

The forward component in \autoref{eq:fbsde_fbsde:sde} is approximated by the classical Euler-Maruyama scheme, i.e.,
\begin{align}\label{scheme:sde:euler}
	X_0^\pi\coloneqq x_0,\quad X_{n+1}^\pi \coloneqq X_n^\pi + \mu(t_n, X_n^\pi)\Delta t_n + \sigma(t_n, X_n^\pi)\Delta W_n^\pi,
\end{align}
for each $n=0, \dots, N-1$. It is well-known -- see, e.g., \cite{kloeden_numerical_1992} -- that under standard Lipschitz assumptions on the drift and diffusion coefficients, these estimates admit to
\begin{align}\label{eq:euler:convergence_rate}
	\limsup_{\abs{\pi}\to 0} \frac{1}{\abs{\pi}}\Expectation{\abs{X_{t_{n}}-X_n^\pi}^2}<\infty.
\end{align}
Classically, the backward component in \autoref{eq:fbsde_fbsde:bsde} is approximated in two steps. In order to meet the necessary adaptivity requirements of the solution pair $(Y, Z)$, one takes appropriate conditional expectations of \autoref{eq:fbsde_fbsde:bsde} and the same equation multiplied with the Brownian increment $\Delta W_n^T$. Using standard properties of stochastic integrals, It\^{o}'s isometry and a \emph{theta-discretization} of the remaining time integrals with parameters $\vartheta_y, \vartheta_z>0$ subsequently give -- see, e.g., \cite{ruijter_fourier_2015}
\begin{subequations}\label{scheme:theta_bsde}
	\begin{align}
		Y_N^\pi &= g(X_N^\pi),\quad Z_N^\pi=\nabla_x g(X_N^\pi)\sigma(t_N, X_N^\pi),\\
		Z_n^\pi &= -\frac{1-\vartheta_z}{\vartheta_z}\CondExpn{n}{Z_{n+1}^\pi} + \frac{1}{\Delta t_n \vartheta_z}\CondExpn{n}{\Delta W_n^T Y_{n+1}^\pi} + \frac{1-\vartheta_z}{\vartheta_z}\CondExpn{n}{\Delta W_n^T f(t_{n+1}, \X{n+1}^\pi)},\label{scheme:theta_bsde:z}\\
		Y_n^\pi &= \Delta t_n \vartheta_y f(t_n, X_n^\pi, Y_n^\pi, Z_n^\pi) + \CondExpn{n}{Y_{n+1}^\pi} + \Delta t_n (1-\vartheta_y)\CondExpn{n}{f(t_{n+1}, \X{n+1}^\pi)}.\label{scheme:theta_bsde:y}
	\end{align}
\end{subequations}
In case $\vartheta_y=\vartheta_z=1$, this scheme is called the standard \emph{Euler scheme for BSDEs}.
\subsection{The OSM scheme}
The novelty of the hereby proposed discretization is that on top of solving \autoref{eq:fbsde_fbsde:bsde}, we also solve the linear BSDE in \autoref{eq:fbsde_fbsde:malliavin_bsde} driving the Malliavin derivatives of the solution pair. Exploiting the relation between $DY$ and $Z$ established by \autoref{thm:malliavin:bsde}, we set the control estimates according to the discrete time approximations of the Malliavin BSDE. As in the case of the forward component itself, the Malliavin derivative in \autoref{eq:fbsde_fbsde:malliavin_sde} is approximated by an Euler-Maruyama discretization, giving estimates
\begin{align}\label{scheme:sde:dx:euler}
	D_nX_{m}^\pi\coloneqq \begin{cases}
		\mathds{1}_{m=n}\sigma(t_n, X_n^\pi), &0\leq m\leq n\leq N,\\
		D_nX_{m-1}^\pi + \nabla_x\mu(t_{m-1}, X_{m-1}^\pi)D_nX_{m-1}^\pi\Delta t_{m-1} + \nabla_x \sigma(t_{m-1}, X_{m-1}^\pi)D_nX_{m-1}^\pi\Delta W_{m-1}, &0\leq n< m\leq N.
	\end{cases}
\end{align}
Unlike in the case of $X_n^\pi$, the convergence of these approximations is not straightforward due to the fact that the initial condition $D_nX_n^\pi=\sigma(t_n, X_n^\pi)$ already depends on the discrete approximation $X_n^\pi$ provided by \autoref{scheme:sde:euler}. Nonetheless, as we shall soon see, our discretization of the linear BSDE in \autoref{eq:fbsde_fbsde:malliavin_bsde} only relies on the approximations $D_nX_{n+1}^\pi$ for each $n=0,\dots,N-1$. This is a significant relaxation of the convergence criterion, as it can be shown that under relatively mild assumptions on the coefficients in \autoref{eq:fbsde_fbsde:sde}, $D_nX_{n+1}^\pi$ defined by \autoref{scheme:sde:dx:euler} inherits the convergence rate of \autoref{eq:euler:convergence_rate} -- see \autoref{appendix:discretization:dx} for details.

The discretization of the backward component in \autoref{eq:fbsde_fbsde:malliavin_bsde} is done as follows. For any $n=0, \dots, N-1$
\begin{align}
	D_{t_{n}}Y_{t_{n}} &= D_{t_{n}}Y_{t_{n+1}} + \int_{{t_{n}}}^{t_{n+1}} \fD(r, \X{r}, \DX{{t_{n}}}{r})\mathrm{d}r - \int_{{t_{n}}}^{{t_{n+1}}} D_{t_{n}}Z_r\mathrm{d}W_r,
\end{align}
subject to the terminal condition.
Multiplying this equation with $\Delta W_n$ from the left, It\^{o}'s isometry implies
\begin{align}
	\begin{split}
		\CondExpn{n}{\int_{t_{n}}^{t_{n+1}}D_{t_{n}}Z_r\mathrm{d}r} &= \left(\CondExpn{n}{\Delta W_n \left(D_{t_{n}}Y_{t_{n+1}} + \int_{t_{n}}^{t_{n+1}}\fD(r, \X{r}, \DX{{t_{n}}}{r})\mathrm{d}r\right)}\right)^T,\\
		D_{t_{n}}Y_{t_{n}} &= \CondExpn{n}{D_{t_{n}}Y_{t_{n+1}} + \int_{t_{n}}^{t_{n+1}}\fD(r, \X{r}, \DX{{t_{n}}}{r})\mathrm{d}r},
	\end{split}
\end{align}
where the transpose operation emerges from having defined $Z$ as a row vector and the Brownian motion as a column vector. In order to avoid implicitness on $Y$, we approximate the continuous time integrals with the left- and right rectangle rules, respectively, and obtain discrete time approximations
\begin{align}\label{eq:malliavin_bsde:discretization:step}
	D_nZ_n^\pi &= \frac{1}{\Delta t_n}\left(\CondExpn{n}{\Delta W_n \left(D_nY_{n+1}^\pi + \Delta t_n \fD(t_{n+1}, \X{n+1}^\pi, \DX{n}{n+1}^\pi)\right)}\right)^T,\\
	D_nY_n^\pi &=\CondExpn{n}{D_nY_{n+1}^\pi + \Delta t_n \fD(t_{n+1}, \X{n+1}^\pi, \DX{n}{n+1}^\pi)},
\end{align}
with $\X{n}^\pi\coloneqq (X_n^\pi, Y_n^\pi, Z_n^\pi)$ and $\DX{n}{n+1}^\pi\coloneqq (D_nX_{n+1}^\pi, D_nY_{n+1}^\pi, D_nZ_{n+1}^\pi)$.
At this point, to make the scheme viable, one relies on estimates $D_nY_{m}^\pi, D_nZ_{m}^\pi$ on top of the Euler-Maruyama approximations of $DX$ given by \autoref{scheme:sde:dx:euler}. This is done by a merged formulation of the Feynman-Kac formulae in \autoref{thm:feynman-kac} and the Malliavin chain rule in \autoref{lemma:malliavin_chain_rule}. Indeed, given the Markov property of the underlying processes, the Malliavin chain rule implies that
\begin{align}
	D_{t_{n}}Y_r=\nabla_x y(r, X_r)D_{t_{n}}X_r,\qquad D_{t_{n}}Z_r = \nabla_x z(r, X_r)D_{t_{n}}X_r\eqqcolon \gamma(r, X_r)D_{t_{n}}X_r,
\end{align}
for some deterministic functions $y:[0, T]\times \mathds{R}^{d\times 1}\to \mathds{R}$ and $z:[0, T]\times \mathds{R}^{d\times 1}\to\mathds{R}^{1\times d}$, where we defined $\gamma:[0, T]\times \mathds{R}^{d\times 1}\to\mathds{R}^{d\times d}$ as the Jacobian matrix of $z(r, X_r)$. Furthermore, due to the Feynman-Kac relations we also have $z(r, X_r)=\nabla_x y(r, X_r)\sigma(r, X_r)$ and therefore
\begin{align}
	D_{t_{n}}Y_r=z(r, X_r)\sigma^{-1}(r, X_r)D_{t_{n}}X_r,\qquad D_{t_{n}}Z_r = \gamma(r, X_r)D_{t_{n}}X_r.
\end{align}
Motivated by these relations, we approximate the discretized Malliavin derivatives in \autoref{eq:malliavin_bsde:discretization:step} according to
\begin{align}\label{eq:discrete_estimates:DY_DZ}
	D_{n}Y_{m}^\pi=Z_m^\pi\sigma^{-1}(t_{m}, X_{m}^\pi)D_{n}X_{m}^\pi,\qquad D_nZ_{m}^\pi = \Gamma_{m}^\pi D_nX_{m}^\pi,\qquad 0\leq, n,m\leq N.
\end{align}
Henceforth, the discrete approximations of the $Y$ process driven by \autoref{eq:fbsde_fbsde:bsde} are given in an identical fashion to \autoref{scheme:theta_bsde:y} with $\vartheta_y\in[0, 1]$ as a free parameter of the discretization. Moreover, in order to be able to control the $\mathds{L}^2$ projection error of $D_nZ_m^\pi$ with discrete Grönwall estimates -- see Step 1 of \autoref{thm:osm:discretization} in particular --, we make the $\nabla_z f$ part of $\fD$ implicit in $D_nZ_n^\pi$, and introduce the notation $\DX{n}{n+1, n}^\pi\coloneqq \left(D_nX_{n+1}^\pi, D_nY_{n+1}^\pi, D_nZ_{n}^\pi\right)$. Subject to the terminal conditions in \autoref{eq:fbsde_fbsde:bsde} and \autoref{eq:fbsde_fbsde:malliavin_bsde}, on top of the Malliavin chain rule estimates in \autoref{eq:discrete_estimates:DY_DZ}, this leads to the following discrete scheme, which we shall call the \emph{One Step Malliavin} (OSM) scheme
\begin{subequations}\label{scheme:osm}
	\begin{align}
		Y_N^\pi &= g(X_N^\pi),\quad Z_N^\pi =\nabla_x g(X_N^\pi) \sigma(t_N, X_N^\pi),\quad \Gamma_N^\pi=[\nabla_x (\nabla_x g\sigma)](t_N, X_N^\pi),\\
		\Gamma_n^\pi \sigma(t_n, X_n^\pi)=D_nZ_n^\pi &= \frac{1}{\Delta t_n}\left(\CondExpn{n}{\Delta W_n \left(D_nY_{n+1}^\pi  + \Delta t_n f^D(t_{n+1}, \X{n+1}^\pi, \DX{n}{n+1, n}^\pi)\right)}\right)^T,\label{scheme:osm:dz}\\
		Z_n^\pi &= \CondExpn{n}{D_nY_{n+1}^\pi + \Delta t_n f^D(t_{n+1}, \X{n+1}^\pi, \DX{n}{n+1, n}^\pi)},\label{scheme:osm:z}\\
		Y_n^\pi &= \vartheta_y \Delta t_n f(t_n, X_n^\pi, Y_n^\pi, Z_n^\pi) + \CondExpn{n}{Y_{n+1}^\pi + (1-\vartheta_y)\Delta t_n f(t_{n+1}, \X{n+1}^\pi)}.\label{scheme:osm:y}
	\end{align}
\end{subequations}
The scheme is made fully implementable by an appropriate parametrization to approximate the arising conditional expectations.

\begin{remark}[Comparison of discretizations]\label{remark:conditional_variance}
	There are two key differences between the standard Euler discretization in \autoref{scheme:theta_bsde} and the OSM scheme in \autoref{scheme:osm}. First, unlike in the former, the OSM scheme's solution is a triple of discrete random processes, including an additional layer of  $\Gamma$ estimates. Moreover, it can be seen that the estimate in \autoref{scheme:osm:z} exhibits a better conditional variance than that of \autoref{scheme:theta_bsde:z}. In case of the standard Euler discretization, the $Z$ process is approximated through It\^{o}'s isometry and the corresponding discrete time approximations include a $1/\Delta t_n$ factor -- second term in \autoref{scheme:theta_bsde:z} -- which leads to a quadratically exploding conditional variance of the resulting estimates. Several variance reduction techniques have been proposed to mitigate this problem -- we mention \cite{gobet_adaptive_2017, alanko_reducing_2013}.
	On the other hand, within the OSM scheme, the $Z$ process is approximated by the continuous solution of the Malliavin BSDE in \autoref{eq:fbsde_fbsde:malliavin_bsde} and therefore it carries the same conditional variance behavior as the $Y$ estimate. In case of a fully-implementable regression Monte Carlo setting, this explains why the OSM scheme may provide more accurate control approximations.
\end{remark}

\paragraph{Alternative formulations.} \autoref{scheme:osm} is not the first approach to the BSDE problem building on \autoref{thm:malliavin:bsde}. Turkedjiev in \cite{turkedjiev_two_2015} proposed a discrete time approximation scheme, where the $Z$ process is estimated by an integration by parts formula stemming from Malliavin calculus and discovered in \cite[Thm.3.1]{ma_representation_2002}. Hu et al. in \cite{hu_malliavin_2011} proposed an explicit scheme in the case of non Markovian BSDEs, where the control process is estimated using a representation formula implied by the linearity of the Malliavin BSDE \autoref{eq:fbsde_fbsde:malliavin_bsde} -- see \cite[Prop.5.5]{el_karoui_backward_1997}. Briand and Labart in \cite{briand_simulation_2014} offer a different approach to BSDEs, where building on chaos expansion formulas, the $Z$ process is taken as the Malliavin derivative of $Y$ given by \autoref{thm:malliavin:bsde}.
The difference between these formulations and \autoref{scheme:osm} is mostly twofold. The OSM scheme is concerned with solving the entire pair of FBSDE systems \autoref{eq:fbsde_fbsde} and not just the backward component in \autoref{eq:fbsde_fbsde:bsde}. This means that unlike in \cite{turkedjiev_two_2015, briand_simulation_2014, hu_malliavin_2011}, discrete time approximations give $\Gamma$ estimates as well.
Additionally, one important difference in the OSM scheme compared to the approaches \cite{turkedjiev_two_2015, hu_malliavin_2011} is that the conditional expectations in \autoref{scheme:osm} always project $\mathcal{F}_{t_{n+1}}$-measurable random variables on $\mathcal{F}_{t_{n}}$, whereas in the case of those works the arguments of the conditional expectations are $\mathcal{F}_T$-measurable. The most important implication of this difference is that -- unlike in \cite{turkedjiev_two_2015, hu_malliavin_2011} -- \autoref{scheme:osm} does not rely on discrete time estimates of the Malliavin derivatives $DX$ over the whole time window, only in between adjacent time steps $D_nX_{n+1}^\pi$. As shown in \autoref{appendix:discretization:dx}, under suitable regularity assumptions, $D_nX_{n+1}^\pi$ converges in the $\mathds{L}^2$-sense with a rate of $1/2$. 
However, similar statements cannot be made about all future time steps' $D_nX_{m}^\pi$ -- see also \cite[Remark 5.1]{hu_malliavin_2011}. This is a significant advantage in case one does not have analytical access to the trajectories of $\{D_{s}X_t\}_{0\leq s,t\leq T}$.

\section{Discretization error analysis}\label{sec:error_analysis}
Having introduced the discrete scheme simultaneously solving the FBSDE system itself and the FBSDE system of its solutions' Malliavin derivatives, we investigate the errors induced by the discretization of continuous processes in \autoref{scheme:osm}. It is known -- see \cite{bouchard_discrete-time_2004} -- that the $\mathds{L}^2$ discretization errors of the backward Euler scheme in \autoref{scheme:theta_bsde} admit to
\begin{align}\label{eq:euler_bsde:error_estimate}
	\max_{0\leq n\leq N} \Expectation{\abs{Y_{t_{n}} - Y_n^\pi}^2} + \Expectation{\sum_{n=0}^{N-1} \int_{t_{n}}^{t_{n+1}}\abs{Z_{r} - Z_n^\pi}^2\mathrm{d}r}\leq C\left(\Expectation{\abs{g(X_T) - g(X_n^\pi)}^2}  + \varepsilon^Z(\abs{\pi}) + \abs{\pi}\right),
\end{align} 
where $\varepsilon^Z(\abs{\pi})\coloneqq \Expectation{\sum_{n=0}^{N-1} \int_{t_{n}}^{t_{n+1}} \abs{Z_r - \bar{Z}_n^{n+1}}^2\mathrm{d}r}$ with $\bar{Z}_n^{n+1}\coloneqq 1/\Delta t_n \CondExpn{n}{\int_{t_{n}}^{t_{n+1}} Z_r\mathrm{d}r}$ according to \cite{zhang_numerical_2004}.
The purpose of the following section is to show a similar result for the proposed OSM scheme and prove that it is \emph{consistent} in the $\mathds{L}^2$-sense, i.e. the discrete time approximations errors converge to zero as the mesh size of the time partition $\abs{\pi}$ vanishes. In particular, we shall see that under standard Lipschitz assumptions on the driver $f$ of the BSDE \autoref{eq:fbsde_fbsde:bsde} and the driver $\fD$ of the linear Malliavin BSDE \autoref{eq:fbsde_fbsde:malliavin_bsde}, and additive noise in the forward diffusion, the convergence is of order $\mathcal{O}(\abs{\pi}^{1/2})$.
\begin{assumption}\label{ass:error_analysis}
	The following assumptions are in place.
	\begin{enumerate}[wide, labelwidth=0pt, labelindent=1em]
		\item[($\mathbf{A}^{\mu, \sigma}$)]SDE
		\begin{enumerate}[label={$(\mathbf{A}^{\mu, \sigma}_{\arabic*})$}, wide=\dimexpr\parindent+1em+\labelsep\relax, leftmargin=*]
			\item the forward equation has constant drift and diffusion coefficients (Arithmetic Brownian motion);\label{ass:error_analysis:sde:abm}
			\item the forward SDE has a uniformly elliptic diffusion coefficient, i.e. for any $\zeta\in\mathds{R}^d$ there exists a $\beta>0$ such that $\zeta^T\sigma\sigma^T\zeta>\beta\abs{\zeta}^2$\footnote{We remark that this condition is equivalent to $A=\sigma\sigma^T$ being a positive definite matrix.};\label{ass:error_analysis:sde:uniform_ellipticity}
		\end{enumerate} \label{ass:error_analysis:sde}
		\item[($\mathbf{A}^{f, g}$)] BSDE
		\begin{enumerate}[label={$(\mathbf{A}^{f, g}_{\arabic*})$}, wide=\dimexpr\parindent+1em+\labelsep\relax, leftmargin=*]
			\item $g\in C_b^{2+\alpha}(\mathds{R})$ with some $\alpha>0$, furthermore $g$ is also bounded;\label{ass:error_analysis:bsde:g}
			\item $f\in C_b^{0, 2, 2, 2}(\mathds{R})$;\label{ass:error_analysis:bsde:f}
			\item $f$ and its partial derivatives $\nabla_x f, \nabla_y f, \nabla_z f$ are all $1/2$-Hölder continuous in time.\label{ass:error_analysis:bsde:holder}
		\end{enumerate}\label{ass:error_analysis:bsde}
	\end{enumerate}
\end{assumption}
The conditions above are not minimal -- see also \autoref{subsection:error_analysis:assumptions_revisited}. Nevertheless, for the sake of the present analysis they are sufficient. In particular, since bounded continuous differentiability implies Lipschitz continuity due to the mean-value theorem, by \autoref{thm:malliavin:bsde} we have that under \autoref{ass:error_analysis} the FBSDE \autoref{eq:fbsde_fbsde:sde}--\autoref{eq:fbsde_fbsde:bsde} is Malliavin differentiable, and the Malliavin derivatives of its solutions satisfy the FBSDE \autoref{eq:fbsde_fbsde:malliavin_sde}--\autoref{eq:fbsde_fbsde:malliavin_bsde}. Additionally, due to \cite[Thm. 2.1]{delarue_forward_2006}, we can also exploit the following useful result from the theory of parabolic PDEs.
\begin{lemma}\label{lemma:pde_Cb_solution}
	Under \autoref{ass:error_analysis} the parabolic PDE in \autoref{eq:parabolic_pde} admits to a unique solution $u\in C_b^{1, 2}(\mathds{R})$.
\end{lemma}
Due to the Markov nature of the FBSDE system, the solutions of \autoref{eq:fbsde_fbsde:bsde} can be written as $Y_t=y(t, X_t), Z_t=z(t, X_t)$ for some deterministic functions $y:[0, T]\times \mathds{R}^{d\times 1}\to \mathds{R}$, $z:[0, T]\times \mathds{R}^{d\times 1}\to \mathds{R}^{1\times d}$. Furthermore, provided by \autoref{lemma:pde_Cb_solution}, one can use a merged formulation of the Malliavin chain rule lemma \autoref{lemma:malliavin_chain_rule} and the non-linear Feynman-Kac relations to get the following formulas for the solutions of \autoref{eq:fbsde_fbsde:malliavin_bsde}
\begin{align}\label{eq:malliavin_chain_rule_PLUS_feynman_kac}
	D_sY_t = \nabla_x y(t, X_t)D_sX_t=z(t, X_t)\sigma^{-1}(t, X_t)D_sX_t,\qquad D_sZ_t = \nabla_x z(t, X_t) D_sX_t\eqqcolon \gamma(t, X_t)D_sX_t,
\end{align}
where $\gamma:[0, T]\times \mathds{R}^{d\times 1}\to \mathds{R}^{d\times d}$ and similarly $\Gamma_t\coloneqq \gamma(t, X_t)$. We remark that in our setting $\sigma\in \mathds{R}^{d\times d}$, the existence of the inverse is guaranteed by the uniform ellipticity condition set on $\sigma$ in \autoref{ass:error_analysis}. In case the Brownian motion and the forward diffusion have different dimensions, similar statements can be made about right inverses -- see \cite{turkedjiev_two_2015}. Another important implication of the estimate above is that \autoref{ass:error_analysis}, through \autoref{lemma:pde_Cb_solution}, also implies that the driver of the Malliavin BSDE $\fD$ is Lipschitz continuous in its spatial arguments within the bounded domain. Indeed, the mean-value theorem for $f\in C_b^{0, 2, 2, 2}(\mathds{R})$ implies that $f$ and all its first-order derivatives in $(x, y, z)$ are  Lipschitz continuous, consequently for any uniformly bounded argument $(DX, DY, DZ)$ the following holds
\begin{align}\label{eq:lipschitz:f_fD}
	\begin{split}
		\abs{f(t_1, \mathbf{x}_1) - f(t_2, \mathbf{x}_2)}&\leq L_f\left(\abs{t_1-t_2}^{1/2} + \abs{x_1 - x_2} + \abs{y_1-y_2} + \abs{z_1-z_2}\right),\\
		\abs{\xi_1}, \abs{\eta_1}, \abs{\zeta_1}\leq L_{\fD}:\quad \abs{\fD(t_1, \mathbf{x}_1, \boldsymbol{\xi}_1) - \fD(t_2, \mathbf{x}_2, \boldsymbol{\xi}_2)}&\leq\begin{aligned}[t]
			L_{\fD}\Big(&\abs{t_1-t_2}^{1/2} + \abs{x_1 - x_2} + \abs{y_1-y_2} + \abs{z_1-z_2}\\
			&+\abs{\xi_1 - \xi_2} + \abs{\eta_1-\eta_2} + \abs{\zeta_1-\zeta_2}\Big),
		\end{aligned} 
	\end{split}
\end{align}
with $\mathbf{x}_i=\left(x_i, y_i, z_i\right)$, $\boldsymbol{\xi}_i\coloneqq \left(\xi_i, \eta_i, \zeta_i\right)$, $i=1, 2$;
for all $t_i\in[0, T]$, $x_i\in\mathds{R}^{d\times 1}$, $y_i\in\mathds{R}$, $z_i, \eta_i\in \mathds{R}^{1\times d}$ and $\xi_i, \zeta_i\in\mathds{R}^{d\times d}$, where $L_f, L_{\fD}>0$. Here we also used the assumption of Hölder continuity established by \ref{ass:error_analysis:bsde:holder}.

Given the usual time partition, it is clear that the discrete approximations \autoref{scheme:osm} are deterministic functions of $X_n^\pi$ and thereupon we put $Y_n^\pi \eqqcolon y^\pi(t_n, X_n^\pi)\eqqcolon y_n^\pi(X_n^\pi),\quad Z_n^\pi \eqqcolon z^\pi(t_n, X_n^\pi)\eqqcolon z_n^\pi(X_n^\pi),\quad \Gamma_n^\pi \eqqcolon \gamma^\pi(t_n, X_n^\pi)\eqqcolon \gamma_n^\pi(X_n^\pi)$. In light of \autoref{eq:discrete_estimates:DY_DZ}, we use the approximations
\begin{align}\label{eq:sec:error_analysis:malliavin_chain_rule}
	D_nY_{n+1}^\pi = Z_{n+1}^\pi\sigma^{-1}(t_{n+1}, X_{n+1}^\pi)D_nX_{n+1}^\pi,\qquad D_nZ_{n}^\pi = \Gamma_{n}^\pi D_nX_{n}^\pi.
\end{align}
We introduce the short-hand notations $\Delta X_n^\pi\coloneqq X_{t_{n}} - X_n^\pi$, $\Delta Y_n^\pi = Y_{t_{n}} - Y_n^\pi, \Delta Z_n^\pi=Z_{t_{n}} - Z_n^\pi$, $\Delta D_nX_{n+1}^\pi\coloneqq D_{t_{n}}X_{t_{n+1}} - D_nX_{n+1}^\pi$, $\Delta D_nY_{n+1}^\pi\coloneqq D_{t_{n}}Y_{t_{n+1}} - D_nY_{n+1}^\pi$ and $\Delta \Gamma_{n}^\pi\coloneqq \Gamma_{t_{n}} - \Gamma_{n}^\pi$.
Under the conditions of \autoref{ass:error_analysis}, provided by \autoref{thm:malliavin:sde} and \autoref{thm:malliavin:bsde}, we have that the processes $(X, Y, Z, DX, DY)$ are all mean-squared continuous in time, i.e. there exists a general constant $C$ such that for all $s, t, r\in [0, T]$
\begin{align}\label{eq:error_analysis:mean-squared_continuities}
	\begin{split}
		\Expectation{\abs{X_t-X_r}^2}&\leq C\abs{t-r},\quad \Expectation{\abs{Y_t-Y_r}^2}\leq C\abs{t-r},\quad \Expectation{\abs{Z_t-Z_r}^2}\leq C\abs{t-r},\\
		\Expectation{\abs{D_sY_t - D_sY_r}^2} &\leq C\abs{t-r},\quad \Expectation{\abs{D_sX_t - D_sX_r}^2}\leq C\abs{t-r}.
	\end{split}
\end{align}
Finally, we use
\begin{align}\label{def:dzprojection}
	\DZprojection{n}{n+1}\coloneqq \frac{1}{\Delta t_n}\CondExpn{n}{\int_{t_{n}}^{t_{n+1}} D_{t_{n}}Z_r\mathrm{d}r}
\end{align}
for the $\mathds{L}^2$-projection of the corresponding Malliavin derivative with respect to the $\mathcal{F}_{t_{n}}$ $\sigma$-algebra, with which we can define the $\mathds{L}^2(\mathds{R}^{d\times d})$-regularity of $DZ$ as follows
\begin{align}\label{def:dz:l2_regularity}
	\varepsilon^{DZ}(\abs{\pi})\coloneqq \sum_{n=0}^{N-1} \Expectation{\int_{t_{n}}^{t_{n+1}} \abs{D_{t_{n}}Z_r - \DZprojection{n}{n+1}}^2\mathrm{d}r}.
\end{align}
Under the condition of constant diffusion coefficients in \autoref{ass:error_analysis}, we have that $D_{t_{n}}Z_r=D_{t_{m}}Z_r=\Gamma_r\sigma$ for any $t_n, t_m<r$. Thereafter, exploiting the fact that due to \autoref{ass:error_analysis} the terminal condition of the Malliavin BSDE \autoref{eq:fbsde_fbsde:malliavin_bsde} is also Lipschitz continuous,  one can apply \cite[Thm.3.1]{zhang_numerical_2004} and get 
\begin{align}\label{eq:error_analyis:epsilon^DZ}
	\limsup_{\abs{\pi}\to 0} \frac{1}{\abs{\pi}} \varepsilon^{DZ}(\abs{\pi})<\infty.
\end{align}

\subsection{Discrete-time approximation error}
The main goal of this section is to give an upper bound for the discrete time approximation errors defined by
\begin{align}\label{eq:osm:disc:total_error}
	\mathcal{E}^{\pi}(\abs{\pi})\coloneqq \max_{0\leq n\leq N} \Expectation{\abs{\Delta Y_n^\pi}^2} + \max_{0\leq n\leq N} \Expectation{\abs{\Delta Z_n^\pi}^2} + \Expectation{\sum_{n=0}^{N-1} \int_{t_{n}}^{t_{n+1}}\abs{\left(\Gamma_r - \Gamma_n^\pi\right)\sigma}^2\mathrm{d}r}\leq C\abs{\pi}.
\end{align} 
This is established by the following theorem.
\begin{theorem}[Consistency of the OSM scheme]\label{thm:osm:discretization}
	Under \autoref{ass:error_analysis}, the scheme defined by \autoref{scheme:osm} for any $\vartheta_y\in[0, 1]$ has $\mathds{L}^2$-convergence of order $1/2$, i.e.
	\begin{align}\label{eq:error_analysis:convergence_rate}
		\limsup_{\abs{\pi}\to 0} \frac{1}{\abs{\pi}}\mathcal{E}^{\pi}(\abs{\pi})<\infty.
	\end{align}
\end{theorem}
\begin{proof}
	Throughout the proof $C$ denotes a constant independent of the time partition, whose value may vary from line to line. We proceed in steps and prove estimates for each component of the discretization error.
	\begin{steps}[wide, labelwidth=0pt, labelindent=0pt]
		\item \label{thm:osm:discretization:step:1}\textit{Estimate for $DZ$.} First, we establish an estimate for the corresponding discretization error of the $DZ$-component with respect to the $\mathds{L}^2$-projection $\DZprojection{n}{n+1}$.
		Let us fix $n=0, \dots, N-1$.
		Multiplying the Malliavin BSDE in \autoref{eq:fbsde_fbsde:malliavin_bsde} with $\Delta W_n$ and applying It\^{o}'s isometry, we find that the definition in \autoref{def:dzprojection} can be written as follows
		\begin{align}
			\Delta t_n\DZprojection{n}{n+1} &= \begin{aligned}[t]
				\left(\CondExpn{n}{\Delta W_n\Delta D_nY_{n+1}^\pi}\right)^T + \left(\CondExpn{n}{\Delta W_n\int_{t_{n}}^{t_{n+1}} \fD(r, \X{r}, \DX{t_{n}}{r})\mathrm{d}r}\right)^T,
			\end{aligned}
		\end{align}
		Combining this with the definition of the discrete scheme (\autoref{scheme:osm:dz}) gives
		\begin{align}
			\Delta t_n(\DZprojection{n}{n+1} - D_nZ_{n}^\pi) &= \begin{aligned}[t]
				&\left(\CondExpn{n}{\Delta W_n\left(\Delta D_nY_{n+1}^\pi -\CondExpn{n}{\Delta D_nY_{n+1}^\pi}\right)}\right)^T\\ 
				&+ \left(\CondExpn{n}{\Delta W_n\left(\int_{t_{n}}^{t_{n+1}} \fD(r, \X{r}, \DX{t_{n}}{r}) - f^D(t_{n+1}, \X{n+1}^\pi, \DX{n}{n+1, n}^\pi)\mathrm{d}r\right)}\right)^T,
			\end{aligned}
		\end{align}
		using the tower property of conditional expectations. In Frobenius norm, the conditional $\mathds{L}^2(\mathds{R}^d)$ Cauchy-Schwarz inequality subsequently implies
		\begin{align}
			\Delta t_n \abs{\DZprojection{n}{n+1} - D_nZ_{n}^\pi}
			&\leq \begin{aligned}[t]
				&(d\Delta t_n)^{1/2}\left(\CondExpn{n}{\abs{\Delta D_nY_{n+1}^\pi -\CondExpn{n}{\Delta D_nY_{n+1}^\pi}}^2}\right)^{1/2}\\
				&+(d\Delta t_n)^{1/2} \left(\CondExpn{n}{\abs{\int_{t_{n}}^{t_{n+1}} \fD(r, \X{r}, \DX{t_{n}}{r}) - f^D(t_{n+1}, \X{n+1}^\pi, \DX{n}{n+1, n}^\pi)\mathrm{d}r}^2}\right)^{1/2},
			\end{aligned}
		\end{align}
		by the independence of Brownian increments. Hence, due to the $L^2([0, T]; \mathds{R}^d)$ Cauchy-Schwarz inequality, we gather
		\begin{align}
			\Delta t_n \abs{\DZprojection{n}{n+1} - D_nZ_{n}^\pi}&\leq \begin{aligned}[t]
				&(d\Delta t_n)^{1/2}\left(\CondExpn{n}{\abs{\Delta D_nY_{n+1}^\pi -\CondExpn{n}{\Delta D_nY_{n+1}^\pi}}^2}\right)^{1/2}\\
				&+d^{1/2}\Delta t_n\left(\CondExpn{n}{\int_{t_{n}}^{t_{n+1}} \abs{\fD(r, \X{r}, \DX{t_{n}}{r}) - \fD(t_{n+1}, \X{n+1}^\pi, \DX{n}{n+1, n}^\pi)}^2\mathrm{d}r}\right)^{1/2}.
			\end{aligned}
		\end{align}
		Using the inequality $a, b\in\mathds{R}: (a+b)^2\leq 2(a^2+b^2)$ we collect the following $\mathds{L}^2(\mathds{R}^{d\times d})$ upper bound
		\begin{align}
			\Delta t_n \Expectation{\abs{\DZprojection{n}{n+1} - D_nZ_{n}^\pi}^2}&\leq \begin{aligned}[t]
				&2d\left(\Expectation{\abs{\Delta D_nY_{n+1}^\pi}^2} -\Expectation{\abs{\CondExpn{n}{\Delta D_nY_{n+1}^\pi}}^2}\right)\\
				&+2d\Delta t_n \Expectation{\int_{t_{n}}^{t_{n+1}} \abs{\fD(r, \X{r}, \DX{t_{n}}{r}) - \fD(t_{n+1}, \X{n+1}^\pi, \DX{n}{n+1	, n}^\pi)}^2\mathrm{d}r}.
			\end{aligned}
		\end{align}
		According to \autoref{eq:lipschitz:f_fD}, the uniform boundedness of $\DX{t_{n}}{r}$ implies that $\fD$ is Lipschitz continuous in all its spatial arguments and $1/2$-Hölder continuous in time, with a universal constant $L_{\fD}$. This, combined with the mean-squared continuities of the $X, Y, Z, D_{t_{n}}X$ and $D_{t_{n}}Y$ in \autoref{eq:error_analysis:mean-squared_continuities}, implies
		\begin{align}\label{eq:dzprojection:upperbound:step:minus2}
			\Delta t_n \Expectation{\abs{\DZprojection{n}{n+1} - D_nZ_{n}^\pi}^2}&\leq \begin{aligned}[t]
				&2d\left(\Expectation{\abs{\Delta D_nY_{n+1}^\pi}^2} -\Expectation{\abs{\CondExpn{n}{\Delta D_nY_{n+1}^\pi}}^2}\right)\\	&+16dL_{\fD}^2\Delta t_n\begin{aligned}[t]
					\Bigg\{&C\Delta t_n^2 + 2\Delta t_n\left(\Expectation{\abs{\Delta X_{n+1}^\pi}^2} + \Expectation{\abs{\Delta Y_{n+1}^\pi}^2} + \Expectation{\abs{\Delta Z_{n+1}^\pi}^2}\right)\\
					&+2\Delta t_n\left(\Expectation{\abs{\Delta D_nX_{n+1}^\pi}^2}+ \Expectation{\abs{\Delta D_nY_{n+1}^\pi}^2}\right)\\
					&+\Expectation{\int_{t_{n}}^{t_{n+1}} \abs{D_{t_{n}}Z_r - D_nZ_{n}^\pi}^2\mathrm{d}r}\Bigg\},
				\end{aligned}
			\end{aligned}
		\end{align}
		where we again used $(a+b)^2\leq 2(a^2+b^2)$ for $a,b\in\mathds{R}$.
		By the definition of $\DZprojection{n}{n+1}$ in \autoref{def:dzprojection}, the last term can be split as follows
		\begin{align}\label{eq:dzprojection:split}
			\Expectation{\int_{t_{n}}^{t_{n+1}}\abs{D_{t_{n}}Z_r - D_nZ_{n}^\pi}^2\mathrm{d}r}=\Expectation{\int_{t_{n}}^{t_{n+1}} \abs{D_{t_{n}}Z_r - \DZprojection{n}{n+1}}^2\mathrm{d}r} + \Delta t_n \Expectation{\abs{\DZprojection{n}{n+1}-D_nZ_{n}^\pi}^2}.
		\end{align}
		Plugging this back in \autoref{eq:dzprojection:upperbound:step:minus2} yields
		\begin{align}
			\Delta t_n \Expectation{\abs{\DZprojection{n}{n+1} - D_nZ_{n}^\pi}^2}&\leq \begin{aligned}[t]
				&2d\left(\Expectation{\abs{\Delta D_nY_{n+1}^\pi}^2} -\Expectation{\abs{\CondExpn{n}{\Delta D_nY_{n+1}^\pi}}^2}\right)\\
				&+16dL_{\fD}^2\Delta t_n\begin{aligned}[t]
					\Bigg\{&C\Delta t_n^2 + 2\Delta t_n\left(\Expectation{\abs{\Delta X_{n+1}^\pi}^2} + \Expectation{\abs{\Delta Y_{n+1}^\pi}^2} + \Expectation{\abs{\Delta Z_{n+1}^\pi}^2}\right)\\
					&+2\Delta t_n\left(\Expectation{\abs{\Delta D_nX_{n+1}^\pi}^2}+ \Expectation{\abs{\Delta D_nY_{n+1}^\pi}^2}\right)\\
					&+ \Expectation{\int_{t_{n}}^{t_{n+1}} \abs{D_{t_{n}}Z_r - \DZprojection{n}{n+1}}^2\mathrm{d}r} + \Delta t_n \Expectation{\abs{\DZprojection{n}{n+1}-D_nZ_{n}^\pi}^2}\Bigg\}.
				\end{aligned}
			\end{aligned}
		\end{align}
		For sufficiently small time steps satisfying $16dL_{\fD}^2\Delta t_n\leq 1/2$, we can therefore gather the estimate
		\begin{align}\label{error:disc:dzprojection}
			\Delta t_n \Expectation{\abs{\DZprojection{n}{n+1} - D_nZ_{n}^\pi}^2}\leq \begin{aligned}[t]
				&4d \left\{\Expectation{\abs{\Delta D_nY_{n+1}^\pi}^2} - \Expectation{\abs{\CondExpn{n}{\Delta D_nY_{n+1}^\pi}}^2}\right\}\\
				&+32dL_{\fD}^2\Delta t_n\begin{aligned}[t]
					\Bigg\{&C\Delta t_n^2 + 2\Delta t_n\left(\Expectation{\abs{\Delta X_{n+1}^\pi}^2} + \Expectation{\abs{\Delta Y_{n+1}^\pi}^2} + \Expectation{\abs{\Delta Z_{n+1}^\pi}^2}\right)\\
					&+2\Delta t_n\left(\Expectation{\abs{\Delta D_nX_{n+1}^\pi}^2}+ \Expectation{\abs{\Delta D_nY_{n+1}^\pi}^2}\right)\\
					&+ \Expectation{\int_{t_{n}}^{t_{n+1}} \abs{D_{t_{n}}Z_r - \DZprojection{n}{n+1}}^2\mathrm{d}r}\Bigg\}.
				\end{aligned}
			\end{aligned}
		\end{align}
		
		\item \textit{Estimate for $Z$}. With the above result in hand, we give an estimate for the control process. Under \autoref{ass:error_analysis}, provided by \autoref{thm:malliavin:bsde}, we identify the control process $Z$ by its continuous modification given by $DY$ and establish pointwise estimates. Indeed, from the dynamics of $D_{t_{n}}Y$ given by \autoref{eq:fbsde_fbsde:malliavin_bsde} and the definition of the discrete scheme in \autoref{scheme:osm:z}, it follows
		\begin{align}
			\Delta Z_n^\pi &= \begin{multlined}[t]
				\CondExpn{n}{\Delta D_nY_{n+1}^\pi} + \CondExpn{n}{\int_{t_{n}}^{t_{n+1}} \fD(r, \X{r}, \DX{t_{n}}{r}) - \fD(t_{n+1}, \X{n+1}^\pi, \DX{n}{n+1, n}^\pi)\mathrm{d}r}.
			\end{multlined}
		\end{align}
		Applying the Young-inequality of the form $(a+b)^2\leq (1+\rho\Delta t_n)a^2+(1+\frac{1}{\rho\Delta t_n})b^2$ with any $\rho>0$; using the Jensen- and $L^2([0, T]; \mathds{R}^d)$ Cauchy-Schwarz inequalities gives
		\begin{align}
			\Expectation{\abs{\Delta Z_{n}^\pi}^2}\leq \begin{aligned}[t]
				&(1+\rho\Delta t_n)\Expectation{\abs{\CondExpn{n}{\Delta D_nY_{n+1}^\pi}}^2}\\
				&+\frac{1}{\rho}(1+\rho\Delta t_n)\Expectation{\int_{t_{n}}^{t_{n+1}} \abs{\fD(r, \X{r}, \DX{t_{n}}{r}) - \fD(t_{n+1}, \X{n+1}^\pi, \DX{n}{n+1, n}^\pi)}^2\mathrm{d}r}.
			\end{aligned}
		\end{align}
		Exploiting the Lipschitz- and Hölder continuity of $\fD$ in \autoref{eq:lipschitz:f_fD} and using the mean-squared continuities of $X, Y, Z, D_{t_{n}}X$ and $D_{t_{n}}Y$ in \autoref{eq:error_analysis:mean-squared_continuities}, we subsequently gather
		\begin{align}\label{eq:error:z:disc:lipschitz}
			\Expectation{\abs{\Delta Z_{n}^\pi}^2}\leq \begin{multlined}[t]
				(1+\rho\Delta t_n)\Expectation{\abs{\CondExpn{n}{\Delta D_nY_{n+1}^\pi}}^2}\\
				+ \frac{8L_{\fD}^2}{\rho}(1+\rho\Delta t_n)\begin{aligned}[t]
					\Bigg\{&C\Delta t_n^2 + 2\Delta t_n \left(\Expectation{\abs{\Delta X_{n+1}^\pi}^2} + \Expectation{\abs{\Delta Y_{n+1}^\pi}^2} + \Expectation{\abs{\Delta Z_{n+1}^\pi}^2}\right)\\
					&+ 2\Delta t_n \left(\Expectation{\abs{\Delta D_nX_{n+1}^\pi}^2} + \Expectation{\abs{\Delta D_nY_{n+1}^\pi}^2}\right)\\
					&+\Expectation{\int_{t_{n}}^{t_{n+1}} \abs{D_{t_{n}}Z_r - D_nZ_{n}^\pi}^2\mathrm{d}r}\Bigg\}.
				\end{aligned}
			\end{multlined}
		\end{align}
		Splitting the last term according to \autoref{eq:dzprojection:split}, substituting the upper bound \autoref{error:disc:dzprojection} and choosing $\rho^*\coloneqq 32L_{\fD}^2$ then yields
		\begin{align}\label{eq:error_analysis:z:step_minus_2}
			\Expectation{\abs{\Delta Z_{n}^\pi}^2}\leq \begin{aligned}[t]
				&(1+\rho^*\Delta t_n)\Expectation{\abs{\CondExpn{n}{\Delta D_nY_{n+1}^\pi}}^2}\\
				&+ \frac{1+\rho^*\Delta t_n}{2}\begin{aligned}[t]
					\Bigg\{&C\Delta t_n^2+ (1 + 16dL_{\fD}^2)\Delta t_n \left(\Expectation{\abs{\Delta X_{n+1}^\pi}^2} + \Expectation{\abs{\Delta Y_{n+1}^\pi}^2} + \Expectation{\abs{\Delta Z_{n+1}^\pi}^2}\right)\\
					&+ (1 + 16dL_{\fD}^2) 	 	\Delta t_n \left(\Expectation{\abs{\Delta D_nX_{n+1}^\pi}^2} + \Expectation{\abs{\Delta D_nY_{n+1}^\pi}^2}\right)\\
					&+(1 + 16dL_{\fD}^2) \Expectation{\int_{t_{n}}^{t_{n+1}} \abs{D_{t_{n}}Z_r - \DZprojection{n}{n+1}}^2\mathrm{d}r}\Bigg\},
				\end{aligned}
			\end{aligned}
		\end{align}
		for any sufficiently small $\Delta t_n<1$.
		At this point, we can make use of the fact that due to \ref{ass:error_analysis:sde:abm} in \autoref{ass:error_analysis} $X_{n}^\pi= \sigma W_{t_{n}}= X_{t_{n}}$ and $D_nX_{n+1}^\pi= \sigma\equiv D_{t_{n}}X_{t_{n+1}}$, which in particular implies $X_{t_{n}} - X_n^\pi\equiv 0, D_{t_{n}}X_{t_{n+1}} - D_nX_{n+1}^\pi \equiv 0$ and
		\begin{align}\label{eq:error_analysis:z:analyticalX}
			\Delta D_nY_{n+1}^\pi &=\Delta Z_{n+1}^\pi,\qquad 	D_{t_{n}}Z_{t_{n}} - D_nZ_{n}^\pi = \Delta \Gamma_{n}^\pi\sigma,
		\end{align}
		in light of \autoref{eq:sec:error_analysis:malliavin_chain_rule}.
		Plugging these estimates back in \autoref{eq:error_analysis:z:step_minus_2} subsequently gives
		\begin{align}\label{eq:error_analysis:estimate:z}
			\Expectation{\abs{\Delta Z_{n}^\pi}^2}\leq \begin{aligned}[t]
				&(1+C_z\Delta t_n)\Expectation{\abs{\Delta Z_{n+1}^\pi}^2}\\
				&+ C_z\left\{\Delta t_n^2+ \Delta t_n\Expectation{\abs{\Delta Y_{n+1}^\pi}^2}+\Expectation{\int_{t_{n}}^{t_{n+1}} \abs{D_{t_{n}}Z_r - \DZprojection{n}{n+1}}^2\mathrm{d}r}\right\}.
			\end{aligned}
		\end{align}
		\item \textit{Estimate for $Y$}. Given $f$'s Lipschitz continuity in $(x, y, z)$ and $1/2$-Hölder continuity in $t$ by \autoref{eq:lipschitz:f_fD}, the mean-squared continuities of $X, Y$ and $Z$ in \autoref{eq:error_analysis:mean-squared_continuities}; through subsequent applications of the Young-, Jensen- and Cauchy-Schwarz inequalities analogously to the previous steps, we derive the following inequality from the dynamics of $Y$ in \autoref{eq:fbsde_fbsde:bsde} and the discrete scheme in \autoref{scheme:osm:y}
		\begin{align}\label{eq:error_analysis:estimate:y}
			\Expectation{\abs{\Delta Y_n^\pi}^2}\leq \begin{aligned}[t]
				&(1+\beta \Delta t_n)\Expectation{\abs{\Delta Y_{n+1}^\pi}^2}\\ 
				&+ \frac{8L_f^2}{\beta}(1+\beta \Delta t_n)\begin{aligned}[t]
					\Big\{&C\Delta t_n^2 + \vartheta_y^2\Delta t_n\left(\Expectation{\abs{\Delta Y_n^\pi}^2} + \Expectation{\abs{\Delta Z_n^\pi}^2}\right) \\ &+ (1-\vartheta_y)^2\Delta t_n\left(\Expectation{\abs{\Delta Y_{n+1}^\pi}^2} + \Expectation{\abs{\Delta Z_{n+1}^\pi}^2}\right)\Big\},
				\end{aligned}
			\end{aligned}
		\end{align}
		with any $\beta>0$.
		
		\item \textit{Combined estimate for $Y$ and $Z$}. Combining the estimates in \autoref{eq:error_analysis:estimate:z} and \autoref{eq:error_analysis:estimate:y} gives
		\begin{align}
			\left(1 - \frac{8L_f^2(1+\beta)\vartheta_y^2}{\beta}\Delta t_n\right)\left(\Expectation{\abs{\Delta Y_n^\pi}^2} + \Expectation{\abs{\Delta Z_n^\pi}^2}\right)\leq\begin{aligned}[t]
				&(1+C_{y}\Delta t_n)\left(\Expectation{\abs{\Delta Y_{n+1}^\pi}^2} + \Expectation{\abs{\Delta Z_{n+1}^\pi}^2}\right)\\
				&+C\left\{\Delta t_n^2+\Expectation{\int_{t_{n}}^{t_{n+1}} \abs{D_{t_{n}}Z_r - \DZprojection{n}{n+1}}^2\mathrm{d}r}\right\},
			\end{aligned}
		\end{align}
		with $C_y=\beta + \frac{8L_f^2(1+\beta)}{\beta}(1-\vartheta_y)^2+C_z$. Then, for any given $\beta>0$ and sufficiently small time step admitting to $\frac{8L_f^2(1+\beta)\vartheta_y^2}{\beta}\Delta t_n< 1$, we derive
		\begin{align}\label{eq:error_analysis:estimate:precombined}
			\Expectation{\abs{\Delta Y_n^\pi}^2} + \Expectation{\abs{\Delta Z_n^\pi}^2}\leq\begin{aligned}[t]
				&(1+C\Delta t_n)\left(\Expectation{\abs{\Delta Y_{n+1}^\pi}^2} + \Expectation{\abs{\Delta Z_{n+1}^\pi}^2}\right)\\
				&+C\left\{\Delta t_n^2+\Expectation{\int_{t_{n}}^{t_{n+1}} \abs{D_{t_{n}}Z_r - \DZprojection{n}{n+1}}^2\mathrm{d}r}\right\}.
			\end{aligned}
		\end{align}
		Thereupon, the discrete Grönwall lemma implies that
		\begin{align}\label{eq:error_analysis:estimate:y_and_z_combined}
			\max_{0\leq n\leq N}\Expectation{\abs{\Delta Y_n^\pi}^2} + \max_{0\leq n\leq N}\Expectation{\abs{\Delta Z_n^\pi}^2}\leq\begin{aligned}[t]
				C\Bigg\{&\Expectation{\abs{g(X_T) - g(X_{N}^\pi)}^2}\\
				&+ \Expectation{\abs{\nabla_x g(X_T)\sigma(t_N, X_T) - \nabla_x g(X_{t_{N}}^\pi)\sigma(t_N, X_N^\pi)}^2}\\
				& + \varepsilon^{DZ}(\abs{\pi}) + \abs{\pi}\Bigg\},
			\end{aligned} 
		\end{align}
		where we also used the definition in \autoref{def:dz:l2_regularity}.
		The proclaimed estimate for the $(Y, Z)$ part then follows from the observation that under \autoref{ass:error_analysis} the terminal conditions of both the BSDE in \autoref{eq:fbsde_fbsde:bsde} and the Malliavin BSDE in \autoref{eq:fbsde_fbsde:malliavin_bsde} are analytically observed; and the fact that, according to \autoref{eq:error_analyis:epsilon^DZ}, $\varepsilon^{DZ}(\abs{\pi})$ is also $\mathcal{O}(\abs{\pi})$.
		
		\item \textit{Final estimate for $\Gamma$}. 
		It remains to show the consistency of the $\Gamma$ estimates. From \autoref{error:disc:dzprojection} and \autoref{eq:dzprojection:split}, we get
		\begin{align}
			\Expectation{\int_{t_{n}}^{t_{n+1}} \abs{D_{t_{n}}Z_r - D_nZ_{n}^\pi}^2\mathrm{d}r}\leq  \begin{aligned}[t]
				&\Expectation{\int_{t_{n}}^{t_{n+1}} \abs{D_{t_{n}}Z_r - \DZprojection{n}{n+1}}^2\mathrm{d}r}\\
				&+ 4d \left\{\Expectation{\abs{\Delta D_nY_{n+1}^\pi}^2} - \Expectation{\abs{\CondExpn{n}{\Delta D_nY_{n+1}^\pi}}^2}\right\}\\
				&+32dL_{\fD}^2\Delta t_n\begin{aligned}[t]
					\Bigg\{&C\Delta t_n^2+ 2\Delta t_n \left(\Expectation{\abs{\Delta X_{n+1}^\pi}^2}+\Expectation{\abs{\Delta Y_{n+1}^\pi}^2} + \Expectation{\abs{\Delta Z_{n+1}^\pi}^2}\right)\\
					&+2\Delta t_n\left(\Expectation{\abs{\Delta D_nX_{n+1}^\pi}^2} + \Expectation{\abs{\Delta D_nY_{n+1}^\pi}^2}\right)\\
					&+\Expectation{\int_{t_{n}}^{t_{n+1}} \abs{D_{t_{n}}Z_r - \DZprojection{n}{n+1}}^2\mathrm{d}r}\Bigg\}.
				\end{aligned}
			\end{aligned}
		\end{align}
		Summation from $n=0, \dots, N-1$ thus gives
		\begin{align}\label{eq:error:gamma:disc:step:minus2}
			\Expectation{\sum_{n=0}^{N-1}\int_{t_{n}}^{t_{n+1}} \abs{D_{t_{n}}Z_r - D_nZ_{n}^\pi}^2\mathrm{d}r}\leq  \begin{aligned}[t]
				&\Expectation{\sum_{n=0}^{N-1}\int_{t_{n}}^{t_{n+1}} \abs{D_{t_{n}}Z_r - \DZprojection{n}{n+1}}^2\mathrm{d}r}
				+ 4d\Expectation{\abs{\Delta D_{N-1}Y_N^\pi}^2}\\ 
				&+ 4d\sum_{n=1}^{N-1}\left\{\Expectation{\abs{\Delta D_{n-1}Y_{n}^\pi}^2} - \Expectation{\abs{\CondExpn{n}{\Delta D_nY_{n+1}^\pi}}^2}\right\}\\
				&+32dL_{\fD}^2\sum_{n=0}^{N-1}\Delta t_n\begin{aligned}[t]
					\Bigg\{&C\Delta t_n^2 + 2\Delta t_n \left(\Expectation{\abs{\Delta X_{n+1}^\pi}^2} + \Expectation{\abs{\Delta Y_{n+1}^\pi}^2}\right)\\
					& + 2\Delta t_n\left(\Expectation{\abs{\Delta Z_{n+1}^\pi}^2} + \Expectation{\abs{\Delta D_nX_{n+1}^\pi}^2}\right)\\
					&+2\Delta t_n\left(\Expectation{\abs{\Delta D_nX_{n+1}^\pi}^2} + \Expectation{\abs{\Delta D_nY_{n+1}^\pi}^2}\right)\\
					&+\Expectation{\int_{t_{n}}^{t_{n+1}} \abs{D_{t_{n}}Z_r - \DZprojection{n}{n+1}}^2\mathrm{d}r}\Bigg\},
				\end{aligned}
			\end{aligned}
		\end{align}
		where we changed the summation index for the first part of the third term.
		Using the relations in \autoref{eq:error_analysis:z:analyticalX} implied by \autoref{ass:error_analysis}, we can upper bound the summation term by the upper bound \autoref{eq:error_analysis:z:step_minus_2}
		\begin{align}
			\Expectation{\abs{\Delta D_{n-1}Y_{n}^\pi}^2} - \Expectation{\abs{\CondExpn{n}{\Delta D_nY_{n+1}^\pi}}^2}\leq \begin{aligned}[t]
				&C\Delta t_n \Expectation{\abs{\Delta Z_{n+1}^\pi}^2}+ C\Delta t_n\Expectation{\abs{\Delta Y_{n+1}^\pi}^2}\\ & + C\Delta t_n^2 + C\Expectation{\int_{t_{n}}^{t_{n+1}} \abs{D_{t_{n}}Z_r - \DZprojection{n}{n+1}}^2}.
			\end{aligned}
		\end{align}
		Substituting this back into \autoref{eq:error:gamma:disc:step:minus2}, the convergence of the $\mathds{L}^2$-regularity of $DZ$ in \autoref{eq:error_analyis:epsilon^DZ}, and the estimate \autoref{eq:error_analysis:estimate:y_and_z_combined} proven in the previous step show the proclaimed convergence of the $\Gamma$ estimates.
		
		This concludes the proof.
	\end{steps}
\end{proof}

The final result in \autoref{eq:error_analysis:convergence_rate} expresses that the $\mathds{L}^2$ convergence rate of the discrete time approximations induced by \autoref{scheme:osm} is of order $\mathcal{O}(\abs{\pi}^{1/2})$ under the conditions imposed in \autoref{ass:error_analysis}. Comparing the convergence bound of \autoref{thm:osm:discretization} to that of the classical backward Euler discretization in \autoref{eq:euler_bsde:error_estimate}, three observations need to be made. First, in contrast to the backward Euler discretization, the OSM scheme admits to a bound where the $Z$ process is controlled by the maximum error over the discrete time steps -- see \autoref{eq:osm:disc:total_error}. This is due to the fact that under the OSM formulation, \autoref{thm:malliavin:bsde} guarantees a continuous version of the control process bounded in the supremum norm, and thus allows for pointwise estimates. 
Additionally, we see that even though the hereby proposed discretization solves a \emph{larger problem} by incorporating $\Gamma$ estimates, it exhibits the same, optimal rate of convergence well-known for the classical backward Euler discretization of BSDEs in \autoref{eq:euler_bsde:error_estimate}. At last, unlike in the aforementioned case, our final estimate does not include the strong discretization errors of the terminal conditions of the BSDEs \autoref{eq:fbsde_fbsde:bsde} and \autoref{eq:fbsde_fbsde:malliavin_bsde}. This is merely due to the fact that under \autoref{ass:error_analysis} we assumed constant diffusion coefficients,  which led to the corresponding terms canceling in \autoref{eq:error_analysis:estimate:y_and_z_combined}. Similarly, we exploited that under our conditions the Malliavin BSDE's terminal condition is Lipschitz continuous, leading to an $\mathcal{O}(\abs{\pi}^{1/2})$ convergence of the $\mathds{L}^2$-regularity of $DZ$ according to \autoref{eq:error_analyis:epsilon^DZ}. In case of irregular terminal conditions and non-analytical forward diffusions, it is expected that the corresponding terms would also contribute to the final estimate. 

\subsection{Assumptions revisited}\label{subsection:error_analysis:assumptions_revisited}

In order to conclude the discussion on the discrete time approximation errors, we elaborate on the conditions set in \autoref{ass:error_analysis}. Key aspects of their relevance are highlighted and potential ways to generalize the results are pointed out in order to encourage further research.

Not surprisingly, compared to classical discretizations excluding the Malliavin components, necessarily stricter conditions need to be posed in order to ensure Malliavin differentiability of the original FBSDE system in \autoref{eq:fbsde_fbsde:sde} -- \autoref{eq:fbsde_fbsde:bsde}.
The differentiability requirements on the coefficients $f$ and $g$ in \ref{ass:error_analysis:bsde:g}--\ref{ass:error_analysis:bsde:f} are inherently linked to the Malliavin differentiability of the FBSDE in \autoref{eq:fbsde_fbsde}. However, the Malliavin differentiability of the solution pair holds under significantly milder assumptions. We refer to \cite{mastrolia_malliavin_2017} for a recent account on the subject, where it is shown that first-order continuous differentiability, with not necessarily bounded $\nabla_x g, \nabla_x f$ is sufficient.

The reason why we nonetheless decided to restrict the assumptions to second-order bounded differentiability is mostly related to \autoref{lemma:pde_Cb_solution} and the Lipschitz continuity of $\fD$ in \autoref{eq:lipschitz:f_fD}. Although the Lipschitz continuity of $\nabla_x f, \nabla_y f, \nabla_z f$ are all guaranteed by the $C^{0, 2, 2, 2}_b$ assumption, the same cannot be said about the Malliavin derivative arguments $\DX{s}{t}$ of $\fD$. More precisely, in order to have Lipschitz continuity in all spatial arguments, one -- on top of the boundedness of the partial derivatives of $f$ -- also needs to have the uniform boundedness of all the Malliavin derivatives $(DX, DY, DZ)$. Due to the Malliavin chain rule estimates in \autoref{eq:malliavin_chain_rule_PLUS_feynman_kac}, under the assumption of constant diffusion coefficients in \ref{ass:error_analysis:sde:abm}, the uniform boundedness of the Malliavin derivatives is implied by the twice bounded differentiability of the solution of the parabolic problem in \autoref{eq:parabolic_pde}. This is guaranteed by \autoref{lemma:pde_Cb_solution}, requiring the conditions in \ref{ass:error_analysis:bsde:g}--\ref{ass:error_analysis:bsde:f} to be satisfied. 
In case the uniform boundedness of $(DY, DZ)$ is not readily available, one can truncate the corresponding arguments of $\fD$ similarly to \cite{chassagneux_numerical_2016}, and discretize the truncated Malliavin problem accordingly. Thereafter, the total discrete time approximation error can be decomposed into a truncation and discretization component, which guarantee convergence for an appropriately chosen, adaptive truncation range. A detailed presentation of this argument will be part of our future research.

Throughout the analysis, we also often relied on the assumption that the underlying forward diffusion admits to constant drift and diffusion coefficients due to \ref{ass:error_analysis:sde:abm}. In particular, this assumption allowed us to neglect the contribution of error terms such as $\Expectation{\abs{X_{t_{n}} - X_n^\pi}^2}$ and $\Expectation{\abs{D_{t_{n}}X_{t_{n+1}} - D_nX_{n+1}^\pi}^2}$ -- see, e.g., \autoref{eq:error_analysis:z:analyticalX}. However, it is well-known that the strong convergence of Euler-Maruyama approximations is of order $1/2$ -- see \autoref{eq:euler:convergence_rate} --, carrying the same order of convergence as the rest of the terms in our estimates. The convergence of the Malliavin derivative $D_nX_{n+1}^\pi$ with respect to an Euler-Maruyama discretization in \autoref{scheme:sde:dx:euler} is more troublesome. In fact, as highlighted by related works in the literature -- see \cite[Remark 5.1]{hu_malliavin_2011} --, it is difficult to guarantee the convergence of $D_sX^\pi$ over the whole time horizon. It is important to highlight that the OSM scheme in \autoref{scheme:osm} does not require approximations of the corresponding Malliavin derivative over the whole time window but only in between adjacent time steps $D_nX_{n+1}^\pi$. This is a major relieve in terms of convergence as one can easily show that within this one time stepping (OSM) scheme, $D_nX_{n+1}^\pi$ inherits the convergence properties of the forward diffusion under mild assumptions -- see \autoref{appendix:discretization:dx}.

The main difficulty with respect to general forward diffusions is related to the Malliavin chain rule approximations given by \autoref{eq:malliavin_chain_rule_PLUS_feynman_kac}. In fact, when $D_nX_{n+1}^\pi\neq D_{t_{n}}X_{t_{n+1}}$ one needs to deal with product terms such as
\begin{align}
	D_{t_{n}}Y_{t_{n+1}}-D_nY_{n+1}^\pi = \begin{aligned}[t]
		&\left[Z_{t_{n+1}}\sigma^{-1}(t_{n+1}, X_{t_{n+1}}) - Z_{n+1}^\pi\sigma^{-1}(t_{n+1}, X_{n+1}^\pi)\right] D_{t_{n}}X_{t_{n+1}}\\
		&+ Z_{n+1}^\pi\sigma^{-1}(t_{n+1}, X_{n+1}^\pi)\left[D_{t_{n}}X_{t_{n+1}}-D_nX_{n+1}^\pi\right].
	\end{aligned}
\end{align}
These pose a significant amount of difficulty when one -- unlike in the case of \ref{ass:error_analysis:sde:abm} -- does not have the uniform boundedness of $\sigma^{-1}$ and $\{D_sX_t\}_{0\leq s, t\leq T}$. Additionally, in order to ensure the boundedness of the discrete estimates $Z_{n+1}^\pi$, a certain truncation procedure would be required, further complicating the analysis. Therefore, we decided to restrict the assumptions to constant diffusion coefficients and to leave the general case for future research.
\begin{remark}[Non-constant drift and Girsanov's theorem]
	We remark that the assumption of a constant drift coefficient is mostly a matter convenience. Indeed, with a straightforward change of measure argument via the Girsanov theorem, one can merge the corresponding non-constant drift contribution onto the driver of the BSDE and -- as long as the drift itself satisfies the continuously bounded differentiable assumptions posed on $\nabla_x f$ -- the same analysis holds.
\end{remark}

\section{Fully implementable schemes with differentiable function approximators and neural networks}\label{sec:regression}
Having established a convergence result for the discrete time approximation's error induced by \autoref{scheme:osm}, we now turn to fully-implementable schemes where the appearing conditional expectations are numerically approximated by a certain machinery. In other words, we are concerned with the following modification of the discrete scheme in \autoref{scheme:osm}
\begin{subequations}\label{scheme:osm:implementable}
	\begin{align}
		\widehat{Y}_N^\pi &= g(X_N^\pi),\quad \widehat{Z}_N^\pi =\nabla_x g(X_N^\pi) \sigma(t_N, X_N^\pi),\quad \widehat{\Gamma}_N^\pi=[\nabla_x (\nabla_x g\sigma)](t_N, X_N^\pi),\\
		\widecheck{\Gamma}_n^\pi \sigma(t_n, X_n^\pi)=D_n\widecheck{Z}_n^\pi &= \frac{1}{\Delta t_n}\left(\CondExpn{n}{\Delta W_n \left(D_n\widehat{Y}_{n+1}^\pi  + \Delta t_n f^D(t_{n+1}, \Xhat{n+1}^\pi, \DXcheck{n}{n+1, n}^\pi)\right)}\right)^T,\quad 
		&&\widehat{\Gamma}_n^\pi\gets \mathcal{P}(\widecheck{\Gamma}_n^\pi),\label{scheme:osm:implementable:dz}\\
		\widecheck{Z}_n^\pi &= \CondExpn{n}{D_n\widehat{Y}_{n+1}^\pi + \Delta t_n f^D(t_{n+1}, \Xhat{n+1}^\pi, \DXhat{n}{n+1, n}^\pi)},\quad 
		&&\widehat{Z}_n^\pi\gets \mathcal{P}(\widecheck{Z}_n^\pi),\label{scheme:osm:implementable:z}\\
		\widecheck{Y}_n^\pi &= \vartheta_y \Delta t_n f(t_n, X_n^\pi, \widecheck{Y}_n^\pi, \widehat{Z}_n^\pi) + \CondExpn{n}{\widehat{Y}_{n+1}^\pi + (1-\vartheta_y)\Delta t_n f(t_{n+1}, \Xhat{n+1}^\pi)},\quad 
		&&\widehat{Y}_n^\pi\gets \mathcal{P}(\widecheck{Y}_n^\pi),\label{scheme:osm:implementable:y}
	\end{align}
\end{subequations}
with $\Xhat{n+1}^\pi\coloneqq \left(\widehat{X}_{n+1}^\pi, \widehat{Y}_{n+1}^\pi, \widehat{Z}_{n+1}^\pi\right)$, $\DXcheck{n}{n+1, n}^\pi\coloneqq (D_nX_{n+1}^\pi, D_n\widehat{Y}_{n+1}^\pi, D_n\widecheck{Z}_n^\pi)$ and $\DXhat{n}{n+1, n}^\pi\coloneqq \left(D_nX_{n+1}^\pi, D_n\widehat{Y}_{n+1}^\pi, D_n\widehat{Z}_{n}^\pi\right)$, where $D_n\widehat{Y}_{n+1}^\pi\coloneqq \widehat{Z}_{n+1}^\pi\sigma^{-1}(t_{n+1}, X_{n+1}^\pi)D_nX_{n+1}^\pi$ and $D_n\widehat{Z}_n^\pi\coloneqq \widehat{\Gamma}_n^\pi D_nX_{n+1}^\pi$ -- similarly as in \autoref{eq:discrete_estimates:DY_DZ}. The final approximations are denoted by $(\widehat{Y}_n^\pi, \widehat{Z}_n^\pi, \widehat{\Gamma}_n^\pi)$ and $\mathcal{P}$ denotes a \emph{machinery} which, given approximations at future time steps, estimates the \emph{true} conditional expectations $(\widecheck{Y}_n^\pi, \widecheck{Z}_n^\pi, \widecheck{\Gamma}_n^\pi)$. It is worth to notice that \autoref{scheme:osm:implementable:z} is explicit, whereas \autoref{scheme:osm:implementable:dz} and \autoref{scheme:osm:implementable:y} are both implicit when $\vartheta_y>0$. Due to the Markov feature of the corresponding problem, we can write all estimates as deterministic functions of the state process $\widecheck{Y}^\pi_n\eqqcolon \widecheck{y}^\pi_n(X_n^\pi)$, $\widecheck{Z}_n^\pi\eqqcolon \widecheck{z}_n^\pi(X_n^\pi)$, $\widecheck{\Gamma}_n^\pi\eqqcolon \widecheck{\gamma}^\pi_n(X_n^\pi)$ and $\widehat{Y}^\pi_n\eqqcolon \widehat{y}^\pi_n(X_n^\pi)$, $\widehat{Z}_n^\pi\eqqcolon \widehat{z}_n^\pi(X_n^\pi)$, $\widehat{\Gamma}_n^\pi\eqqcolon \widehat{\gamma}^\pi_n(X_n^\pi)$ at each time instance.

In the literature there exist several techniques to numerically approximate conditional expectations, see, e.g., \cite{bally_quantization_2003, briand_simulation_2014, bouchard_discrete-time_2004}.
In what follows, we investigate two specific approaches in the context of the OSM scheme. We first give an extension to the BCOS method \cite{ruijter_fourier_2015} which shall later be used as a benchmark method for one-dimensional problems. Our main approximation tool is based on a least-squares Monte Carlo formulation similar to those of the Deep BSDE methods \cite{han_solving_2018, hure_deep_2020}, where the functions parametrizing the solution triple are fully-connected, feedforward neural networks. Due to the universal approximation properties of neural networks in Sobolev spaces, this will allow us to distinguish between two variants. In the first one, the $\Gamma$ process is parametrized by a matrix-valued neural network whose parameters are optimized in a stochastic gradient descent iteration. In the second, this parametrization is circumvented and, in light of \autoref{thm:feynman-kac}, the $\Gamma$ estimates are directly calculated as the Jacobian of the $Z$ process. However, such directly linked estimates induce an additional source of error, which shall be addressed in \autoref{thm:osm:regression}, where we give an error bound for the complete approximation error of the fully-implementable OSM scheme, given the cumulative regression errors of neural network regressions, similarly to the ones proven in \cite{han_convergence_2020, hure_deep_2020}.

\subsection{The BCOS method}\label{sec:bcos}
We recall the most fundamental notions of the BCOS method \cite{ruijter_fourier_2015}. In order to keep the presentation concise, for the sake of this section we restrict ourselves to the one-dimensional case. BCOS is an extension of the COS method \cite{fang_novel_2009} to the setting of FBSDE systems, whose main idea is to recover the probability densities of certain random variables given that their characteristic function is available.
The key ideas of the BCOS method can be summarized as follows. In general, for a Markov problem, conditional expectations are of the form
\begin{align}
	I(x)\coloneqq \CondExp{v(t_{n+1}, X_{n+1}^{\pi})}{X_{n}^\pi=x}=\int_\mathds{R} v(t_{n+1}, \rho) p(\rho\vert x)\mathrm{d}x,
\end{align}
where $p(\rho\vert x)$ is the conditional transition density function from state $(t, x)$ to state $(t_{n+1}, \rho)$. Assuming that the integrand above decays in the infinite limit, one can truncate the integration range to a sufficiently wide finite domain $[a, b]$. Thereafter, the Fourier cosine expansion of the deterministic mapping $v(t_{n+1}, \cdot): [a, b]\to\mathds{R}$ reads as\footnote{We adhere to the standard notation where $\sideset{}{'}\sum_{k=0}^{K-1}a_k\coloneqq a_0/2+\sum_{k=1}^{K-1}a_k$, i.e. where the first element is multiplied by $1/2$.} 
\begin{align}
	v(t_{n+1}, \rho)=\sideset{}{'}\sum_{k=0}^{\infty} \mathcal{V}(t_{n+1})\cos(k\pi\frac{\rho-a}{b-a}),
\end{align}
where the series coefficients are given by $\mathcal{V}(t_{n+1})\coloneqq \frac{2}{b-a}\int_a^b v(t_{n+1}, \rho)\cos(k\pi\frac{\rho-a}{b-a})\mathrm{d}\rho$. Plugging these estimates back in the conditional expectation, with an additional truncation of the Fourier expansion to a finite number of $K$ coefficients, gives the approximation \cite{fang_novel_2009}
\begin{align}\label{eq:bcos:general_estimates:0th}
	I(x)\approx \widehat{I}(x)\coloneqq \sideset{}{'}\sum_{k=0}^{K-1} \mathcal{V}(t_{n+1})\Re{\Phi(k\vert x)},
\end{align}
where $\Phi(k\vert x)\coloneqq \phi(\frac{k\pi}{b-a}\vert x)e^{ik\pi\frac{x-a}{b-a}}$ and 
$\phi(u\vert x)$ is the conditional characteristic function of the Markov transition. In case the underlying Markov process is an  Euler-Maruyama approximation of the solution to a forward SDE, the conditional characteristic function is given by $\phi(u\vert x)= \exp(i u\mu(t_n, x)\Delta t_n - \frac{1}{2}u^2\sigma^2(t_n, x)\Delta t_n)$.
Using an integration by parts argument -- see \cite[Appendix A.1]{ruijter_fourier_2015} and \autoref{appendix:bcos} -- similar results can be constructed for conditional expectations of the forms
\begin{align}
	J(x)&\coloneqq \CondExpnx{n}{v(t_{n+1}, X_{n+1}^\pi)\Delta W_n}\approx \widehat{J}(x)\coloneqq \Delta t_n \sigma(t_n, x)\sideset{}{'}\sum_{k=0}^{K-1}-\frac{k\pi}{b-a} \mathcal{V}(t_{n+1})\Im{\Phi(k\vert x)},\label{eq:bcos:general_estimates:1th}\\
	K(x)&\coloneqq \CondExpnx{n}{v(t_{n+1}, X_{n+1}^\pi)(\Delta W_n)^2}\approx \widehat{K}(x)\coloneqq \begin{aligned}[t]
		&\Delta t_n\sideset{}{'}\sum_{k=0}^{K-1}\mathcal{V}(t_{n+1})\Re{\Phi(k\vert x)}\\ &- \Delta t_n^2 \sigma^2(t_n, x)\sum_{k=0}^{K-1}\left(\frac{k\pi}{b-a}\right)^2 \mathcal{V}(t_{n+1})\Re{\Phi(k\vert x)}.
	\end{aligned}\label{eq:bcos:general_estimates:2th}
\end{align}
Built on these approximations, the BCOS method goes as follows. One first needs to recover the coefficients of the terminal conditions either analytically or via Discrete Cosine Transforms (DCT). These coefficients are plugged into conditional expectations of the form \autoref{eq:bcos:general_estimates:0th}, \autoref{eq:bcos:general_estimates:1th} and \autoref{eq:bcos:general_estimates:2th}, providing estimates for the solutions at $t_{N-1}$. In order to make the scheme fully-implementable, one also relies on a machinery which recovers these coefficients while going to time step $n$, from time step $n+1$ in a backward recursive algorithm. This step can either be done by Fast Fourier Transforms (FFT) \cite{ruijter_fourier_2015} when the coefficients of the SDE are constant, or with DCT when they are not \cite{ruijter_numerical_2016}. When one is faced with an implicit conditional expectation ($\vartheta_y>0$) Picard iterations are performed, which -- under Lipschitz assumptions and sufficiently small time steps -- converge exponentially fast to the unique fixed point solution.

In particular, the BCOS approximations for \autoref{scheme:osm:implementable} read as follows -- for a more detailed derivation, see \autoref{appendix:bcos}
\begin{subequations}\label{bcos:osm}
	\begin{align}
		\widehat{y}_N^\pi(x)&=g(x),\quad \widehat{z}_N^\pi(x)=\partial_x g(x)\sigma(T, x),\quad \widehat{\gamma}_N^\pi(x)=\partial_x\left(\partial_x g\sigma\right)(T, x),\\
		\widehat{\gamma}^\pi_n(x)\sigma(t_n, x) &= \sideset{}{'}\sum_{k=0}^{K-1} \widehat{\mathcal{DZ}}_k(t_{n+1})\cos(k\pi\frac{x-a}{b-a}),\\
		\widehat{z}^\pi_n(x) &= \begin{aligned}[t]
			&\sigma(t_n, x)(1+\partial_x\mu(t_n, x)\Delta t_n)\sideset{}{'}\sum_{k=0}^{K-1}\widehat{\mathcal{W}}_{k}(t_{n+1})\Re{\Phi(k\vert x)}\\
			&-\sigma^2(t_n, x)\partial_x\sigma(t_n, x)\Delta t_n\sideset{}{'}\sum_{k=0}^{K-1} \frac{k\pi}{b-a}\widehat{\mathcal{W}}_k(t_{n+1})\Im{\Phi(k\vert x)}\\
			&+\Delta t_n \widehat{\gamma}^\pi_n(x)\sigma(t_n, x)\sideset{}{'}\sum_{k=0}^{K-1} \widehat{\mathcal{F}}^z_k(t_{n+1})\Re{\Phi(k\vert x)},
		\end{aligned}\\
		\widehat{y}_n^\pi(x) &= \sideset{}{'}\sum_{k=0}^{K-1} \widehat{\mathcal{Y}}_k(t_n)\cos\left(k\pi\frac{x-a}{b-a}\right),
	\end{align}
\end{subequations}
where we defined
\begin{align}\label{eq:w_def}
	\begin{split}
		h_{n+1}^\pi(X_{n+1}^\pi)&\coloneqq \widehat{y}_{n+1}^\pi(X_{n+1}^\pi) + (1-\vartheta_y)\Delta t_nf(t_{n+1}, X_{n+1}^\pi,  \widehat{y}_{n+1}^\pi(X_{n+1}^\pi), \widehat{z}_{n+1}^\pi(X_{n+1}^\pi)),\\
		w_{n+1}^\pi(X_{n+1}^\pi)&\coloneqq 
		\left(1 + \partial_yf(t_{n+1}, \Xhat{n+1}^\pi)\right)\widehat{z}_{n+1}^\pi(X_{n+1}^\pi)\sigma^{-1}(t_{n+1}, X_{n+1}^\pi) + \Delta t_n \partial_x f(t_n, \Xhat{n+1})
	\end{split}
\end{align}
for the explicit parts of the discrete approximations \autoref{scheme:osm:implementable:y} and \autoref{scheme:osm:implementable:z}, respectively. The coefficients 
\begin{align}
	\begin{split}
		\mathcal{W}_k(t_{n+1})&\coloneqq \frac{2}{b-a}\int_a^b w_{n+1}^\pi(\rho)\cos(k\pi\frac{\rho-a}{b-a})\mathrm{d}\rho,\quad \mathcal{H}_k(t_{n+1})\coloneqq \frac{2}{b-a}\int_a^b h_{n+1}^\pi(\rho)\cos(k\pi\frac{\rho-a}{b-a})\mathrm{d}\rho,\quad\\ \mathcal{F}^z_k(t_{n+1})&\coloneqq \frac{2}{b-a}\int_a^b \partial_z f(t_{n+1}, \rho)\cos(k\pi\frac{\rho-a}{b-a})\mathrm{d}\rho
	\end{split}
\end{align}
are approximated by their DCT counterparts $\widehat{\mathcal{W}}_k(t_{n+1})$, $\widehat{\mathcal{H}}_k(t_{n+1})$ and $\widehat{\mathcal{F}}^z_k(t_{n+1})$, respectively.
$\widehat{\mathcal{DZ}}_k(t_{n+1})$ is recovered with DCT on the approximations $\CondExpnx{n}{\Delta t_n^{-1}\Delta W_n w_{n+1}^\pi(X_{n+1}^\pi) D_nX_{n+1}^\pi} / \left(1 -\CondExpnx{n}{\Delta W_n\partial_z f(t_{n+1}, \Xhat{n+1}^\pi)}\right)$.
Thereafter, the BCOS formulas in \autoref{eq:bcos:general_estimates:0th}, \autoref{eq:bcos:general_estimates:1th} and \autoref{eq:bcos:general_estimates:2th}, together with the Euler-Maruyama estimates \autoref{scheme:sde:dx:euler}, imply the estimates for $\Gamma$ and $Z$. The $Z$ estimates are plugged into the approximation of the $Y$ process in \autoref{scheme:osm:implementable:y}. The coefficients $\widehat{\mathcal{Y}}_k(t_{n})$ are recovered from the estimates $y^{P, \pi}_{n}(x)=\vartheta_y\Delta t_n f(t_n, x, y_n^{P-1, \pi}(x), \widehat{z}^\pi_n(x))+ \CondExpnx{n}{h_{n+1}^\pi}$ after a sufficient number of $P$ Picard iterations are taken. This completes the BCOS algorithm for the OSM scheme.

For a detailed account on the contributions of the corresponding truncation and approximation errors of the BCOS method we refer to \cite{ruijter_fourier_2015, ruijter_numerical_2016, fang_novel_2009} and the references therein.
Although the method can be extended to higher-dimensional diffusion processes, it suffers from the curse of dimensionality through the inevitable spatial discretization required in the Fourier frequency domain.
\subsection{Neural networks}

In recent years, neural networks have shown excellent empirical results when deployed in a regression Monte Carlo framework for BSDEs \cite{han_solving_2018, hure_deep_2020, fujii_asymptotic_2019}.
In what follows, we are concerned with the class of feedforward, fully-connected deep neural networks, particularly in the context of approximating high-dimensional conditional expectations. This family of functions $\Psi(\cdot\vert\Theta):\mathds{R}^{d\times 1}\to \mathds{R}^{q\times d}$ can be described as a hierarchical sequence of compositions
\begin{align}
	\Psi(x\vert \Theta)\coloneqq a_{\text{out}}\circ A_{L+1}(\cdot\vert\theta_{L+1})\circ a\circ A_{L}(\cdot\vert\theta_{L})\circ a\circ \dots\circ a\circ A_1(\cdot\vert\theta_1)\circ x.
\end{align}
The affine transformations $A_l, l=1, \dots, L$ are called \emph{hidden layers} and are of the form $A_l(y\vert\theta^l\coloneqq (W_{l-1}^l, b_l))\coloneqq W_{l-1}^l y + b_l$, with $W_{l-1}^l\in\mathds{R}^{S_l\times S_{l-1}}$ being a matrix of \emph{weights} and $b_l\in\mathds{R}^{S_{l}}$, $S_{l-1}, S_l\in \mathds{N}$ the \emph{biases}. Furthermore, $a: \mathds{R}\to\mathds{R}$ describes a non-linear \emph{activation} function, which is applied element-wise on the output of each affine transformation. The size $S_l$ denotes how many \emph{neurons} are contained in the given layer. The \emph{output layer} is defined by $A_{L+1}(y\vert\theta_{L+1}\coloneqq (W_{L}^{L+1}, b_{L+1}))\coloneqq W_{L}^{L+1}y + b_{L+1}$ with $W_{L}^{L+1}\in\mathds{R}^{q\times d\times S_{L}}, b_{L+1}\in\mathds{R}^{q\times d}$. The complete parameter space of such an architecture is therefore given by $\Theta\coloneqq \left(\theta_1, \dots, \theta_{L+1}\right)\in \mathds{R}^{q\times d\times (S_L + 1) + \sum_{l=1}^{L} S_{l-1}\times S_{l} + S_l}$. Widely common choices for the non-linearity include: Rectified Linear Units (ReLU), sigmoid and the hyperbolic tangent activations. The optimal parameter space $\Theta^*$ is usually approximated by first formulating a \emph{loss function} which measures an abstract distance from the desired behavior, and then iteratively minimizing this loss through a \emph{stochastic gradient descent} (SGD) type algorithm. For more details, we refer to \cite{goodfellow_deep_2016}.

The use of deep learning is often motivated by the so-called \emph{Universal Approximation Theorems (UAT)} which establish that neural networks can approximate a wide class of functions with arbitrary accuracy. The first version of the UAT property was proven by Cybenko in \cite{cybenko_approximation_1989}. However, as in the applications of this paper derivative approximations play an important role, we present the following extension of Hornik et al. \cite{hornik_approximation_1991}, which extends the UAT property to \emph{Sobolev spaces}.
In what follows, we use the common notations for $W^{k, p}(U)\coloneqq \{f\in L^p(U): \norm{f}_{W^{k, p}}\coloneqq \sum_{\abs{\alpha}\leq k} \int_U \abs{D^\alpha f}^p\mathrm{d}\lambda <\infty\}$ for Sobolev spaces, where $\alpha$ denotes a multi-index, $D^\alpha$ is the differentiation operator in the weak sense and $\lambda$ is the Lebesgue measure. In particular, we use $H^k(U)\coloneqq W^{k, 2}(U)$. Then the UAT in Sobolev spaces can be stated as follows -- for a proof see \cite[Corollary 6]{hornik_approximation_1991}.
\begin{theorem}[Universal Approximation Theorem in Sobolev Spaces, \cite{hornik_approximation_1991}]\label{thm:hornik:uat:sobolev}
	Let $a:\mathds{R}\to\mathds{R}$ be an $\ell$-finite activation function, i.e. $a\in C^1(\mathds{R})$ and $\int_\mathds{R} \abs{D^{\ell} a}<\infty$. Let $U\subseteq \mathds{R}^d$ be a compact subset. Denote the class of single hidden layer neural networks by $\Sigma(a)\coloneqq \{\psi: \mathds{R}^d\to\mathds{R}^q: \psi(x\vert \Theta=(W_0^1, b_1, W_1^2, b_2))=W_1^2a(W_0^1 x + b_1)+b_2, W_0^1\in \mathds{R}^{d\times S_{1}}, b_1\in \mathds{R}^{S_{1}}, W_1^2\in\mathds{R}^{S_1\times q}, b_2\in \mathds{R}^q, S_1\in \mathds{N}\}$. Then $\Sigma(a)$ is dense in $W^{m, p}(U)$ for each $0\leq m\leq \ell$, i.e. for any $\epsilon>0$ and $f\in W^{m, p}$ there exists a $\psi\in \Sigma(a)$ such that $\norm{\psi - f}_{W^{m, p}}<\epsilon$.
	
	In particular, we have that for any $\ell=1$-finite activation $a$, $f\in H^k(U)$ and $\epsilon>0$ there exists a $\psi\in \Sigma(a)$ such that
	\begin{align}
		\int_U \abs{\psi - f}^2\mathrm{d}\lambda + \int_U \abs{\nabla_x \psi - D f}^2\mathrm{d}\lambda < \epsilon.
	\end{align}
\end{theorem}
The main implication of the UAT property is that given a compact domain on $\mathds{R}^d$ and an appropriate activation function, one can approximate any Sobolev function by shallow neural networks\footnote{It is clear that the above statement generalizes to deep neural networks containing multiple hidden layers.} with arbitrary accuracy. It is worth to highlight that in the context of a regression Monte Carlo application, this does not establish an implementable \emph{regression bias} due to the lack of bounds on the width of the hidden layer.
We remark that the above version is not a state of the art result and refer to \cite{pinkus_approximation_1999} for a classical survey on the subject.

\paragraph{Layer Normalization.} 

Normalization is a standard tool to enhance the convergence of stochastic gradient descent like algorithms \cite{goodfellow_deep_2016}. In standard examples \cite{han_solving_2018} this is usually done by a so-called \emph{batch normalization} technique. However, as we shall see, in our setting batch normalization is computationally intensive as it ruins batch independence and implies quadratic dependence of the Jacobian tensor on the chosen batch size.
Hence, we instead deploy \emph{layer normalization} \cite{ba_layer_2016} where normalization takes place across the output activations of a given hidden layer.
Therefore, the final network architecture considered in \autoref{sec:numerical_experiments} is described by the sequence of compositions
\begin{align}\label{architecture:nn}
	\Psi(x\vert \bar{\Theta})\coloneqq a^{\text{out}}\circ A^{L+1}(\cdot\vert\theta^{L+1})\circ a\circ A^{L}(\cdot\vert\theta^{L})\circ \bar{n}\circ a\circ \dots\circ \bar{n}\circ A^1(\cdot\vert\theta^1)\circ x,
\end{align}
with $\bar{n}(\cdot\vert \beta_l)$ and $\bar{\Theta}\coloneqq (\Theta, \beta_1, \dots, \beta_{L-1})$, where $\beta_l$ denotes the $l$'th normalization layer's parameters -- see \cite{ba_layer_2016}.

\subsection{A Deep BSDE approach}
In what follows, we formulate a Deep BSDE approach similar to \cite{hure_deep_2020}, which scales well in high-dimensional settings and tackles the fully-implementable scheme \autoref{scheme:osm:implementable} in a neural network least-squares Monte Carlo framework. The main difference between our approach and that of \cite{hure_deep_2020} is that, unlike in the discretization problem \autoref{scheme:theta_bsde}, we solve the $d$-dimensional linear BSDE of the Malliavin derivatives in \autoref{eq:fbsde_fbsde:malliavin_bsde} -- on top of the scalar BSDE \autoref{eq:fbsde_fbsde:bsde}. We separate the solutions of these two BSDEs and perform two distinct neural network regressions at each time step. We distinguish between two approaches. The first involves an additional layer of parametrization in which the matrix-valued $\Gamma$ process is approximated by an $\mathds{R}^{d\times d}$-valued neural network. In the second, we take advantage of neural networks being dense function approximators in Sobolev spaces provided by \autoref{thm:hornik:uat:sobolev}, circumvent parametrizing the $\Gamma$ process and instead obtain it as the direct derivative of the $Z$ process via automatic differentiation -- in a way very similar to the second scheme (DBDP2) of \cite{hure_deep_2020}. In doing so, we require a so-called Jacobian training where the loss is dependent of the derivative of the neural network involved.

In order to motivate the merged problem formulation, notice that by \autoref{ass:error_analysis} on the coefficients of the BSDE, the arguments of the conditional expectations in \autoref{scheme:osm:implementable} are all $\mathds{L}^2$-integrable random variables. Consequently, \autoref{scheme:osm:implementable:z}, combined with the martingale representation theorem, implies the existence of a unique random process $D_n\widetilde{Z}_r$ such that
\begin{align}
	D_n\widehat{Y}_{n+1}^\pi + \Delta t_n f^D(t_{n+1}, \Xhat{n+1}^\pi, \DXcheck{n}{n+1, n}^\pi) &= \widecheck{Z}_n^\pi + \int_{t_{n}}^{t_{n+1}} D_n\widetilde{Z}_r\mathrm{d}W_r.\label{def:dz:martingale_representation}
\end{align}
It\^{o}'s isometry implies that the $\mathds{L}^2$-projection of $D_n\widetilde{Z}_r$ coincides with $D_n\widecheck{Z}_n^\pi$ in \autoref{scheme:osm}
\begin{align}\label{eq:dz:martingale_representation=dz}
	D_n\widecheck{Z}_n^\pi= \frac{1}{\Delta t_n}\CondExpn{n}{\int_{t_{n}}^{t_{n+1}} D_n\widetilde{Z}_r\mathrm{d}r}.
\end{align}
Thereupon, $\widecheck{Z}_n^\pi + D_n\widecheck{Z}_n^\pi\Delta W_n$ is not just the best $\mathds{L}^2$-projection of the left-hand side of \autoref{def:dz:martingale_representation} but also of the arguments of the conditional expectations on the right-hand side of \autoref{scheme:osm:implementable:dz}. Hence, it simultaneously solves the discretization problems \autoref{scheme:osm:implementable:dz} and \autoref{scheme:osm:implementable:z}.

Motivated by these observations the Deep BSDE approach then goes as follows -- the complete algorithm is collected in \Cref{alg:osm}. We set $\widehat{Y}_N^\pi=g(X_N^\pi)$, $\widehat{Z}_N^\pi=\nabla_x g(X_N^\pi)\sigma(T, X_N^\pi)$ and $\widehat{\Gamma}_N^\pi=\nabla_x (\nabla_x g \sigma)(T, X_N^\pi)$. Thereafter, each time step's $Y$, $Z$ and $\Gamma$ is parametrized by three independent fully-connected feedforward neural networks $\varphi(\cdot\vert \theta^y):\mathds{R}^{d\times 1}\to\mathds{R}$, $\psi(\cdot\vert\theta^z):\mathds{R}^{d\times 1}\to\mathds{R}^{1\times d}$ and $\chi(\cdot\vert\theta^\gamma):\mathds{R}^{d\times 1}\to\mathds{R}^{d\times d}$ of the type \autoref{architecture:nn}. The parameter sets $(\theta^z, \theta^\gamma)$ and $\theta^y$ are trained in two separate regressions. First, in light of \autoref{eq:dz:martingale_representation=dz}, we define the loss function of the regression problem corresponding to \autoref{scheme:osm:implementable:dz}--\autoref{scheme:osm:implementable:z} by 
\begin{align}\label{loss:z:parametrized}
	\mathcal{L}_n^{z, \gamma}(\theta^z, \theta^\gamma)\coloneqq \begin{aligned}[t]
		\mathds{E}\Big[\Big|&(1 + \Delta t_n \nabla_y f(t_{n+1}, \Xhat{n+1}^\pi))D_n\widehat{Y}_{n+1}^\pi+ \Delta t_n \nabla_x f(t_{n+1}, \Xhat{n+1}^\pi)D_nX_{n+1}^\pi \\ &- \psi(X_n^\pi\vert\theta^z) + \Delta t_n\nabla_z f(t_{n+1}, \Xhat{n+1}^\pi) \chi(X_n^\pi\vert\theta^\gamma)\sigma(t_n, X_n^\pi) - \chi(X_n^\pi\vert\theta^\gamma)\sigma(t_n, X_n^\pi)\Delta W_n\Big|^2\Big],
	\end{aligned}
\end{align}
where we approximate $D_nZ_n^\pi$ by $\chi(X_n^\pi\vert\theta^\gamma)D_nX_n^\pi$, according to the Malliavin chain rule.
We gather an approximation of the minimal parameter set $(\theta_n^{z, *}, \theta_n^{\gamma, *})\in \argmin_{(\theta^z, \theta^\gamma)}\mathcal{L}_n^z(\theta^z, \theta^\gamma)$ after minimizing an empirically observed version of the loss function through a stochastic gradient descent optimization, resulting in approximations $\widehat{\theta}^z_n$ and $\widehat{\theta}_n^\gamma$ -- see \Cref{alg:osm}. The final approximations are given by $\widehat{Z}_n^\pi\coloneqq \psi(X_n^\pi\vert \widehat{\theta}_n^z)$ and $\widehat{\Gamma}_n^\pi\coloneqq \chi(X_n^\pi\vert \widehat{\theta}_n^\gamma)$.

Similarly to the second scheme in \cite{hure_deep_2020}, an alternative formulation can be given which avoids parametrizing the $\Gamma$ process, and instead approximates it as the direct derivative of the $Z$ process provided by the Malliavin chain rule lemma \autoref{lemma:malliavin_chain_rule}. Eventually, this implies the direct connection $\chi(X_n^\pi\vert\theta^\gamma)\equiv \nabla_x \psi(X_n^\pi\vert\theta^z)$, with which the corresponding loss function becomes
\begin{align}\label{loss:z:ad}
	\mathcal{L}_n^{z, \nabla z}(\theta^z)\coloneqq \begin{aligned}[t]
		\mathds{E}\Big[\Big|&(1 + \Delta t_n \nabla_y f(t_{n+1}, \Xhat{n+1}^\pi))D_n\widehat{Y}_{n+1}^\pi+ \Delta t_n \nabla_x f(t_{n+1}, \Xhat{n+1}^\pi)D_nX_{n+1}^\pi \\ &- \psi(X_n^\pi\vert\theta^z) + \Delta t_n\nabla_z f(t_{n+1}, \Xhat{n+1}^\pi) \nabla_x \psi(X_n^\pi\vert\theta^z)\sigma(t_n, X_n^\pi) - \nabla_x \psi(X_n^\pi\vert\theta^z)\sigma(t_n, X_n^\pi)\Delta W_n\Big|^2\Big],
	\end{aligned}
\end{align}
where we exploited the relation between the $\Gamma$ and $Z$ processes, provided by the Malliavin chain rule, and set $D_n\widehat{Z}_n^\pi=\nabla_x \widehat{z}_n^\pi(X_n^\pi)D_nX_n^\pi$. The SGD approximation of the optimal parameter space $\theta^{z, *}_n\in\argmin_{\theta^z}\mathcal{L}_n^{z,\nabla z}(\theta^z)$ is denoted by $\widehat{\theta}_n^z$, and the final approximations are of the form  $\widehat{Z}_n^\pi\coloneqq \psi(X_n^\pi\vert \widehat{\theta}_n^z)$ and $\widehat{\Gamma}_n^\pi\coloneqq \nabla_x \psi(X_n^\pi\vert \widehat{\theta}_n^z)$.

Subsequently, these approximations are plugged into the regression problem of \autoref{scheme:osm:implementable:y}. This step, apart from the additional theta-discretization, is identical to that of \cite{hure_deep_2020} and the loss function reads as
\begin{align}\label{loss:y}
	\mathcal{L}_n^y(\theta^y)\coloneqq \Expectation{\abs{\widehat{Y}_{n+1}^\pi + (1-\vartheta_y)\Delta t_n f(t_{n+1}, \Xhat{n+1}^\pi) -\varphi(X_n^\pi\vert\theta^y) + \vartheta_y\Delta t_n f(t_n, X_n^\pi, \varphi(X_n^\pi\vert\theta^y), \widehat{Z}_n^\pi) - \widehat{Z}_n^\pi\Delta W_n}^2}.
\end{align}
The stochastic gradient descent approximation of the optimal parameter space $\theta_n^{y, *}\in \argmin_{\theta^y} \mathcal{L}_n^y(\theta^y)$ is denoted by $\widehat{\theta}_n^y$ and the final approximation is given by $\widehat{Y}_n^\pi\coloneqq \varphi(X_n^\pi\vert\widehat{\theta}_n^y)$.
At last, motivated by the continuity of the processes $\{(Y_t, Z_t)\}_{0\leq t\leq T}$ in the Malliavin framework, we initialize the parameters of the next time step's parametrizations according to
\begin{align}\label{eq:transfer_learning}
	\theta^z = \widehat{\theta}_n^z,\qquad \theta^\gamma=\widehat{\theta}^\gamma_n,\qquad \theta^y = \widehat{\theta}_n^y.
\end{align}
Such a \emph{transfer learning} trick guarantees a good initialization of the SGD iterations for $\widecheck{Y}_{n-1}^\pi, \widecheck{Z}_{n-1}^\pi, \widecheck{\Gamma}_{n-1}^\pi$, simplifying the learning problem and reducing the number of iteration steps required for convergence. For an empirical assessment on the efficiency of this transfer learning trick we refer to \cite[Sec.5.3]{chen_deep_2021}.

\begin{algorithm}[htp]
	%\setstretch{1.45}
	\KwIn{$\pi(N)$, $\vartheta_y\in[0, 1]$ -- discretization parameters}
	\KwIn{$B\in \mathds{N}^+$, $I\in \mathds{N}$, $\eta:\mathds{N}\to\mathds{R}$ -- training parameters}
	\KwResult{$\{(\widehat{Y}_n^\pi, \widehat{Z}_n^\pi, \widehat{\Gamma}_n^\pi)\}_{n=0,\dots, N}$ -- discrete time approximations over $\pi$}
	
	$\widehat{Y}_N^\pi\gets g(X_N^\pi)$, $\widehat{Z}_N^\pi \gets  \nabla_x g(X_N^\pi)\sigma(t_N, X_N^\pi),\quad \widehat{\Gamma}_N^\pi\gets \nabla_x(\nabla_x g\sigma)(t_N, X_N^\pi)$ -- collect terminal condition
	
	$\varphi(\cdot\vert\theta^y):\mathds{R}^{d\times 1}\to \mathds{R}$, $\psi(\cdot\vert\theta^z):\mathds{R}^{d\times 1}\to\mathds{R}^{1\times d}$, $\chi(\cdot\vert\theta^\gamma):\mathds{R}^{d\times 1}\to \mathds{R}^{d\times d}$ -- neural network parametrizations
	
	\For{$n=N-1, \dots, 0$}{
		
		\uIf{$n=N-1$}{
			$\theta^{z, (0)}, \theta^{y, (0)}$ -- initialize parameter sets, according to \cite{glorot_understanding_2010}
		}
		\Else{
			$\theta^{z, (0)} \gets \widehat{\theta}_{n+1}^z,\quad \theta^{y, (0)}\gets \widehat{\theta}_{n+1}^y$ -- transfer learning initialization
		}
	
		\vspace{0.2cm}
		\underline{\textbf{Solve \autoref{scheme:osm:implementable:z}--\autoref{scheme:osm:implementable:dz}.}}
		
		\For{$i=0, \dots, I-1$}{
			$\{\{X_m^\pi(b)\}_{0\leq m\leq N}\}_{b=1}^B$ -- Euler-Maruyama simulations by \autoref{scheme:sde:euler}
			
			$\{D_nX_{n+1}^\pi(b)\}_{b=1}^B$ -- Euler-Maruyama approximations by \autoref{scheme:sde:dx:euler}

			calculate empirical loss of \autoref{loss:z:parametrized} or \autoref{loss:z:ad}
			\begin{align}
				\widehat{\mathcal{L}}_n^{z, \gamma}(\theta^{z, (i)}, \theta^{\gamma(i)}) =\begin{aligned}[t]
					\frac{1}{B}\sum_{b=1}^B \lvert &(1 + \Delta t_n \nabla_y f(t_{n+1}, \Xhat{n+1}^\pi(b)))D_n\widehat{Y}_{n+1}^\pi(b)\\
					&+ \Delta t_n \nabla_x f(t_{n+1}, \Xhat{n+1}^\pi(b))D_{n}X_{n+1}^\pi(b)
					- \psi(X_n^\pi(b)\vert\theta^{z, (i)})\\
					&+ \Delta t_n\nabla_z f(t_{n+1}, \Xhat{n+1}^\pi(b))\chi(X_n^\pi(b)\vert\theta^{\gamma, (i)})\sigma(t_n, X_n^\pi)\\
					&- \chi(X_n^\pi(b)\vert\theta^{\gamma, (i)})\sigma(t_n, X_n^\pi(b))\Delta W_n(b)\rvert^2
				\end{aligned} 
			\end{align}
			
			$(\theta^{z, (i+1)}, \theta^{\gamma, (i+1)}) \gets (\theta^{z, (i)}, \theta^{\gamma, (i)}) - \eta(i)\nabla_{(\theta^z, \theta^\gamma)}\widehat{\mathcal{L}}_n^z(\theta^{z, (i)}, \theta^{\gamma, (i)})$ -- stochastic gradient descent update
		}
		
		$\widehat{\theta}_n^z\gets \theta^{z, (i+1)}$, $\widehat{\theta}^\gamma_n\gets \theta^{\gamma, (i+1)}$ -- collect optimal parameter estimations
		
		$\widehat{z}_n^\pi(\cdot)\gets \psi(\cdot\vert\widehat{\theta}_n^z), \quad \widehat{\gamma}_n^\pi(\cdot)\gets \chi(\cdot\vert \widehat{\theta}_n^\gamma)$ -- collect approximations of $\widecheck{Z}_n^\pi, \widecheck{\Gamma}_n^\pi$
		
		\vspace{0.2cm}
		\underline{\textbf{Solve \autoref{scheme:osm:implementable:y}.}}
		
		\For{$i=0, \dots, I-1$}{
			$\{\{X_m^\pi(b)\}_{0\leq m\leq N}\}_{b=1}^B$ -- Euler-Maruyama simulations by \autoref{scheme:sde:euler}
			
			calculate empirical loss of \autoref{loss:y}
			\begin{align}
				\widehat{\mathcal{L}}_n^y(\theta^{y, (i)}) = \begin{aligned}[t]
					\frac{1}{B}\sum_{b=1}^B \lvert &\widehat{Y}_{n+1}^\pi(b) + (1-\vartheta_y)\Delta t_n f(t_{n+1}, \Xhat{n+1}^\pi(b))-\varphi(X_n^\pi(b)\vert\theta^{y, (i)})\\ &+\vartheta_y\Delta t_n f(t_n, X_n^\pi(b), \varphi(X_n^\pi(b)\vert\theta^{y, (i)}), \widehat{Z}_n^\pi(b))- \widehat{Z}_n^\pi(b)\Delta W_n(b)\rvert^2
				\end{aligned}
			\end{align}
			
			$\theta^{y, (i+1)} \gets \theta^{y, (i)} - \eta(i)\nabla_{\theta}\widehat{\mathcal{L}}_n^y(\theta^{y, (i)})$ -- stochastic gradient descent step
		}
		
		$\widehat{\theta}_n^y\gets \theta^{y, (i+1)}$ -- collect optimal parameter estimations
		
		$\widehat{y}_n^\pi(\cdot)\gets \varphi(\cdot\vert\widehat{\theta}_n^y)$ -- collect approximations of $\widecheck{Y}_n^\pi$
	}
\caption{One-Step Malliavin Algorithm (OSM)}
\label{alg:osm}
\end{algorithm}

\paragraph{Dimensionality, linearity and vector-Jacobian products.}  The main reason why no numerical scheme has been proposed to solve the Malliavin BSDE in \autoref{eq:fbsde_fbsde:malliavin_bsde} is related to dimensionality. Since the $\Gamma$ process is an $\mathds{R}^{d\times d}$-valued process, its computational complexity in a least-squares Monte Carlo method has a quadratic dependence on the number of dimensions $d$. Indeed, a least-squares Monte Carlo approach for the BSDE \autoref{eq:fbsde:bsde} essentially comes down to the approximation of $d+1$-many conditional expectations. If, in addition, one would also like to solve the Malliavin BSDE \autoref{eq:fbsde_fbsde:malliavin_bsde} this leads to $d^2$ additional conditional expectations to be approximated, induced by the $\Gamma$ process. This observation justifies the use of deep neural network parametrizations which enable good scalability in high-dimensions. Moreover, notice that the training of the loss function \autoref{loss:z:ad} through an SGD optimization requires differentiating the loss with respect to the parameters $\theta^z$ in each step. With the loss already depending on the Jacobian of the mapping $\psi(\cdot\vert\theta^z)$, this in particular implies that in each SGD step one needs to calculate the Hessian of a vector-valued mapping $\psi$ with respect to the parameters $\theta^z$. Consequently, for high-dimensional problems the training of \autoref{loss:z:ad} becomes excessively intensive from a computational point of view. Nonetheless, what makes the Deep BSDE approach corresponding to \autoref{loss:z:ad} efficiently implementable is the linearity of the Malliavin BSDE \autoref{eq:fbsde_fbsde:malliavin_bsde}. In fact, due to linearity, one can circumvent explicitly calculating the Jacobian matrix of $Z$ as it suffices to calculate the vector-Jacobian product
\begin{align}\label{eq:vjp}
	\nabla_z f(t_{n+1}, \Xhat{n+1}^\pi)\nabla_x \psi(X_n^\pi\vert\theta^z) = \nabla_x \braket{v}{\psi(X_n^\pi\vert\theta^z)},\qquad v\coloneqq \nabla_z f(t_{n+1}, \Xhat{n+1}^\pi),
\end{align}
which boils down to computing a gradient instead.
This mitigates the computational costs of minimizing the automatic differentiated loss function in \autoref{loss:z:ad} in an SGD iteration.

\subsection{Regression error analysis}
In order to conclude the discussion on fully-implementable schemes for \autoref{scheme:osm:implementable}, we extend the discretization error results established by \autoref{thm:osm:discretization}, so that it incorporates the approximation errors of the arising conditional expectations. Even though we focus on the Deep BSDE approach, our arguments naturally extend to the BCOS estimates.
We consider shallow neural networks, with $S_1$-many hidden neurons and a hyperbolic tangent activation.
While distinguishing between the parametrized and automatic differentiated $\Gamma$ variants -- see \autoref{loss:z:parametrized} and \autoref{loss:z:ad}, respectively --, we rely on the following subclass of shallow neural networks
\begin{align}
	\Sigma_{C_{b}^{2}}(\Theta)\coloneqq \Big\{\psi(x\vert\theta^z(S_1)) \coloneqq W_1^2(S_1) \tanh(W_0^1(S_1) x + b_0) + b_1: \sum_{i=1}^{d}\sum_{j=1}^{S_{1}} \abs{[W_1^2(S_1)]_{i, j}} + \abs{[W_0^1(S_1)]_{j, i}}\leq \Upsilon({S_{1}})\Big\},
\end{align}
for some dominating sequence $\Upsilon: \mathds{N}_+\to\mathds{R}$. Then, due to the boundedness of the hyperbolic tangent function and its first two derivatives, the following upper bounds are in place for any $\psi(\cdot\vert\theta^z)\in \Sigma_{C_{b}^{2}}$
\begin{align}\label{eq:nn:derivative_bounds}
	\sup_{x\in\mathds{R}^{d\times 1}}\abs{\psi(x\vert\theta^z)} \leq \Upsilon({S_{1}}),\qquad \sup_{x\in\mathds{R}^{d\times 1}}\abs{\nabla_x \psi(x\vert\theta^z)}\leq\Upsilon^2({S_{1}}),\qquad \sup_{x\in\mathds{R}^{d\times 1}}\abs{\Hess \psi(x\vert\theta^z)}\leq\Upsilon^3({S_{1}}).
\end{align}
In light of \autoref{thm:hornik:uat:sobolev}, the hyperbolic tangent function is $\ell=1$-finite. Subsequently the family of shallow networks of the form \autoref{architecture:nn} is dense in $H^1(U)$ for any compact subset $U\subset \mathds{R}^{d\times 1}$.

The final approximations are denoted by $\widehat{Y}_{n+1}^\pi\coloneqq \widehat{y}_{n}^\pi(X_n^\pi)\eqqcolon \varphi(X_n^\pi\vert \widehat{\theta}_n^y)$, $\widehat{z}_{n+1}^\pi\coloneqq \widehat{z}_{n}^\pi(X_n^\pi)\eqqcolon \psi(X_n^\pi\vert \widehat{\theta}_n^z)$ and $\widehat{\Gamma}_{n+1}^\pi\coloneqq \widehat{\gamma}_{n}^\pi(X_n^\pi)\eqqcolon \chi(X_n^\pi\vert \widehat{\theta}_n^\gamma)$.
We introduce the short hand notations $\Delta \widecheck{Y}_n^\pi\coloneqq Y_{t_{n}} - \widecheck{Y}_n^\pi$, $\Delta \widecheck{Z}_n^\pi\coloneqq Z_{t_{n}} - \widecheck{Z}_n^\pi$, $\Delta \widecheck{\Gamma}_n^\pi\coloneqq \Gamma_{t_{n}} - \widecheck{\Gamma}_n^\pi$, and $\Delta \widehat{Y}_n^\pi\coloneqq Y_{t_{n}} - \widehat{Y}_n^\pi$, $\Delta \widehat{Z}_n^\pi\coloneqq Z_{t_{n}} - \widehat{Z}_n^\pi$, $\Delta \widehat{\Gamma}_n^\pi\coloneqq \Gamma_{t_{n}} - \widehat{\Gamma}_n^\pi$.
In light of the UAT property in \autoref{thm:hornik:uat:sobolev}, we define the \emph{regression biases}
\begin{align}\label{def:regression_errors}
	\begin{split}
		\epsilon_n^y&\coloneqq \inf_{\theta^y}\Expectation{\abs{\widecheck{y}_n^\pi(X_{n}^\pi) - \varphi(X_n^\pi\vert\theta^y)}^2},\\
		\epsilon_n^{z}&\coloneqq \inf_{\theta^z}\Expectation{\abs{\widecheck{z}_n^\pi(X_n^\pi) - \psi(X_n^\pi\vert\theta^z)}^2},\qquad 
		\epsilon_n^\gamma \coloneqq \inf_{\theta^\gamma} \Expectation{\abs{\left(\widecheck{\gamma}_n^\pi(X_n^\pi) - \chi(X_n^\pi\vert\theta^\gamma)\right)\sigma(t_n, X_n^\pi)}^2},\\
		\epsilon_n^{z, \nabla z}&\coloneqq \inf_{\theta^z}\Expectation{\abs{\widecheck{z}_n^\pi(X_n^\pi) - \psi(X_n^\pi\vert\theta^z)}^2 + \Delta t_n \abs{\left(\nabla_x \widecheck{z}_n^\pi(X_n^\pi) - \nabla_x \psi(X_n^\pi\vert\theta^z)\right)\sigma(t_n, X_n^\pi)}^2}.
	\end{split}
\end{align}
The goal is to establish an upper bound for the total approximation error defined by
\begin{align}
	\widehat{\mathcal{E}}^\pi(\abs{\pi})\coloneqq \max_n \Expectation{\abs{\Delta \widehat{Y}_n^\pi}^2} + \max_n \Expectation{\abs{\Delta \widehat{Z}_n^\pi}^2} + \Expectation{\sum_{n=0}^{N-1} \int_{t_{n}}^{t_{n+1}} \abs{\Gamma_r - \widehat{\Gamma}_n^\pi}^2\mathrm{d}r},
\end{align}
depending on not just the discretization but also the regression errors arising from the approximations of the conditional expectations in \autoref{scheme:osm:implementable}.
\begin{theorem}\label{thm:osm:regression}
	Let the conditions of \autoref{ass:error_analysis} be in place. Then, for sufficiently small $\abs{\pi}$, the total approximation error of the OSM scheme defined by the loss function \autoref{loss:z:parametrized} admits to
	\begin{align}\label{eq:osm:total_error:parametrized}
		\widehat{\mathcal{E}}^\pi(\abs{\pi})\leq C\left(\abs{\pi} + N\sum_{n=0}^{N-1} \{\epsilon_n^y+\epsilon_n^{z}\} +\sum_{n=0}^{N-1} \epsilon_n^\gamma\right).
	\end{align}
	Furthermore, in case the $\Gamma$ process is taken as the direct derivative of the $Z$ process as in \autoref{loss:z:ad}, the total error can be bounded by
	\begin{align}\label{eq:osm:total_error:ad}
		\widehat{\mathcal{E}}^\pi(\abs{\pi})\leq C\left(\abs{\pi} + N\sum_{n=0}^{N-1} \{\epsilon_n^y+\epsilon_n^{z, \nabla z}\} + \frac{\Upsilon^6({S_{1}})}{N}\right),
	\end{align}
	where $C$ is a constant independent of the time partition $\pi^N$.
\end{theorem}
\begin{proof}
	Throughout the proof $C$ denotes a constant independent of the time partition, whose value may vary from line to line. We only highlight arguments which significantly differ from the ones of \autoref{thm:osm:discretization}.
	
	\begin{steps}[wide, labelwidth=0pt, labelindent=0pt]
		\item \textit{Discrete estimates for $Y$, $Z$ and $\Gamma$}. Steps analogously to \autoref{thm:osm:discretization} -- see \autoref{error:disc:dzprojection} and \autoref{eq:error_analysis:estimate:precombined} in particular --, on top of the inequality $\rho>0: (1 - \rho)a^2 - (1/\rho) b^2\leq (1 - \rho)a^2 + (1 - 1/\rho)b^2\leq (a+b)^2$, lead to
		\begin{align}
			\Delta t_n\Expectation{\abs{\DZprojection{n}{n+1} - D_n\widehat{Z}_n^\pi}^2}	&\leq \begin{aligned}[t]
				&8d \left\{\Expectation{\abs{\Delta D_n\widehat{Y}_{n+1}^\pi}^2} - \Expectation{\abs{\CondExpn{n}{\Delta D_n\widehat{Y}_{n+1}^\pi}}^2}\right\}\\
				&+16dL_{\fD}^2\Delta t_n\begin{aligned}[t]
					\Bigg\{&C\Delta t_n^2 + 2\Delta t_n \left(\Expectation{\abs{\Delta X_{n+1}^\pi}^2} + \Expectation{\abs{\Delta \widehat{Y}_{n+1}^\pi}^2}\right)\\
					&+ 2\Delta t_n\left(\Expectation{\abs{\Delta \widehat{Z}_{n+1}^\pi}^2 + \Expectation{\abs{\Delta D_nX_{n+1}^\pi}^2}}\right)\\
					&+2\Delta t_n \Expectation{\abs{\Delta D_n\widehat{Y}_{n+1}^\pi}^2}\\
					&+\Expectation{\int_{t_{n}}^{t_{n+1}} \abs{D_{t_{n}}Z_r - \DZprojection{n}{n+1}}^2\mathrm{d}r}\Bigg\},
				\end{aligned}\\
				& +2\Delta t_n\Expectation{\abs{D_n\widehat{Z}_n^\pi - D_n\widecheck{Z}_n^\pi}^2},
			\end{aligned}\label{eq:regression_error:ground_estimates:gamma}\\
			(1-\beta\Delta t_n)\left(\Expectation{\abs{\Delta \widehat{Y}_n^\pi}^2} + \Expectation{\abs{\Delta \widehat{Z}_n^\pi}^2}\right)&\leq\begin{aligned}[t]
				&(1+C\Delta t_n)\left(\Expectation{\abs{\Delta \widehat{Y}_{n+1}^\pi}^2} + \Expectation{\abs{\Delta \widehat{Z}_{n+1}^\pi}^2}\right)\\
				&+C\left\{\Delta t_n^2+\Expectation{\int_{t_{n}}^{t_{n+1}} \abs{D_{t_{n}}Z_r - \DZprojection{n}{n+1}}^2\mathrm{d}r}\right\}\\
				&+ \frac{1}{\beta\Delta t_n}\left(\Expectation{\abs{\widehat{Y}_{n}^\pi - \widecheck{Y}_{n}^\pi}^2}+ \Expectation{\abs{\widehat{Z}_{n}^\pi - \widecheck{Z}_n^\pi}^2}\right)\\
				&+ C\Delta t_n\Expectation{\abs{\left(\widehat{\Gamma}_n^\pi - \widecheck{\Gamma}_n^\pi\right)\sigma(t_n, X_n^\pi)}^2},
			\end{aligned}\label{eq:regression_error:ground_estimates:y_z}
		\end{align}
		with any $\beta>0$.
		
		\item \textit{Regression errors induced by the loss functions.} Using the definition \autoref{def:dz:martingale_representation} and the relation \autoref{eq:dz:martingale_representation=dz}, the loss function in \autoref{loss:z:parametrized} can be rewritten as follows
		\begin{align}\label{eq:regression_error:widetilde_l}
			\begin{split}
				\mathcal{L}_n^{z, \gamma}(\theta^z, \theta^\gamma)&= \begin{aligned}[t]
					&\Expectation{\abs{\widecheck{Z}_n^\pi-\psi(X_n^\pi\vert\theta^z) + \Delta t_n\nabla_z f(t_{n+1}, \Xhat{n+1}^\pi) \left(\chi(X_n^\pi\vert \theta^\gamma) - \widecheck{\Gamma}_n^\pi\right)\sigma(t_n, X_n^\pi)}^2}\\
					&+ \Delta t_n \Expectation{\abs{\left(\widecheck{\Gamma}_n^\pi - \chi(X_n^\pi\vert\theta^\gamma)\right)\sigma(t_n, X_n^\pi)}^2} + \Expectation{\int_{t_{n}}^{t_{n+1}}\abs{D_n\widetilde{Z}_r - D_n\widecheck{Z}_n^\pi}^2\mathrm{d}r}
				\end{aligned}\\
				&\eqqcolon \widetilde{\mathcal{L}}_n^{z, \gamma}(\theta^z, \theta^\gamma)+ \Expectation{\int_{t_{n}}^{t_{n+1}}\abs{D_n\widetilde{Z}_r - D_n\widecheck{Z}_n^\pi}^2\mathrm{d}r}.
			\end{split}
		\end{align}
		The inequality $(a+b)^2\leq (1+\varrho_1\Delta t_n)a^2+(1 + 1/(\varrho_1 \Delta t_n))b^2$, on top of the bounded differentiability of $f$ provided by \autoref{ass:error_analysis}, implies
		\begin{align}
			\widetilde{\mathcal{L}}_n^{z, \gamma}(\theta^z, \theta^\gamma)\leq \begin{aligned}[t]
				&(1+\varrho_1\Delta t_n)\Expectation{\abs{\widecheck{Z}_n^\pi - \psi(X_n^\pi\vert\theta^z)}^2}\\
				&+ \left[\frac{L_{\nabla f}^2}{\rho_1}(1+\varrho_1\Delta t_n) + 1\right]\Delta t_n \Expectation{\abs{\left(\widecheck{\Gamma}_n^\pi - \chi(X_n^\pi\vert\theta^\gamma)\right)\sigma(t_n, X_n^\pi)}^2}.
			\end{aligned}
		\end{align}
		By the inequality $(a+b)^2\geq (1-\varrho_2\Delta t_n)a^2+(1 - 1/(\varrho_2 \Delta t_n))b^2\geq (1-\varrho_2\Delta t_n)a^2-1/(\varrho_2 \Delta t_n)b^2$, the following also holds
		\begin{align}
			\widetilde{\mathcal{L}}_n^{z, \gamma}(\theta^z, \theta^\gamma)\geq \begin{aligned}[t]
				&(1-\varrho_2\Delta t_n)\Expectation{\abs{\widecheck{Z}_n^\pi - \psi(X_n^\pi\vert\theta^z)}^2}+\left(1- \frac{L^2_{\nabla f}}{\varrho_2}\right)\Delta t_n \Expectation{\abs{\left(\widecheck{\Gamma}_n^\pi - \chi(X_n^\pi\vert\theta^\gamma)\right)\sigma(t_n, X_n^\pi)}^2}.
			\end{aligned}
		\end{align}
		Choosing $\varrho^*_2\coloneqq 2L^2_{\nabla f}$, we subsequently have
		\begin{align}
			\widetilde{\mathcal{L}}_n^{z, \gamma}(\theta^z, \theta^\gamma)\geq (1-\varrho_2^*\Delta t_n)\Expectation{\abs{\widecheck{Z}_n^\pi - \psi(X_n^\pi\vert\theta^z)}^2} + \frac{\Delta t_n}{2}\Expectation{\abs{\left(\widecheck{\Gamma}_n^\pi - \chi(X_n^\pi\vert\theta^\gamma)\right)\sigma(t_n, X_n^\pi)}^2}.
		\end{align}
		Assuming that $(\widehat{\theta}_n^z, \widehat{\theta}_n^\gamma)$ is a perfect approximation -- see \autoref{remark:sgd} -- of the minimal parameter space $(\theta_n^{z, *}, \theta_n^{\gamma, *})\in \argmin_{\theta^z, \theta^\gamma}\mathcal{L}_n^{z, \gamma}(\theta^z, \theta^\gamma)$ -- which in light of \autoref{eq:regression_error:widetilde_l} also minimizes $\widetilde{\mathcal{L}}_n^{z, \gamma}(\theta^z, \theta^\gamma)$ --, we have $\widetilde{\mathcal{L}}_n^{z, \gamma}(\widehat{\theta}_n^z, \widehat{\theta}_n^\gamma)\leq \widetilde{\mathcal{L}}_n^{z, \gamma}(\theta^z, \theta^\gamma)$ for any $(\theta^z, \theta^\gamma)$. In particular,
		for any sufficiently small $\Delta t_n$ satisfying $\varrho_2^*\Delta t_n\leq 1/2$, we gather
		\begin{align}\label{eq:regression_error:loss:z_and_gamma}
			\Expectation{\abs{\widecheck{Z}_n^\pi - \widehat{Z}_n^\pi}^2} + \Delta t_n\Expectation{\abs{\left(\widecheck{\Gamma}_n^\pi - \widehat{\Gamma}_n^\pi\right)\sigma(t_n, X_n^\pi)}^2} \leq C\begin{aligned}[t]\Bigg(
				&\inf_{\theta^z}\Expectation{\abs{\widecheck{Z}_n^\pi - \psi(X_n^\pi\vert\theta^z)}^2}\\
				&+ \inf_{\theta^\gamma}\Delta t_n\Expectation{\abs{\left(\widecheck{\Gamma}_n^\pi - \chi(X_n^\pi\vert\theta^\gamma)\right)\sigma(t_n, X_n^\pi)}^2}\Bigg).
			\end{aligned}
		\end{align}
		
		Through analogous steps to \cite[Thm. 4.1, step 3-4]{hure_deep_2020} a similar estimate can be established for the loss function \autoref{loss:y}, ultimately giving
		\begin{align}\label{eq:regression_error:loss:y}
			\Expectation{\abs{\widecheck{Y}_n^\pi - \widehat{Y}_n^\pi}^2}\leq C\inf_{\theta^y}\Expectation{\abs{\widecheck{Y}_n^\pi - \varphi(X_n^\pi\vert\theta^y)}^2}\eqqcolon C\epsilon_n^y.
		\end{align}
		
		\item \textit{Approximation error bound in the parametrized case.} Recalling the definitions in \autoref{def:regression_errors}, combining \autoref{eq:regression_error:ground_estimates:y_z} with the estimates \autoref{eq:regression_error:loss:z_and_gamma} and \autoref{eq:regression_error:loss:y} on top of the discrete Grönwall lemma, implies the total approximation error of $Y$ and $Z$ in \autoref{eq:osm:total_error:parametrized} -- given small enough time steps admitting to $\beta \Delta t_n<1$. The $\Gamma$ estimate then follows in a similar manner to Step 5 in \autoref{thm:osm:discretization} using the estimates \autoref{eq:regression_error:ground_estimates:gamma} and \autoref{eq:regression_error:ground_estimates:y_z}, observing that $(1+C\Delta t_n)/(1-\beta\Delta t_n) - 1$ is $\mathcal{O}(\abs{\pi})$ given $\beta\Delta t_n<1$. This completes the total approximation error of \autoref{eq:osm:total_error:parametrized}.

		\item \textit{Derivative representation error of $Z$ and $\Gamma$.} In order to prove \autoref{eq:osm:total_error:ad}, we need to establish an error estimate bounding the difference between the spatial derivative of \autoref{scheme:osm:implementable:z} and the target of \autoref{scheme:osm:implementable:dz}. Notice that under the conditions of \autoref{ass:error_analysis} and \autoref{eq:nn:derivative_bounds}, the arguments of the conditional expectations are all $C_b^2$ in $x$. Then, formal differentiation of \autoref{scheme:osm:implementable:z} with the Leibniz rule and the integration-by-parts formula in \autoref{eq:app:integration-by-parts:multi-d:vector} applied on  \autoref{scheme:osm:implementable:dz}, gives
		% explicit calculations
		\iffalse
		\begin{align}
			\nabla_x \widecheck{z}_n^\pi(X_n^\pi) &= \begin{aligned}[t]
				&\Delta t_n \CondExpn{n}{\nabla_x\nabla_y f(t_{n+1}, \Xhat{n+1}^\pi)\widehat{Z}_{n+1}^\pi} + \CondExpn{n}{(1+\Delta t_n\nabla_y f(t_{n+1}, \Xhat{n+1}^\pi))\nabla_x \widehat{z}_{n+1}^{\pi}(\Xhat{n+1}^\pi)}\\
				&+\Delta t_n \CondExpn{n}{\nabla_{xx}^2 f(t_{n+1}, \Xhat{n+1}^\pi)}\sigma\\
				&+\Delta t_n\CondExpn{n}{\nabla_x \nabla_z f(t_{n+1}, \Xhat{n+1}^\pi)}\widehat{\gamma}_n^\pi(X_n^\pi)\sigma\\
				&+ \Delta t_n \CondExpn{n}{\nabla_z f(t_{n+1}, \Xhat{n+1}^\pi)}\nabla_x\widehat{\gamma}_n^\pi(X_n^\pi)\sigma
			\end{aligned}\\
			\widecheck{\gamma}_n^\pi(X_n^\pi)\sigma &= \begin{aligned}[t]
				&\Delta t_n \CondExpn{n}{\nabla_x\nabla_y f(t_{n+1}, \Xhat{n+1}^\pi)\widehat{Z}_{n+1}^\pi}\sigma+\CondExpn{n}{(1+\Delta t_n \nabla_y f(t_{n+1}, \Xhat{n+1}^\pi))\nabla_x \widehat{z}_{n+1}^\pi(X_{n+1}^\pi)}\sigma\\
				&+ \Delta t_n \CondExpn{n}{\nabla_{xx}^2 f(t_{n+1}, \Xhat{n+1}^\pi)}\sigma^2\\
				&+\Delta t_n (\widecheck{\gamma}_n^\pi\sigma)^T \CondExpn{n}{\nabla_x\nabla_z f(t_{n+1}, \Xhat{n+1}^\pi)}\sigma
			\end{aligned}
		\end{align}
		\fi
		\begin{align}
			\left(\nabla_x \widecheck{z}_n^\pi(X_n^\pi) - \widecheck{\gamma}_n^\pi(X_n^\pi)\right)\sigma = \begin{aligned}[t]
				&\Delta t_n [\left(\widehat{\gamma}_n^\pi(X_n^\pi) - \widecheck{\gamma}_n^\pi(X_n^\pi)\right)\sigma]^T \CondExpn{n}{\nabla_x\nabla_z f(t_{n+1}, \Xhat{n+1}^\pi)}\sigma\\
				&+\Delta t_n \CondExpn{n}{\nabla_x \nabla_z f(t_{n+1}, \Xhat{n+1}^\pi)}\widehat{\gamma}_n^\pi(X_n^\pi)\sigma\\
				&-\Delta t_n [\widehat{\gamma}_n^\pi(X_n^\pi)\sigma]^T\CondExpn{n}{\nabla_x\nabla_z f(t_{n+1}, \Xhat{n+1}^\pi)}\sigma\\
				&+ \Delta t_n \CondExpn{n}{\nabla_z f(t_{n+1}, \Xhat{n+1}^\pi)}\nabla_x \widehat{\gamma}_n^\pi(X_n^\pi)\sigma^2.
			\end{aligned}
		\end{align}
		By the bounded differentiability conditions in \ref{ass:error_analysis:bsde:f}, we have that
		\begin{align}
			\Expectation{\abs{\left(\nabla_x \widecheck{z}_n^\pi(X_n^\pi) - \widecheck{\gamma}_n^\pi(X_n^\pi)\right)\sigma}^2}\leq \begin{aligned}[t]
				&4 \Delta t_n^2 L_{\nabla^{2}f} \abs{\sigma}^2 \Expectation{\abs{\left(\widehat{\gamma}_n^\pi(X_n^\pi) - \widecheck{\gamma}_n^\pi(X_n^\pi)\right)\sigma}^2} +  4\Delta t_n^2 L_{\nabla^{2}f} \abs{\sigma}^4\Expectation{\abs{\widehat{\gamma}_n^\pi(X_n^\pi)}^2}\\
				&+ 4\Delta t_n^2 L^2_{\nabla f} \abs{\sigma}^4	\Expectation{\abs{\nabla_x \widehat{\gamma}_n^\pi(X_n^\pi)}^2}.
			\end{aligned}
		\end{align}
		Splitting the first term according to $\widehat{\gamma}_n^\pi(X_n^\pi) - \widecheck{\gamma}_n^\pi(X_n^\pi) = \widehat{\gamma}_n^\pi(X_n^\pi) - \nabla_x \widecheck{z}_n^\pi(X_n^\pi) + \nabla_x \widecheck{z}_n^\pi(X_n^\pi) - \nabla_x \widecheck{\gamma}_n^\pi(X_n^\pi)$, using the direct estimate $\widehat{\gamma}_n^\pi(X_n^\pi)\equiv \nabla_x \widehat{z}_n^\pi(X_n^\pi)$ implied by \autoref{loss:z:ad}, and recalling the bounds in \autoref{eq:nn:derivative_bounds}, subsequently yields
		\begin{align}
			\Expectation{\abs{\left(\nabla_x \widecheck{z}_n^\pi(X_n^\pi) - \widecheck{\gamma}_n^\pi(X_n^\pi)\right)\sigma}^2}\leq C\left(\Expectation{\abs{\left(\nabla_x \widecheck{z}_n^\pi(X_n^\pi) - \nabla_x \widehat{z}_n^\pi\right)\sigma}^2} + \Upsilon^4(S_1)+ \Upsilon^6({S_1})\right),
		\end{align}
		for small enough time steps admitting to $8\Delta t_n^2 L_{\nabla^2 f}^2\abs{\sigma}^2<1$.
		Combining this estimate with the upper bound \autoref{eq:regression_error:loss:z_and_gamma}, recalling the definition of $\epsilon_n^{z, \nabla z}$ in \autoref{def:regression_errors}, we gather
		\begin{align}\label{eq:regression_error:loss:z_and_gamma:ad}
			\Expectation{\abs{\widecheck{Z}_n^\pi - \widehat{Z}_n^\pi}^2} + \Delta t_n\Expectation{\abs{\left(\widecheck{\Gamma}_n^\pi - \nabla_x \widehat{Z}_n^\pi\right)\sigma(t_n, X_n^\pi)}^2} \leq C\begin{aligned}[t]\Bigg(\epsilon_n^{z, \nabla z} + \Delta t_n^3\Upsilon^6(S_1)\Bigg),
			\end{aligned}
		\end{align}
		for small enough time steps $\Delta t_n<1$ and diverging $\Upsilon(S_1)$. The total approximation error estimate in \autoref{eq:osm:total_error:ad} then follows in a similar manner, combining \autoref{eq:regression_error:loss:z_and_gamma:ad} with \autoref{eq:regression_error:loss:y}, \autoref{eq:regression_error:ground_estimates:y_z} and the discrete Grönwall lemma, as in the previous step.
		
		This completes the proof.
		
	\end{steps}
\end{proof}

\autoref{thm:osm:regression} establishes the convergence of the Deep BSDE approach to \autoref{scheme:osm:implementable}, given the UAT property of neural networks provided by \autoref{thm:hornik:uat:sobolev}. The first terms in both \autoref{eq:osm:total_error:parametrized} and \autoref{eq:osm:total_error:ad} correspond to the discrete time approximation errors in \autoref{thm:osm:discretization}. The second terms correspond to the approximations of the neural network regression Monte Carlo approach. Provided by \autoref{thm:hornik:uat:sobolev}, the corresponding regression biases defined by \autoref{def:regression_errors} can be made arbitrarily small with the choice of shallow neural networks already. In exchange to avoid the parametrization in the automatic differentiation approach in \autoref{loss:z:ad}, one needs to restrict the parametrization to the case of $\Sigma_{C_{b}^{2}}$ neural networks and subsequently has to deal with an additional error term in \autoref{eq:osm:total_error:ad}, which depends on the increasing sequence $\Upsilon(S_1)$, controlling the magnitude of the parameters. If this dominating sequence is such that $\Upsilon^6(S_1)/N\to 0$ while $S_1, N\to \infty$ this ensures the existence of neural networks $\varphi(\cdot\vert\theta^{y}), \psi(\cdot\vert\theta^z)\in \Sigma_{C_{b}^{2}}$ such that the total approximation error converges. We shall, however, notice that the claim above guarantees nothing more, and in fact does not guarantee the convergence of the final approximations including regression errors, which we highlight in the remark below.

\begin{remark}[Limitations of \autoref{thm:osm:regression}]\label{remark:sgd}
	In the proof of \autoref{thm:osm:regression} we neglected the presence of three additional error terms. These are the following.
	\begin{enumerate}
		\item First, the definitions in \autoref{def:regression_errors} only express the regression biases due to the choice of a finite number of parameters. The actual regression errors also incorporate the approximation error of the optimal parameter space $\widehat{\theta}^z_n$ and induce a term $\Expectation{\abs{\varphi(X_n^\pi\vert\theta_n^{y, *}) - \varphi(X_n^\pi\vert\widehat{\theta}^y_n)}^2}$, 
		which stems from the fact that unlike in a linear regression method -- see, e.g., \cite{bender_least-squares_2012} --, one does not have a closed-form expression for the true minimizers $(\theta_n^{z, *}, \theta_n^{\gamma, *})$, $\theta_n^{y, *}$ but can only gather an approximation of them with a stochastic gradient descent optimization. The present understanding of this term is poor, mainly due to the non-convexity of the corresponding target function -- see \cite{jentzen_strong_2021} and the references therein. Currently, there exists no theoretical guarantee which would ensure the convergence of SGD iterations in the FBSDE context.
		
		\item The second term arises due to the fact that in practice one can only calculate an empirical version of the expectations in $\mathcal{L}^y_n, \mathcal{L}_n^{z, \gamma}, \mathcal{L}_n^{z, \nabla z}$. This induces a Monte Carlo simulation error of finitely many samples. However, as we shall see in the upcoming numerical section, thanks to the soft memory limitation of a single SGD step, one can pass so many realizations of the underlying Brownian motion throughout the optimization cycle that the magnitude of the corresponding error term becomes negligible compared to other sources of error.
		
		\item The final observation that needs to be highlighted is the compactness assumption on the domain in \autoref{thm:hornik:uat:sobolev}. This error term can be dealt with in a similar fashion to \cite[Remark 4.2]{hure_deep_2020}, where a localization argument is constructed in such a way that -- under suitable truncation ranges -- convergence is ensured.
	\end{enumerate}
\end{remark}

\section{Numerical experiments}\label{sec:numerical_experiments}

In order to show the accuracy and robustness of the proposed scheme, we present results of numerical experiments on three different types of problems.
We distinguish between the two Deep BSDE approaches for the OSM scheme, based on whether the $\Gamma$ process is parametrized with an $\mathds{R}^{d\times d}$-valued neural network -- see \autoref{loss:z:parametrized} --, or it is obtained as the direct Jacobian of the parametrization of the $Z$ process via automatic differentiation -- as in \autoref{loss:z:ad}. We label these variants by (P) and (D), respectively.
As a reference method, we compare the results of the OSM scheme to the first scheme (DBDP1) of Huré et al. \cite{hure_deep_2020}, which corresponds to the Euler discretization of \autoref{scheme:theta_bsde} when $\vartheta_y=\vartheta_z=1$. In accordance with their findings, we found the parametrized version (DBDP1) more robust than the automatic differentiated one (DBDP2) in high-dimensional settings.

Each BSDE is discretized with $N$ equidistant time intervals, giving $\Delta t_n = T/N$ for all $n=0, \dots, N-1$. For the implicit $\vartheta_y$ parameter of the discretization in \autoref{scheme:osm}, we choose values $\vartheta_y\in\{0, 1/2, 1\}$.
In all upcoming examples we use fully-connected, feedforward neural networks of $L=2$ hidden layers with $S_l=100+d$ neurons in each layer. In line with \autoref{thm:osm:regression}, a hyperbolic tangent activation is deployed, yielding continuously differentiable parametrizations. Layer normalization \cite{ba_layer_2016} is applied in between the hidden layers. 
For the stochastic gradient descent iterations, we use the Adam optimizer with the adaptive learning rate strategy of \cite{chen_deep_2021} -- see $\eta(i)$ in \Cref{alg:osm}. The optimization is done as follows: in each backward recursion we allow $I=2^{15}$ SGD iterations for the $N-1$'th time step. Thereafter, we make use of the transfer learning initialization given by \autoref{eq:transfer_learning}, and reduce the number of iterations to $I=2^{11}$ for all preceding time steps. In each iteration step, the optimization receives a new, independent sample of the underlying forward diffusion with $B=2^{10}$ sample paths, meaning that in total the iteration processes $2^{25}$ and $2^{21}$ many realizations of the Brownian motion at time step $n=N-1$ and $n<N-1$, respectively. In order to speed up normalization, neural network trainings were carried out with single floating point precision.
For the implementation of the BCOS method, we choose $K=2^{9}$ Fourier coefficients, $P=5$ Picard iterations and truncate the infinite integrals to a finite interval of $[a, b]=[x_0 + \kappa_\mu-  L\sqrt{\kappa_\sigma}, x_0 + \kappa_\mu + L\sqrt{\kappa_\sigma}]$ where $\kappa_{\mu}= \mu(0, x_0) T$, $\kappa_{\sigma} = \sigma(0, x_0)T$. As in \cite{ruijter_numerical_2016}, we fix $L=10$.

The OSM method has been implemented in TensorFlow 2. In order to exploit static graph efficiency, all core methods are decorated with \texttt{tf.function} decorators. The library used in this paper will be publicly accessible under \href{https://github.com/balintnegyesi/MalliavinDeepBSDE}{github}. All experiments below were run on an \textsc{Dell} Alienware Aurora R10 machine, equipped with an AMD Ryzen 9 3950X CPU (16 cores, 64Mb cache, 4.7GHz) and an Nvidia GeForce RTX 3090 GPU (24Gb).
In order to assess the inherent stochasticity of both the regression Monte Carlo method and the SGD iterations, we run each experiment $5$ times and report on the mean and standard deviations of the resulting independent approximations.
$\mathds{L}^2$-errors are estimated over an independent sample of size $M=2^{10}$ produced by the same machinery as the one used for the simulations. Hence, the final error estimates are calculated as
\begin{align}
	\widehat{\mathds{E}}\left[\abs{\Delta \widehat{Y}_n^\pi}^2\right] = \frac{1}{M} \sum_{m=1}^M \abs{\Delta \widehat{Y}_{n}^{\pi}(m)}^2,\quad \widehat{\mathds{E}}\left[\abs{\Delta \widehat{Z}_n^\pi}^2\right] = \frac{1}{M} \sum_{m=1}^M \abs{\Delta \widehat{Z}_{n}^{\pi}(m)}^2,\quad \widehat{\mathds{E}}\left[\abs{\Delta \widehat{\Gamma}_n^\pi}^2\right] = \frac{1}{M} \sum_{m=1}^M \abs{\Delta \widehat{\Gamma}_{n}^{\pi}(m)}^2 
\end{align}
where $\Delta Y_n^{\pi}(m)$ corresponds to the $m$'th path of test sample, and similarly for other error measures.

\subsection{Example 1: reaction-diffusion with diminishing control}
The first, \emph{reaction-diffusion} type equation is taken from \cite[Example 2]{gobet_adaptive_2017}. Such equations are common in financial applications. The coefficients of the BSDE \autoref{eq:fbsde} are as follows
\begin{align}\label{example:1}
	\begin{split}
		\mu&=\mathbf{0}_d,\qquad \sigma=I_d,\qquad
		f(t, x, y, z) = \frac{\omega(t, \lambda x)}{\left[1 + \omega(t, \lambda x)\right]^2}\left[\lambda^2 d (y-\gamma) - 1 - \frac{\lambda^2}{2}d\right],\qquad 
		g(x) =\gamma + \frac{\omega(T, \lambda x)}{1 + \omega(T, \lambda x)},
	\end{split}
\end{align}
where $\omega(t, x) = \exp(t+\sum_{i=1}^d x_i)$. These parameters satisfy \autoref{ass:error_analysis}. The driver is independent of $Z$ and $\fD$ does not depend on the $Y$ process. Consequently, the solutions of \autoref{eq:fbsde_fbsde:bsde} and \autoref{eq:fbsde_fbsde:malliavin_bsde} can be separated into two disjoint problems. The analytical solutions are given by
\begin{align}
	X_t=W_t,\quad y(t, x) = \frac{\omega(t, \lambda x)}{1 + \omega(t, \lambda x)},\quad z(t, x) = \lambda \frac{\omega(t, \lambda x)}{(1+\omega(t+\lambda x))^2}\mathbf{1}_d,\quad 
	\gamma(t, x) = \lambda^2\frac{\omega(t, \lambda x)(1-\omega(t, \lambda x))}{(1+\omega(t, \lambda x))^3}\mathbf{1}_{d, d}.
\end{align}
We choose $T=0.5$, $\gamma = 0.6, \lambda = 1$ and fix $x_0=\mathbf{1}_d$. We consider $d\in\{1, 10\}$ with $\vartheta_y\in\{0, 1\}$.

In \autoref{fig:Ex1}, the convergence of the two fully-implementable schemes is assessed. \autoref{fig:Ex1:d1} depicts the convergence for $d=1$. The BCOS estimates, drawn with lines, show the same order of convergence as in \autoref{thm:osm:discretization}, confirming the theoretical findings of the discretization error analysis. The Deep BSDE approximations, depicted with scattered error bars, exhibit higher error figures, showcasing the presence of an additional regression component. Nevertheless, the complete approximation error of the corresponding regression estimates admit to the same order of convergence as in \autoref{thm:osm:regression}. The $\Gamma$ approximations corresponding to the parametrized (P) and automatic differentiated (D) cases, demonstrate the difference between the bounds in \autoref{eq:osm:total_error:parametrized} and \autoref{eq:osm:total_error:ad}. Indeed, we observe an extra error stemming from the bounded differentiability component of the neural networks -- see \autoref{eq:nn:derivative_bounds}. The convergence of the regression approximations flattens out for the finest time partition $N=100$ -- see the regression error of $Y$ in particular -- at a level of $\sim\mathcal{O}(10^{-7})$, indicating the presence of a regression bias induced by the restriction on a finite number of parameters.
In \autoref{fig:Ex1:d10}, the same dynamics are depicted for $d=10$, where we observe the same order of convergence, in accordance with \autoref{thm:osm:regression}. Note that the regression estimates of the $Z$ process converge until, and including, the finest time partition $N=100$ in case of the OSM disretization. On the other hand, with the approach of Huré et al. \cite{hure_deep_2020} the decay stops at $N=50$, indicating the impact of diverging conditional variances, as anticipated in \autoref{remark:conditional_variance}. \autoref{tab:Ex1} contains the means and standard deviations of a collection of error measures with respect to $5$ independent runs of the same regression Monte Carlo method. It can be seen that -- regardless of the value of $\vartheta_y$ -- the OSM scheme yields an order of magnitude improvement in the approximation of the $Z$ process, while showing identical error figures in the $Y$ process. Errors under the automatic differentiated case (D) with \autoref{loss:z:ad} are slightly better than in the parametrized approach (P). The $\Gamma$ approximations show comparable accuracies. The total runtime of the OSM regressions is approximately double of that of \cite{hure_deep_2020}, which is intuitively explained by the fact that \autoref{scheme:osm:implementable} solves two BSDEs at each point in time. Execution times under the automatic differentiated variant are slightly higher than in the parameterized case, confirming the extra computational complexity of Jacobian training in \autoref{loss:z:ad}. The neural network regression Monte Carlo method yields sharp, robust estimates with small standard deviations over independent runs of the algorithm, in particular corresponding the $Z$ process.

\begin{figure}[t]
	\centering
	\begin{subfigure}[t]{\textwidth}
		\centering
		\includegraphics[width=0.75\textwidth]{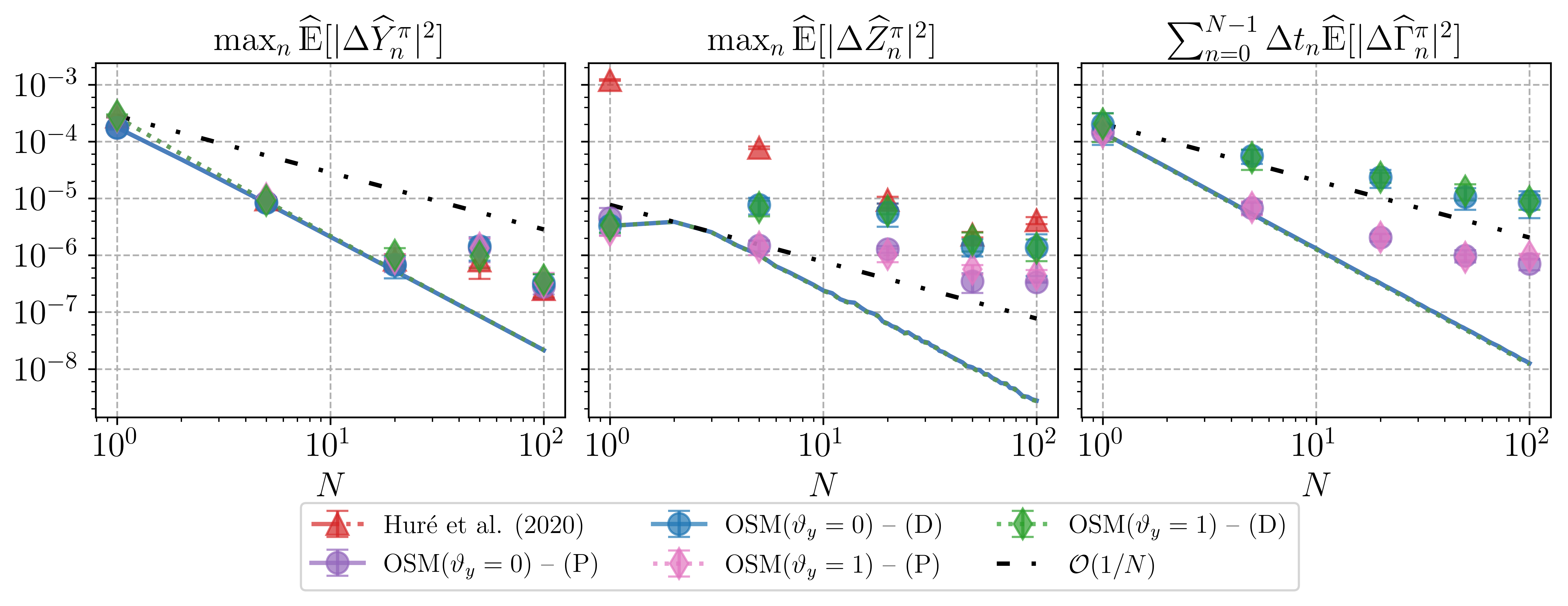}
		\caption{BCOS and Deep BSDE, $d=1$. From left to right: maximum mean-squared approximation errors of $Y$ and $Z$; average mean-squared approximation error of $\Gamma$. Lines correspond to BCOS estimates, scattered error bars to the means and standard deviations of $5$ independent neural network regressions.}
		\label{fig:Ex1:d1}
	\end{subfigure}
	
	\begin{subfigure}[t]{\textwidth}
		\centering
		\includegraphics[width=0.75\textwidth]{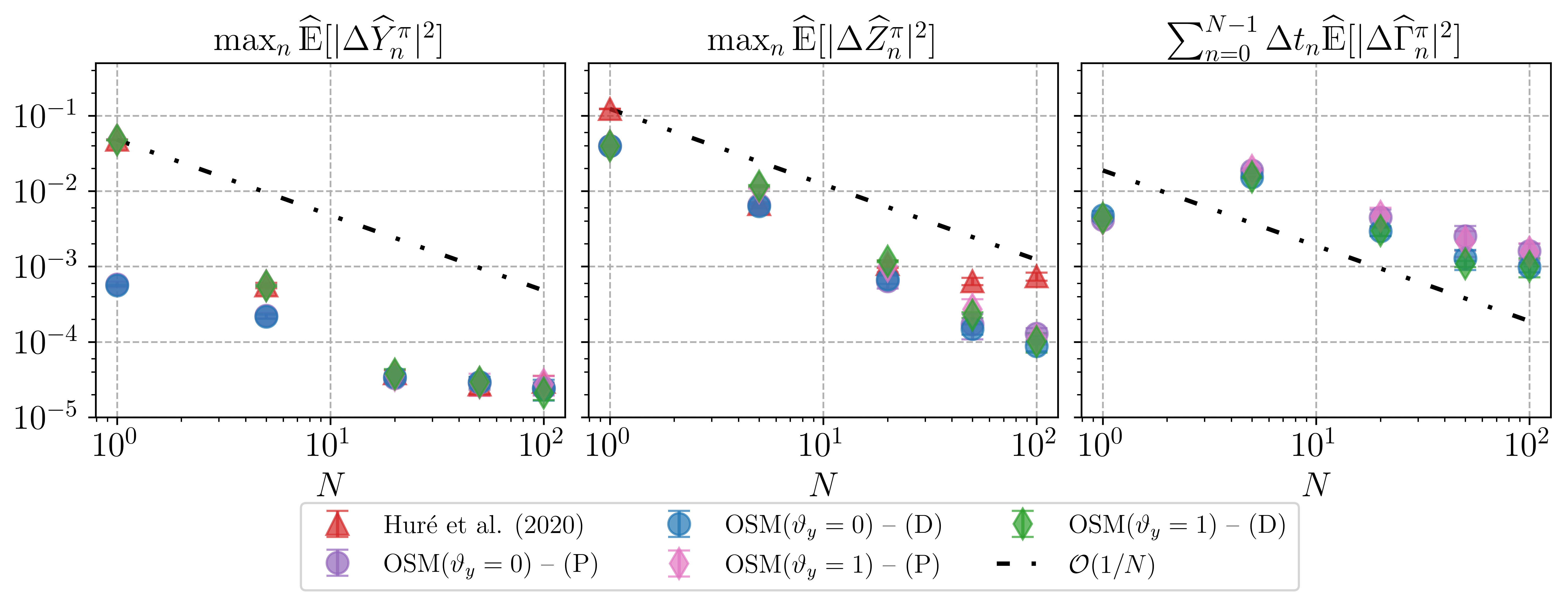}
		\caption{Deep BSDE, $d=10$. From left to right: maximum mean-squared approximation errors of $Y$ and $Z$; average mean-squared approximation error of $\Gamma$. Means and standard deviations are calculated over $5$ independent runs of the algorithm.}
		\label{fig:Ex1:d10}
	\end{subfigure}
	\caption{Example 1 in \autoref{example:1}. Convergence of approximation errors. Mean-squared errors are calculated over an independent sample of $M=2^{10}$ realizations of the underlying Brownian motion.}
	\label{fig:Ex1}
\end{figure}
\begin{table}
	\centering
	\caption{Example 1 in \autoref{example:1}, $d=10$, $N=100$. Summary of Deep BSDE estimates. Mean-squared errors are calculated over an independent sample of $M=2^{10}$ realizations of the underlying Brownian motion. Means and standard deviations (in parentheses) obtained over $5$ independent runs of the algorithm. Best estimates within one standard deviation highlighted in gray. $\Gamma$ estimates from Huré et al. in \cite{hure_deep_2020} are obtained via automatic differentiation.}
	\resizebox{\textwidth}{!}{\begin{tabular}{llllll}
\toprule
{} & \multicolumn{2}{l}{OSM($\vartheta_y=0$)} & \multicolumn{2}{l}{OSM($\vartheta_y=1$)} &                                             Huré et al. (2020) \\
{} &                                                            (P) &                                                            (D) &                                                          (P) & \multicolumn{2}{l}{(D)} \\
\midrule
$|\Delta \widehat{Y}_0^\pi|/|Y_0|$                                                  &    {\cellcolor{gray!20}$\num{3e-04} \left(\num{3e-04}\right)$} &    {\cellcolor{gray!20}$\num{3e-04} \left(\num{2e-04}\right)$} &                        $\num{6e-04}\left(\num{2e-04}\right)$ &    {\cellcolor{gray!20}$\num{2e-04} \left(\num{2e-04}\right)$} &                        $\num{1.1e-03}\left(\num{4e-04}\right)$ \\
$|\Delta \widehat{Z}_0^\pi|/|Z_0|$                                                  &    {\cellcolor{gray!20}$\num{7e-03} \left(\num{3e-03}\right)$} &    {\cellcolor{gray!20}$\num{8e-03} \left(\num{2e-03}\right)$} &  {\cellcolor{gray!20}$\num{9e-03} \left(\num{2e-03}\right)$} &    {\cellcolor{gray!20}$\num{9e-03} \left(\num{5e-03}\right)$} &    {\cellcolor{gray!20}$\num{9e-03} \left(\num{2e-03}\right)$} \\
$|\Delta \widehat{\Gamma}_0^\pi|$                                                   &                        $\num{1.2e-02}\left(\num{3e-03}\right)$ &    {\cellcolor{gray!20}$\num{8e-03} \left(\num{3e-03}\right)$} &  {\cellcolor{gray!20}$\num{9e-03} \left(\num{1e-03}\right)$} &    {\cellcolor{gray!20}$\num{8e-03} \left(\num{2e-03}\right)$} &                        $\num{9.9e+02}\left(\num{8e+01}\right)$ \\
$\max_n\widehat{\mathds{E}}[|\Delta \widehat{Y}_n^\pi|^2]$                          &  {\cellcolor{gray!20}$\num{2.4e-05} \left(\num{5e-06}\right)$} &  {\cellcolor{gray!20}$\num{2.4e-05} \left(\num{7e-06}\right)$} &                      $\num{2.7e-05}\left(\num{8e-06}\right)$ &  {\cellcolor{gray!20}$\num{2.1e-05} \left(\num{4e-06}\right)$} &                        $\num{2.9e-05}\left(\num{6e-06}\right)$ \\
$\max_n\widehat{\mathds{E}}[|\Delta \widehat{Z}_n^\pi|^2]$                          &                        $\num{1.3e-04}\left(\num{2e-05}\right)$ &    {\cellcolor{gray!20}$\num{9e-05} \left(\num{1e-05}\right)$} &                      $\num{1.1e-04}\left(\num{2e-05}\right)$ &  {\cellcolor{gray!20}$\num{1.0e-04} \left(\num{3e-05}\right)$} &                        $\num{7.4e-04}\left(\num{9e-05}\right)$ \\
$\sum_{n=0}^{N-1}\Delta t_n\widehat{\mathds{E}}[|\Delta \widehat{\Gamma}_n^\pi|^2]$ &                          $\num{8e-04}\left(\num{2e-04}\right)$ &  {\cellcolor{gray!20}$\num{5.0e-04} \left(\num{7e-05}\right)$} &                        $\num{8e-04}\left(\num{2e-04}\right)$ &    {\cellcolor{gray!20}$\num{5e-04} \left(\num{1e-04}\right)$} &                        $\num{5.0e+03}\left(\num{8e+02}\right)$ \\
runtime (s)                                                                         &                       $\num{1.20e+03}\left(\num{1e+01}\right)$ &                       $\num{1.44e+03}\left(\num{2e+01}\right)$ &                     $\num{1.19e+03}\left(\num{1e+01}\right)$ &                       $\num{1.43e+03}\left(\num{5e+01}\right)$ &  {\cellcolor{gray!20}$\num{5.7e+02} \left(\num{3e+01}\right)$} \\
\bottomrule
\end{tabular}
}
	\label{tab:Ex1}
\end{table}

\subsection{Example 2: Hamilton-Jacobi-Bellman with LQG control}
The Hamilton-Jacobi-Bellman (HJB) equation is a non-linear PDE derived from Bellman's dynamic programming principle, whose solution is the \emph{value function} of a corresponding \emph{stochastic control} problem. In what follows, we consider the linear-quadratic-Gaussian (LQG) control, which describes a linear system driven by additive noise \cite{han_solving_2018}. The FBSDE system \autoref{eq:fbsde}, associated with the HJB equation has the following coefficients
\begin{align}\label{example:2}
	\begin{split}
		\mu&=\mathbf{0}_d,\qquad \sigma=\sqrt{2}I_d,\qquad
		f(t, x, y, z) = \abs{z}^2,\qquad 
		g(x) = x^TAx + v^Tx + c,
	\end{split}
\end{align}
where $A\in\mathds{R}^{d\times d}, v\in\mathds{R}^{d\times 1}, c\in\mathds{R}$. Unlike in \cite{han_solving_2018}, the hereby considered terminal condition is a quadratic mapping of space. This choice is made so that we have access to \emph{semi-analytical}, pathwise reference solutions $\{(Y_t, Z_t, \Gamma_t)\}_{0\leq t\leq T}$. Indeed, considering the associated parabolic problem \autoref{eq:parabolic_pde}, it is straightforward to show that the solution is given by
\begin{align}
	\begin{split}
		X_t=\sigma W_t,\quad y(t, x) &= x^TP(t)x + Q^T(t)x + R(t),\\
		z(t, x) &= \sigma\left(\left[P(t) + P^T(t)\right]x + Q(t)\right),\quad \gamma(t, x)=\sigma \left[P(t) + P^T(t)\right],
	\end{split}
\end{align}
where the purely time dependent functions $P:[0, T]\to\mathds{R}^{d\times d}, Q:[0, T]\to\mathds{R}^{d\times 1}, R:[0, T]\to\mathds{R}$ satisfy the following set of Riccati type ordinary differential equations (ODE)
\begin{align}\label{eq:hjb:ricatti_odes}
	\begin{split}
		\dot{P}(t) - \left[P(t) + P^T(t)\right]^2&=0,\quad \dot{Q}(t) - 2\left[P(t) + P^T(t)\right]Q(t)=0,\quad \dot{R}(t) + \Tr{P(t) + P^T(t)} - \abs{Q(t)}^2=0,\\
		P(T) &= A,\quad Q(T)=v,\quad R(T)=c,
	\end{split}
\end{align}
with $\dot{P}=\mathrm{d}P/\mathrm{d}t$, $\dot{Q}=\mathrm{d}Q/\mathrm{d}t$ and $\dot{R}=\mathrm{d}R/\mathrm{d}t$. The reference solution is then obtained by integrating \autoref{eq:hjb:ricatti_odes} over a refined time grid of $N_\text{ODE}=10^4$ intervals.\footnote{This is done using \href{https://docs.scipy.org/doc/scipy/reference/generated/scipy.integrate.odeint.html}{scipy.integrate.odeint}.}
We take $A = I_d, v=\mathbf{0}_d, c=0$, $T=0.5$ and fix $x_0=\mathbf{1}_d$.
An interesting feature of the FBSDE system defined by \autoref{example:2} is that the driver is independent of $Y$ meaning that the Malliavin BSDE in \autoref{eq:fbsde_fbsde:malliavin_bsde} can be solved separately from the backward equation. Consequently, the discrete time approximations of $Z$ and $\Gamma$ in \autoref{scheme:osm} do not depend on $\vartheta_y$. Moreover, the driver is quadratically growing in $Z$, in particular, it is only Lipschitz continuous over compact domains. We pick $\vartheta_y=1/2$ and investigate the solution in $d\in\{1, 50\}$.

In \autoref{fig:Ex2:regression} the regression errors of the Deep BSDE approach are assessed in $d=1$. The true regression targets in \autoref{scheme:osm:implementable} are benchmarked according to BCOS. In fact, at time step $n$, the corresponding cosine expansion coefficients are recovered by means of DCT, given neural network approximations $\widehat{Y}_{n+1}^\pi, \widehat{Z}_{n+1}^\pi, \widehat{\Gamma}_{n+1}^\pi$. These coefficients are subsequently plugged in \autoref{bcos:osm} to gather BCOS estimates. For large enough Fourier domains and sufficiently many Picard iterations, the COS error becomes negligible compared to the discretization component and the resulting estimates approximate the true regression labels $\widecheck{Y}_n^\pi, \widecheck{Z}_n^\pi, \widecheck{\Gamma}_n^\pi$. Hence, they can then be used to assess the regression errors induced by the Monte Carlo method. \autoref{fig:Ex2:regression_error:time} depicts these regression errors over time for $N=100$. As it can be seen, the model of Huré et al. \cite{hure_deep_2020} and the OSM scheme result in similar regression error components for the $Y$ process. However, the regression errors of the $Z$ process are three orders of magnitude worse in case of the reference method \cite{hure_deep_2020}, and in fact, dominate the total approximation error at $n=N-1$. In contrast, the OSM estimates -- middle plot of \autoref{fig:Ex2:regression_error:time} -- exhibit the same order of regression error as for the $Y$ process. This demonstrates the advantageous conditional variance behavior of the corresponding OSM estimates, as pointed out in \autoref{remark:conditional_variance}. The regression errors of the $\Gamma$ process show comparable figures.
The cumulative regression errors, corresponding to the second term in \autoref{thm:osm:regression}, are collected in \autoref{fig:Ex2:approximation:convergence}. In case of the model in \cite{hure_deep_2020}, the cumulative regression error of the $Z$ process blows up as the mesh size $\abs{\pi}=T/N$ decreases. On the contrary, the cumulative regression errors in all processes $(Y, Z, \Gamma)$ are at a constant level of $\mathcal{O}(10^{-5})$ for the OSM scheme. In light of \autoref{remark:sgd}, this indicates that the chosen, finite network architecture incorporates a regression bias which cannot be further reduced. In our experiments, we found that it is difficult to decrease this component by changing the number of hidden layers $L$ or neurons per hidden layer $S_l$. Assessing this phenomenon requires a better understanding of both narrow UAT estimates and the convergence of SGD iterations.

In \autoref{fig:Ex2:approximation} the $d=50$ dimensional case is depicted. In order to have dimension independent scales, \emph{relative} mean-squared errors are reported. \autoref{fig:Ex2:approximation:over_time} collects the relative approximation error over the discretized time window when $N=100$. Compared to \cite{hure_deep_2020}, the OSM estimates yield a significant improvement in each part of the solution triple. In particular, the approximation errors of the $Z$ process are three orders of magnitude better with both the parametrized (P) and automatic differentiated (D) approaches. In case of the $\Gamma$ process, two observations can be made. First, the corresponding curve demonstrates that naive automatic differentiation of the $Z$ approximations in \cite{hure_deep_2020} does not provide reliable $\Gamma$s. Moreover, it can be seen that the parametrized version (P) of the Deep BSDE approach given by \autoref{loss:z:parametrized} provides an order of magnitude better average $\Gamma$ errors. The convergence of the total approximation errors is depicted in \autoref{fig:Ex2:approximation:convergence}. The neural network regression estimates converge for both the parametrized (P) and the automatic differentiated (D) loss functions until $N=50$, when the regression bias becomes apparent. Additionally, the convergence of the $\Gamma$ approximations is significantly better in the parametrized case, suggesting that for such a quadratically scaling driver the last term of \autoref{eq:osm:total_error:ad} is a driving error component.

In \autoref{tab:Ex2} means and standard deviations of a collection of error measures are gathered, with respect to $5$ independent runs of the same regression Monte Carlo method, for both $d=1$ and $d=50$. The numbers are in line with the observations above. In particular, we highlight that the error terms corresponding to the $Z$ and $\Gamma$ approximations are four orders of magnitude better than in case of the reference method \cite{hure_deep_2020}. The parametrized version (P) of the Deep BSDE shows consistently better convergence. The total runtime of the neural network regression Monte Carlo approach is moderately increased between $d=1$ and $d=50$. In fact, the average execution time of a single SGD step for the parametrized (P) case in \autoref{loss:z:parametrized} increases from $\num{2.8e-3}(\num{4e-4})$ to $\num{3.3e-3}(\num{4e-4})$ seconds in between $d=1$ and $d=50$. The same numbers for the automatic differentiated formulation (D) in \autoref{loss:z:ad} are $\num{3.8}(\num{4e-4})$ and $\num{4.4e-3}(\num{5e-4})$ seconds. These figures demonstrate the aforementioned methods' scalability for high-dimensional FBSDE systems. Finally, we point out that the OSM estimates are robust over independent runs of the algorithm as showcased by the small standard deviations in \autoref{tab:Ex2}.

\begin{figure}[t]
	\centering
	\begin{subfigure}[t]{\textwidth}
		\centering
		\includegraphics[width=0.75\textwidth]{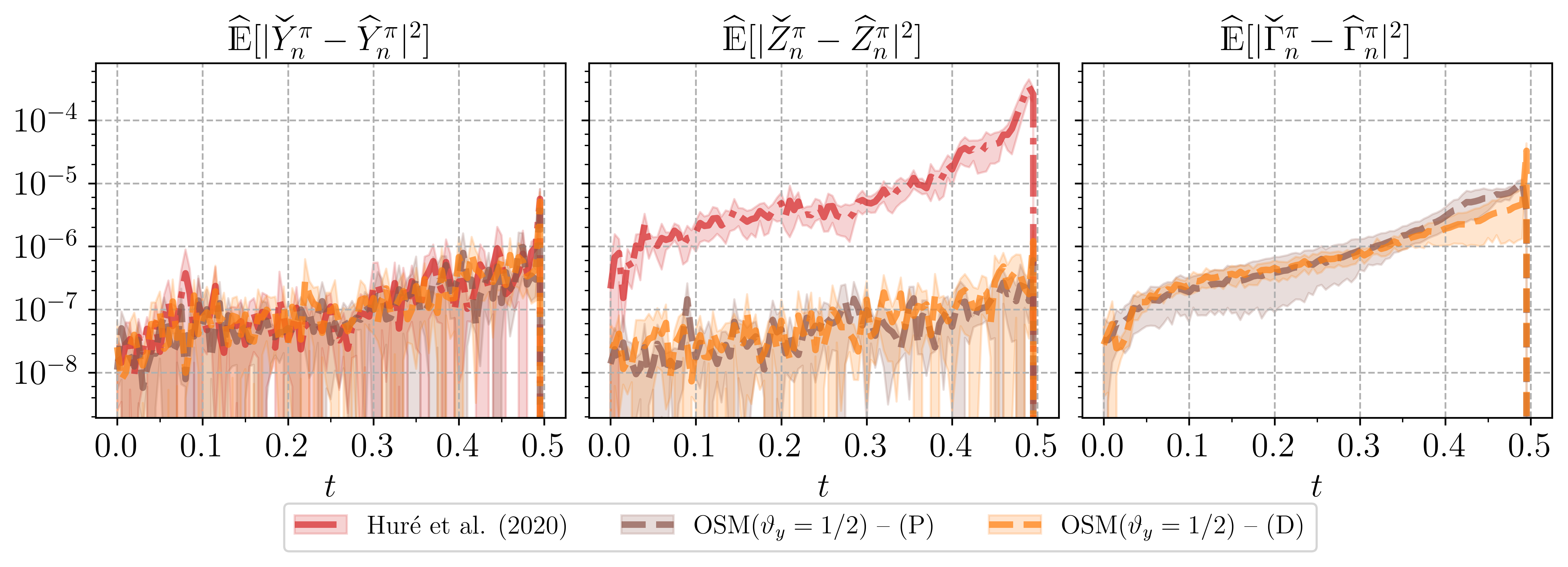}
		\caption{Regression errors over time, $d=1$, $N=100$. From left to right: mean-squared regression errors of the $Y$, $Z$ and $\Gamma$ approximations over the discrete time window.}
		\label{fig:Ex2:regression_error:time}
	\end{subfigure}
	
	\begin{subfigure}[t]{\textwidth}
		\centering
		\includegraphics[width=0.75\textwidth]{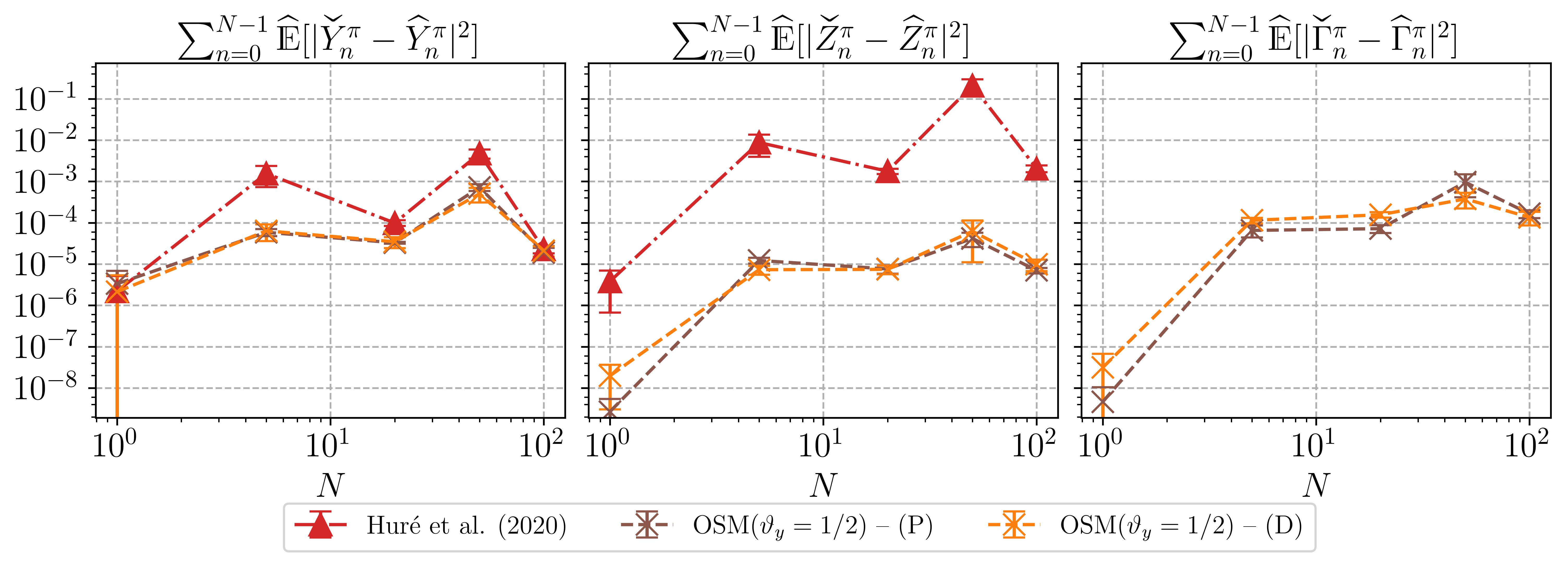}
		\caption{Convergence of cumulative regression errors, $d=1$. From left to right: cumulative regression errors of the $Y$, $Z$ and $\Gamma$ approximations over the number of time steps $N$.}
		\label{fig:Econvergencex2:regression:cumulative}
	\end{subfigure}
	\caption{Example 2 in \autoref{example:2}. Neural network regression errors in $d=1$. The true regression targets of \autoref{scheme:osm:implementable} are identified by BCOS estimates. Mean-squared errors are calculated over an independent sample of $M=2^{10}$ realizations of the underlying Brownian motion. Means and standard deviations are obtained over $5$ independent runs of the algorithm.}
	\label{fig:Ex2:regression}
\end{figure}
\begin{figure}[t]
	\centering
	\begin{subfigure}[t]{\textwidth}
		\centering
		\includegraphics[width=0.75\textwidth]{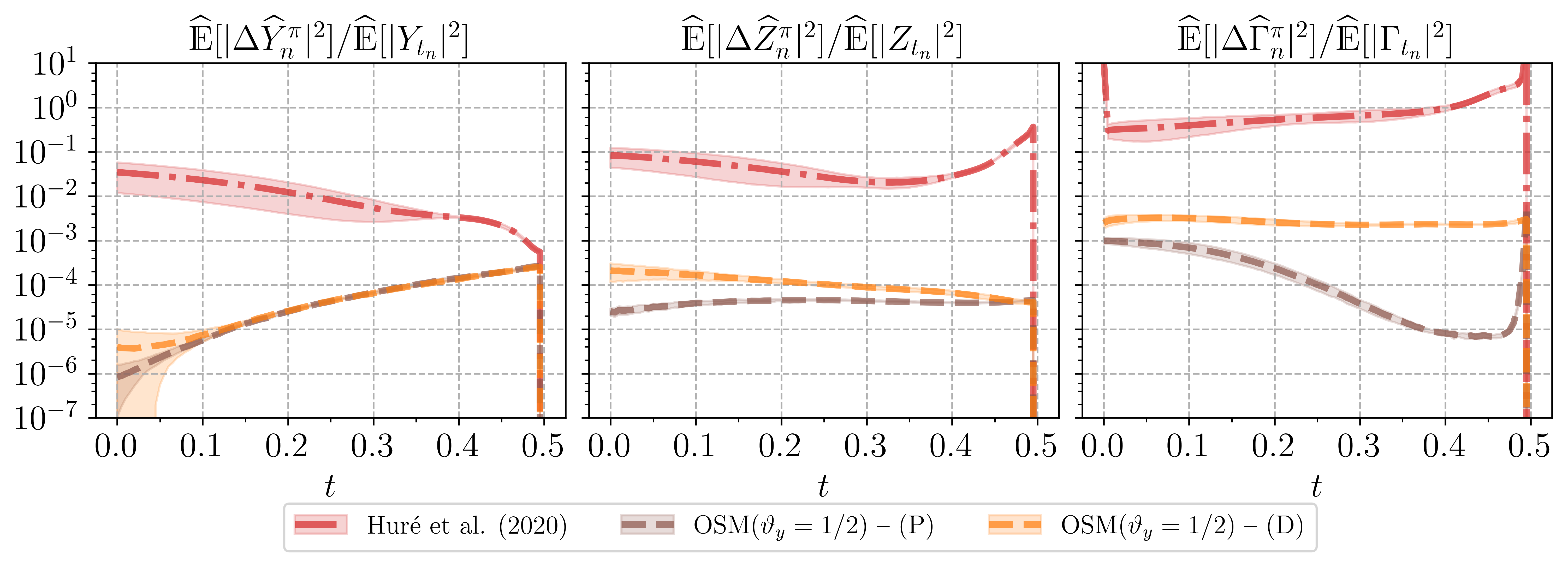}
		\caption{Relative approximation errors over time, $d=50$, $N=100$. From left to right: relative mean-squared approximation errors of $Y, Z$ and $\Gamma$ over the discrete time window.}
		\label{fig:Ex2:approximation:over_time}
	\end{subfigure}
	
	\begin{subfigure}[t]{\textwidth}
		\centering
		\includegraphics[width=0.75\textwidth]{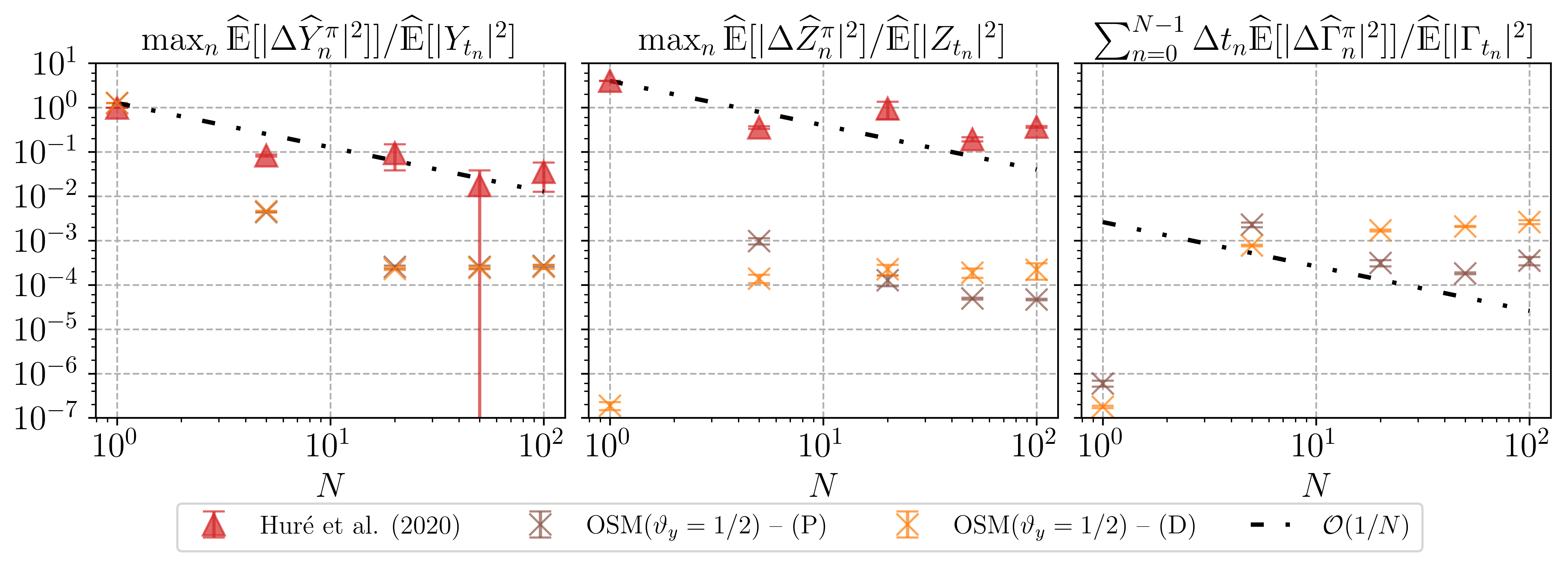}
		\caption{Convergence of relative approximation errors, $d=50$. From left to right: maximum relative mean-squared error of the $Y$, $Z$ approximations; average relative mean-squared error of the $\Gamma$ approximations.}
		\label{fig:Ex2:approximation:convergence}
	\end{subfigure}
	\caption{Example 2 in \autoref{example:2}. $d=50$. Relative approximation errors. Mean-squared errors are calculated over an independent sample of $M=2^{10}$ realizations of the underlying Brownian motion. Means and standard deviations are obtained over $5$ independent runs of the algorithm. $\Gamma$ estimates from Huré et al. in \cite{hure_deep_2020} are obtained via automatic differentiation.}
	\label{fig:Ex2:approximation}
\end{figure}
\begin{table}
	\centering
	\caption{Example 2 in \autoref{example:2}. Summary of Deep BSDE estimates. Mean-squared errors are calculated over an independent sample of $M=2^{10}$ realizations of the underlying Brownian motion. Means and standard deviations (in parentheses) obtained over $5$ independent runs of the algorithm. Best estimates within one standard deviation highlighted in gray. $\Gamma$ estimates from Huré et al. in \cite{hure_deep_2020} are obtained via automatic differentiation.}
	\label{tab:Ex2}
	\begin{subtable}{0.8\textwidth}
		\centering
		\caption{$d=1, N=100$.}
		\label{tab:Ex2:d1}
		\begin{tabular}{llll}
\toprule
{} & \multicolumn{2}{l}{OSM($\vartheta_y=1/2$)} &                                             Huré et al. (2020) \\
{} &                                                            (P) & \multicolumn{2}{l}{(D)} \\
\midrule
$|\Delta \widehat{Y}_0^\pi|/|Y_0|$                                                  &  {\cellcolor{gray!20}$\num{1.1e-03} \left(\num{5e-04}\right)$} &                        $\num{2e-03}\left(\num{1e-03}\right)$ &  {\cellcolor{gray!20}$\num{1.5e-03} \left(\num{3e-04}\right)$} \\
$|\Delta \widehat{Z}_0^\pi|/|Z_0|$                                                  &  {\cellcolor{gray!20}$\num{1.3e-04} \left(\num{9e-05}\right)$} &  {\cellcolor{gray!20}$\num{8e-05} \left(\num{9e-05}\right)$} &                          $\num{1e-03}\left(\num{1e-03}\right)$ \\
$|\Delta \widehat{\Gamma}_0^\pi|/|\Gamma_0|$                                        &  {\cellcolor{gray!20}$\num{1.0e-04} \left(\num{5e-05}\right)$} &                        $\num{2e-04}\left(\num{1e-04}\right)$ &                       $\num{1.05e+00}\left(\num{7e-02}\right)$ \\
$\max_n\widehat{\mathds{E}}[|\Delta \widehat{Y}_n^\pi|^2]$                          &    {\cellcolor{gray!20}$\num{8e-06} \left(\num{2e-06}\right)$} &  {\cellcolor{gray!20}$\num{8e-06} \left(\num{3e-06}\right)$} &                        $\num{1.1e-04}\left(\num{1e-05}\right)$ \\
$\max_n\widehat{\mathds{E}}[|\Delta \widehat{Z}_n^\pi|^2]$                          &    {\cellcolor{gray!20}$\num{8e-07} \left(\num{3e-07}\right)$} &                      $\num{1.4e-06}\left(\num{6e-07}\right)$ &                        $\num{6.4e-03}\left(\num{3e-04}\right)$ \\
$\sum_{n=0}^{N-1}\Delta t_n\widehat{\mathds{E}}[|\Delta \widehat{\Gamma}_n^\pi|^2]$ &    {\cellcolor{gray!20}$\num{8e-07} \left(\num{4e-07}\right)$} &                      $\num{2.8e-06}\left(\num{9e-07}\right)$ &                        $\num{5.5e-03}\left(\num{7e-04}\right)$ \\
runtime (s)                                                                         &                       $\num{1.18e+03}\left(\num{4e+01}\right)$ &                     $\num{1.41e+03}\left(\num{3e+01}\right)$ &  {\cellcolor{gray!20}$\num{5.7e+02} \left(\num{4e+01}\right)$} \\
\bottomrule
\end{tabular}

		%\resizebox{\textwidth}{!}{\input{Figs/HJBLQODE_T0p5_X1_default_final_d1_N100_multiidx.tex}}
	\end{subtable}
	
	\begin{subtable}{0.8\textwidth}
		\centering
		\caption{$d=50, N=100$.}
		\label{tab:Ex2:d50}
		\begin{tabular}{llll}
\toprule
{} & \multicolumn{2}{l}{OSM($\vartheta_y=1/2$)} &                                              Huré et al. (2020) \\
{} &                                                            (P) & \multicolumn{2}{l}{(D)} \\
\midrule
$|\Delta \widehat{Y}_0^\pi|/|Y_0|$                                                  &    {\cellcolor{gray!20}$\num{8e-04} \left(\num{5e-04}\right)$} &    {\cellcolor{gray!20}$\num{1e-03} \left(\num{1e-03}\right)$} &                         $\num{1.7e-01}\left(\num{8e-02}\right)$ \\
$|\Delta \widehat{Z}_0^\pi|/|Z_0|$                                                  &  {\cellcolor{gray!20}$\num{5.0e-03} \left(\num{5e-04}\right)$} &                        $\num{1.4e-02}\left(\num{3e-03}\right)$ &                         $\num{2.8e-01}\left(\num{7e-02}\right)$ \\
$|\Delta \widehat{\Gamma}_0^\pi|/|\Gamma_0|$                                        &  {\cellcolor{gray!20}$\num{3.1e-02} \left(\num{2e-03}\right)$} &                        $\num{4.9e-02}\left(\num{7e-03}\right)$ &                         $\num{3.5e+00}\left(\num{1e-01}\right)$ \\
$\max_n\widehat{\mathds{E}}[|\Delta \widehat{Y}_n^\pi|^2]$                          &  {\cellcolor{gray!20}$\num{2.7e+00} \left(\num{1e-01}\right)$} &                        $\num{2.5e+00}\left(\num{3e-01}\right)$ &                           $\num{7e+01}\left(\num{4e+01}\right)$ \\
$\max_n\widehat{\mathds{E}}[|\Delta \widehat{Z}_n^\pi|^2]$                          &                        $\num{3.4e-02}\left(\num{1e-03}\right)$ &  {\cellcolor{gray!20}$\num{3.1e-02} \left(\num{3e-03}\right)$} &                         $\num{2.8e+02}\left(\num{1e+01}\right)$ \\
$\sum_{n=0}^{N-1}\Delta t_n\widehat{\mathds{E}}[|\Delta \widehat{\Gamma}_n^\pi|^2]$ &  {\cellcolor{gray!20}$\num{4.1e-04} \left(\num{6e-05}\right)$} &                        $\num{3.3e-03}\left(\num{2e-04}\right)$ &                         $\num{2.9e+00}\left(\num{2e-01}\right)$ \\
runtime (s)                                                                         &                       $\num{1.36e+03}\left(\num{1e+01}\right)$ &                       $\num{1.62e+03}\left(\num{4e+01}\right)$ &  {\cellcolor{gray!20}$\num{6.16e+02} \left(\num{1e+01}\right)$} \\
\bottomrule
\end{tabular}

		%\resizebox{\textwidth}{!}{\input{Figs/HJBLQODE_T0p5_X1_default_final_d50_N100_multiidx.tex}}
	\end{subtable}
\end{table}

\subsection{Example 3: space-dependent diffusion coefficients}
Our final example is taken from \cite{milstein_numerical_2006, ruijter_numerical_2016} and it is meant to demonstrate that the conditions in \autoref{ass:error_analysis} can be substantially relaxed. The FBSDE system \autoref{eq:fbsde} is defined by the following coefficients
\begin{align}\label{example:3}
	\begin{split}
		\mu_i(t, x) &= \frac{(1+x_i^2)}{\left(2+x_i^2\right)^3},\quad \sigma_{ij}(t, x) = \frac{1+x_ix_j}{2+x_ix_j}\delta_{ij},\\
		f(t, x, y, z) &= \begin{aligned}[t]
			&\frac{1}{\lambda (t+\tau)}\exp(-\frac{x^Tx}{\lambda (t+\tau)})\left[4\sum_{i=1}^d\frac{x_i^2(1+x_i^2)}{(2+x_i^2)^3} + \sum_{i=1}^d\frac{(1+x_i^2)^2}{(2+x_i^2)^2}\left(1 - 2\frac{x_i^2}{\lambda (t+\tau)}\right)-\sum_{i=1}^d\frac{x_i^2}{t+\tau}\right]\\
			&+\sqrt{\frac{1+y^2+\exp(-\frac{2x^Tx}{\lambda (t+\tau)})}{1+2y^2}}\sum_{i=1}^d \frac{z_ix_i}{(2+x_i^2)^2},\qquad  g(x) = \exp(-\frac{x^Tx}{\lambda (T+\tau)}).
		\end{aligned}
	\end{split}
\end{align}
The analytical solutions are given by
\begin{align}
	\begin{split}
		y(t, x) &= \exp(-\frac{x^Tx}{\lambda(t+\tau)}),\quad z_j(t, x)= -\frac{1+x_j^2}{2+x_j^2}\frac{2\exp(-\frac{x^Tx}{\lambda (t+\tau)})}{\lambda (t+\tau)}x_j,\quad \gamma_{ij}(t, x) = \partial_j z_i(t, x).
	\end{split}
\end{align}
We use $T=10, \lambda=10, \tau=1$, $d=1$ and fix $x_0=\mathbf{1}_d$. Notice that $\mu$ and $\sigma$ are both $C^2_b$. In conjecture with \autoref{appendix:discretization:dx}, this implies that the Euler-Maruyama schemes in \autoref{scheme:sde:euler} and \autoref{scheme:sde:dx:euler} have an $\mathds{L}^2$ convergence rate of order $1/2$. Additionally, by It\^{o}'s formula, the unique solution of the SDE is given by the closed form expression \cite{milstein_numerical_2006}
\begin{align}\label{eq:example3:x}
	X_t = \Lambda(x_0 + \arctan(x_0) + W_t),
\end{align}
where $\Lambda: \mathds{R}\to\mathds{R}$ is defined implicitly $\Lambda(r) + \arctan(r) \coloneqq r$ for any $r\in\mathds{R}$, and applied element-wise. It is straightforward to check that $\Lambda \in C_b^1(\mathds{R}; \mathds{R})$, in particular $\Lambda^{'}(r) = \frac{1+\Lambda^2(r)}{2+\Lambda^2(r)}$ implying that $\Lambda$ is a bijective. In light of the Malliavin chain rule formula in \autoref{lemma:malliavin_chain_rule}, we then also have
\begin{align}\label{eq:example3:dx}
	D_sX_t = \frac{1+\Lambda^2(x + \arctan(x) + W_t)}{2+\Lambda^2(x + \arctan(x) + W_t)}\mathds{1}_{s\leq t}.
\end{align}
We assess the convergence of the Euler-Maruyama estimates in \autoref{scheme:sde:euler}--\autoref{scheme:sde:dx:euler} by solving the non-linear equation in \autoref{eq:example3:x} for each realization of the Brownian motion.\footnote{This is done by \href{https://docs.scipy.org/doc/scipy/reference/generated/scipy.optimize.root.html}{\texttt{scipy.optimize.root}}'s \texttt{df-sane} algorithm which deploys the method in \cite{la_cruz_spectral_2006}.}
The results of the numerical simulations in $d=1$ are given in \autoref{fig:Ex3:approximation} for the parametrized Deep BSDE case and $\vartheta_y=0, 1/2, 1$. 
We see that, in line with \autoref{appendix:discretization:dx}, $D_nX_{n+1}^\pi$ inherits the convergence rate of $X_n^\pi$. The convergence rates of $(\widehat{Y}_n^\pi, \widehat{Z}_n^\pi, \widehat{\Gamma}_n^\pi)$ are of the same order as in \autoref{thm:osm:regression}. The BCOS estimates and the Deep BSDE approach exhibit coinciding error figures until a magnitude of $\mathcal{O}(10^{-6})$ is reached, when the regression bias becomes apparent.
Similar convergence behavior is observed in high-dimensions. The results suggest that the convergence of the OSM scheme can be extended to the non-additive noise case.

\begin{figure}[t]
	\centering
	\includegraphics[width=0.75\textwidth]{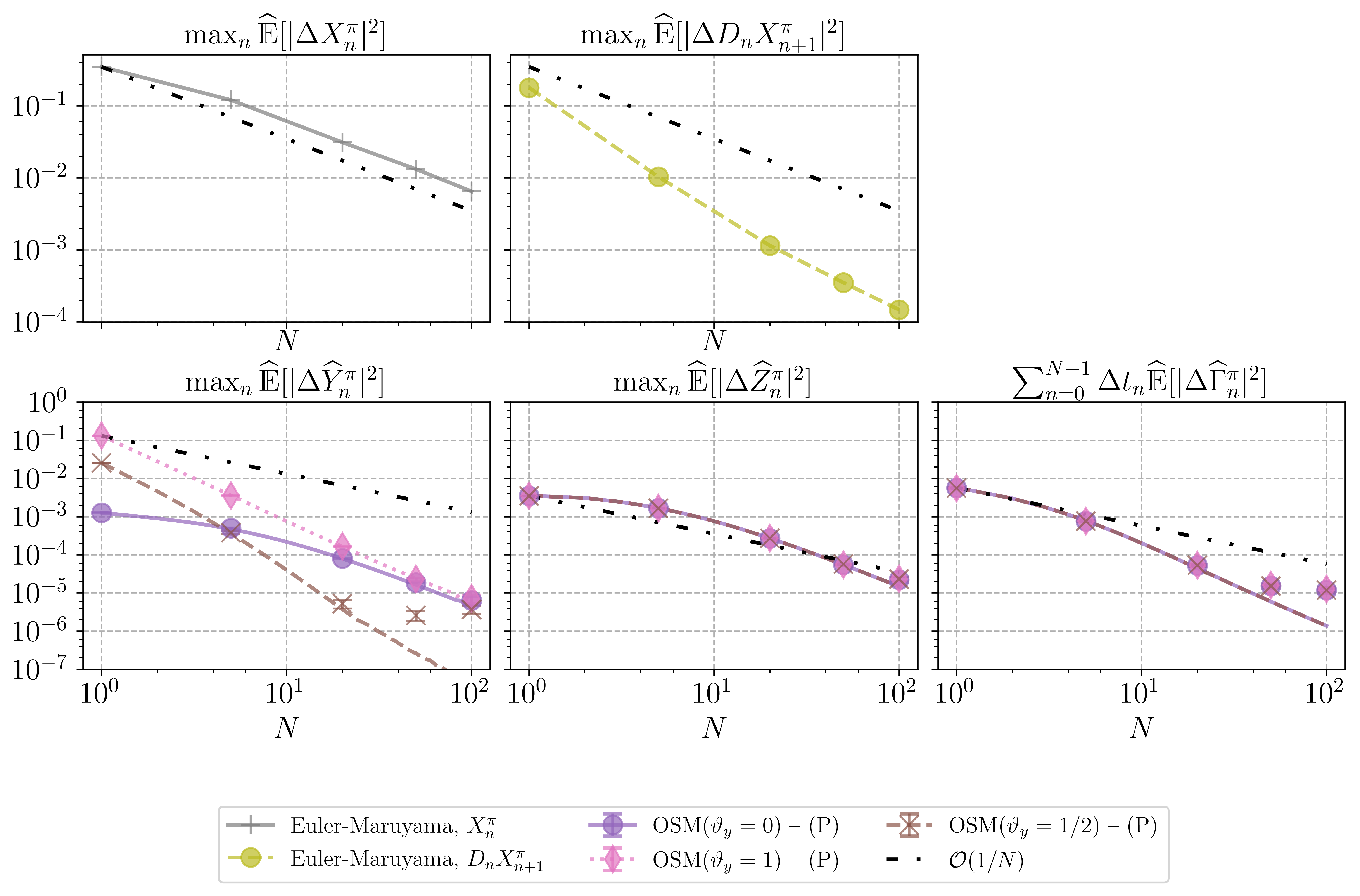}
	\caption{Example 3 in \autoref{example:3}. Convergence of approximation errors for $d=1$. From left to right, top to bottom: maximum mean-squared errors of Euler-Maruyama approximations of $X$ and $DX$; maximum mean-squared approximation errors of $Y$ and $Z$; average mean-squared approximation error of $\Gamma$. Lines correspond to BCOS estimates, scattered error bars to the means and standard deviations of $5$ independent neural network regressions. The mean errors are obtained over an independent sample of $M=2^{10}$ trajectories of the underlying Brownian motion.}
	\label{fig:Ex3:approximation}
\end{figure}

\section{Conclusion}\label{sec:conclusion}

In this paper we introduced the One Step Malliavin (OSM) scheme, a new discretization for Malliavin differentiable FBSDE systems where the control process is estimated by solving the linear BSDE driving the Malliavin derivatives of the solution pair. The main contributions can be summarized as follows. 
The discretization in \autoref{scheme:osm} includes $\Gamma$ estimates, linked to the Hessian matrix of the associated parabolic problem. In \autoref{thm:osm:discretization} we have shown that under standard Lipschitz assumptions and additive noise in the forward diffusion, the aforementioned discrete time approximations admit to an $\mathds{L}^2$ convergence of order $1/2$. We gave two fully-implementable schemes. In case of one-dimensional problems, we extended the BCOS method \cite{ruijter_fourier_2015}, and gathered approximations via Fourier cosine expansions in \autoref{bcos:osm}. For high-dimensional equations, similarly to recent Deep BSDE methods \cite{han_solving_2018, hure_deep_2020}, we formulated a neural network regression Monte Carlo approach, where the corresponding processes of the solution triple are parametrized by fully-connected, feedforward neural networks. We carried out a complete regression error analysis in \autoref{thm:osm:regression} and showed that the neural network parametrizations are consistent with the discretization, in terms of regression biases controlled by the universal approximation property. We supported our theoretical findings by numerical experiments and demonstrated the accuracy and robustness of the proposed approaches for a range of high-dimensional problems. Using BCOS estimates as benchmarks for one-dimensional equations, we empirically assessed the regression errors induced by stochastic gradient descent. Our findings with the Deep BSDE approach showcase accurate approximations for each process in \autoref{scheme:osm:implementable}, and in particular exhibit significantly improved approximations of the $Z$ process for heavily control dependent equations.

\paragraph{Acknowledgments} The first author would like to thank Adam Andersson for the fruitful discussions in the early stages of this work. The first author also acknowledges financial support from the Peter Paul Peterich Foundation via the TU Delft University Fund.
\begin{appendix}
	\section{Convergence of \texorpdfstring{$D_nX_{n+1}^\pi$}{DX}}\label{appendix:discretization:dx}
	We show the convergence of $D_nX_{n+1}^\pi$ estimates of the Euler-Maruyama discretization \autoref{scheme:sde:dx:euler} under the assumptions
	
	\begin{enumerate}[label=($\widetilde{\mathbf{A}}^{\sigma, \mu}_{\arabic*}$), wide, labelwidth=0pt, labelindent=0pt]
		\item $\sigma$ is uniformly bounded;\label{ass:sde:dx:1}
		\item $\mu\in C_b^{0, 1}(\mathds{R}^{d\times 1}; \mathds{R})$, $\sigma\in C_b^{0, 1}(\mathds{R}^{d\times 1};\mathds{R}^{d\times d})$. In particular both of them are Lipschitz continuous in $x.$\label{ass:sde:dx:2}
	\end{enumerate}
	From the estimation \autoref{scheme:sde:dx:euler} and the linear SDE of the Malliavin derivative in \autoref{eq:fbsde_fbsde:malliavin_sde} -- using the inequality $(a+b+c)^2\leq 4(a^2+b^2+c^2)$, on top of the $L^2([0, T];\mathds{R}^{d\times d})$ Cauchy-Schwarz inequality and It\^{o}'s isometry --, it follows
	\begin{align}
		\Expectation{\abs{D_{t_{n}}X_{t_{n+1}} - D_nX_{n+1}^\pi}^2}\leq \begin{aligned}[t]
			&4\Expectation{\abs{\sigma(t_n, X_{t_{n}}) - \sigma(t_n, X_n^\pi)}^2}\\
			&+ 4\Delta t_n \Expectation{\int_{t_{n}}^{t_{n+1}} \abs{\nabla_x \mu(r, X_r)D_{t_{n}}X_r - \nabla_x \mu(t_n, X_n^\pi)\sigma(t_n, X_n^\pi)}^2\mathrm{d}r}\\
			&+ 4\Expectation{\int_{t_{n}}^{t_{n+1}} \abs{\nabla_x \sigma(r, X_r)D_{t_{n}}X_r - \nabla_x\sigma(t_n, X_n^\pi)\sigma(t_n, X_n^\pi)}^2\mathrm{d}r}.
		\end{aligned}
	\end{align}
	Bounded continuous differentiability in \ref{ass:sde:dx:2}, in particular, implies Lipschitz continuity. Furthermore, by the uniform boundedness of the diffusion coefficient and the mean-squared continuity of $D_{t_{n}}X$ in \autoref{mean-squared continuity:dx}, we gather
	\begin{align}
		\Expectation{\abs{D_{t_{n}}X_{t_{n+1}} - D_nX_{n+1}^\pi}^2}\leq 4L_{\sigma}^2\Expectation{\abs{X_{t_{n}} - X_n^\pi}^2} + C\Delta t_n,
	\end{align}
	for any $\Delta t_n<1$. Then, due to the discretization error of the Euler-Maruyama estimates given by \autoref{scheme:sde:euler}, we conclude $\limsup_{\abs{\pi}\to 0} \frac{1}{\abs{\pi}} \Expectation{\abs{D_{t_{n}}X_{t_{n+1}} - D_nX_{n+1}^\pi}^2}<\infty$.
	
	\section{Integration by parts formulas}
	For the formula in \autoref{eq:bcos:general_estimates:1th} we refer to \cite[A.1]{ruijter_fourier_2015}. In order to prove \autoref{eq:bcos:general_estimates:2th}, let $v:[0, T]\times \mathds{R}\to\mathds{R}$ and consider
	\begin{align}
		\CondExpnx{n}{v(t_{n+1}, X_{n+1}^{\pi}(\Delta W_n))\Delta W_n^2} = \CondExpnx{n}{\frac{1}{\sqrt{2\pi \Delta t_n}}\int_\mathds{R} v(t_{n+1}, X_{n+1}^\pi(\nu) \nu^2e^{-\frac{1}{2\Delta t_n}\nu^2}\mathrm{d}\nu},
	\end{align}
	with the Euler-Maruyama approximations $X_{n+1}^\pi(\Delta W_n)=x+\mu(t_n, x)\Delta t_n + \sigma(t_n, x)\Delta W_n$.
	For a sufficiently smooth $v$, integration by parts implies
	\begin{align}
		\CondExpnx{n}{\frac{1}{\sqrt{2\pi \Delta t_n}}\int_\mathds{R} v(t_{n+1}, X_{n+1}^\pi) \nu^2e^{-\frac{1}{2\Delta t_n}\nu^2}\mathrm{d}\nu} =\begin{aligned}[t]
			\mathds{E}_{n}^x\Bigg[\frac{1}{\sqrt{2\pi \Delta t_n}}\Big\{&-\Delta t_n\left[\nu v(t_{n+1}, X_{n+1}^\pi) e^{-\nu^2/(2\Delta t_n)}\right]_{-\infty}^{+\infty}\\
			&+\Delta t_n \int_\mathds{R} v(t_{n+1}, X_{n+1}^\pi) e^{-\frac{1}{2\Delta t_n}\nu^2}\mathrm{d}\nu\\
			&+ \Delta t_n\sigma(t_n, x)\int_{\mathds{R}}\partial_x v(t_{n+1}, X_{n+1}^\pi)\nu  e^{-\frac{1}{2\Delta t_n}\nu^2}\mathrm{d}\nu\Big\}\Bigg],
		\end{aligned}
	\end{align}
	For a $v$ with sufficient radial decay, we therefore conclude that 
	\begin{align}
		\CondExpnx{n}{v(t_{n+1}, X_{n+1}^\pi)\Delta W_n^2} = \Delta t_n \CondExpnx{n}{v(t_{n+1}, X_{n+1}^\pi)} + \Delta t_n^2 \sigma^2(t_n, x) \CondExpnx{n}{\partial_{xx}^2 v(t_{n+1}, X_{n+1}^\pi)},
	\end{align}
	by the estimate in \autoref{eq:bcos:general_estimates:1th}.
	
	Thereupon, given a cosine expansion approximation of $v(t_{n+1}, \rho)\approx  \sideset{}{'}\sum_{k=0}^{K-1} \mathcal{V}_k(t_{n+1})\cos(k\pi\frac{\rho - a}{b-a})$, the corresponding spatial derivative approximations are given by
	$\partial_x v(t_{n+1}, \rho) \approx \sideset{}{'}\sum_{k=0}^{K-1} -\mathcal{V}_{k}(t_{n+1}) \frac{k\pi}{b-a}\sin(k\pi\frac{\rho - a}{b-a}),\quad \partial_{xx}^2v(t_{n+1}, \rho) \approx \sideset{}{'}\sum_{k=0}^{K-1} -\mathcal{V}_{k}(t_{n+1}) \left(\frac{k\pi}{b-a}\right)^2\cos(k\pi\frac{\rho - a}{b-a}).$
	Then the approximations in \autoref{eq:bcos:general_estimates:1th}--\autoref{eq:bcos:general_estimates:2th} follow from the expressions $
	\CondExpnx{n}{\sin(k\pi\frac{X_{n+1}^\pi - a}{b-a})} = \Im{\Phi\left(k\vert x\right)},\quad \CondExpnx{n}{\cos(k\pi\frac{X_{n+1}^\pi - a}{b-a})} = \Re{\Phi\left(k\vert x\right)},$
	where $\Phi(k\vert x)$ is defined as in \autoref{sec:bcos}.
	
	\paragraph{Multi-dimensional extensions.} In case the underlying forward process is a $\mathds{R}^{d\times 1}$-dimensional Brownian motion, the following extension can be given. Let $v: [0, T]\times \mathds{R}^{d\times 1}\to \mathds{R}$ be a scalar-valued function. Then reasoning similar to \cite[A.1]{ruijter_fourier_2015} shows that $\CondExpn{n}{\left(\Delta W_n\right)_{i1} v(t_{n+1}, X_{n+1})} = \sum_{k=1}^d \Delta t_n \CondExpn{n}{\partial_k v(t_{n+1}, X_{n+1}^\pi)}\left(\sigma(t_{n+1}, X_{n+1}^\pi)\right)_{ki}$. In matrix notation
	\begin{align}
		\left(\CondExpn{n}{\Delta W_n v(t_{n+1}, X_{n+1}^\pi)}\right)^T= \Delta t_n \CondExpn{n}{\nabla_x v(t_{n+1}, X_{n+1}^\pi)}\sigma(t_{n+1}, X_{n+1}).
	\end{align}
	Alternatively, for a vector-valued mapping $\psi:[0, T]\times \mathds{R}^{d\times 1}\to \mathds{R}^{1\times d}$, similar arguments give the following, component-wise formula $\CondExpn{n}{\left(\Delta W_n\right)_{i1} \left(\psi(t_{n+1}, X_{n+1})\right)_{1j}} = \sum_{k=1}^d \Delta t_n \CondExpn{n}{\partial_k \left(\psi(t_{n+1}, X_{n+1}^\pi)\right)_{1j}}\left(\sigma(t_{n}, X_{n}^\pi)\right)_{ki}$. In matrix notation
	\begin{align}\label{eq:app:integration-by-parts:multi-d:vector}
		\left(\CondExpn{n}{\Delta W_n \psi(t_{n+1}, X_{n+1}^\pi)}\right)^T= \Delta t_n \CondExpn{n}{\nabla_x \psi(t_{n+1}, X_{n+1}^\pi)}\sigma(t_{n}, X_{n}^\pi),
	\end{align}
	where $\nabla_x \psi$ is the Jacobian matrix of $\psi$.
	
	\section{BCOS estimates}\label{appendix:bcos}
	Let us fix $d=1$.
	The BCOS approximations of the OSM scheme in \autoref{bcos:osm} can be derived as follows.
	Using the definition in \autoref{eq:w_def} and the Euler-Maruyama estimates in \autoref{scheme:sde:dx:euler}, the $\Gamma$ estimates in \autoref{scheme:osm:implementable:dz} can be written according to
	\begin{align}
		D_n\widecheck{Z}_n^\pi = \widecheck{\gamma}_n^\pi(x) \sigma(t_n, x) &= \begin{aligned}[t]
			&\frac{1}{\Delta t_n}\sigma(t_n, x)\left(1 + \Delta t_n \partial_x\mu(t_n, x)\right)\CondExpnx{n}{\Delta W_n w_{n+1}^\pi(\Xhat{n+1}^\pi)}\\
			&+\frac{1}{\Delta t_n}\sigma(t_n, x)\partial_x\sigma(t_n, x)\CondExpnx{n}{\Delta W_n^2 w_{n+1}^\pi(\Xhat{n+1}^\pi)}\\
			&+ \CondExpnx{n}{\Delta W_n \partial_z f(t_{n+1},\Xhat{n+1}^\pi)}\widecheck{\gamma}^\pi_n(x)\sigma(t_n, x).
		\end{aligned}
	\end{align}
	A cosine expansion approximation for $w_{n+1}^\pi(\Xhat{n+1}^\pi)$ and $\partial_z f(t_{n+1}, \Xhat{n+1}^\pi)$ can be obtained by means of DCT, yielding approximations $\{\widehat{\mathcal{W}}_k(t_{n+1})\}_{k=0, \dots, K-1}$, $\{\widehat{\mathcal{F}}_k^z(t_{n+1})\}_{k=0, \dots, K-1}$ respectively. Consequently, plugging these approximations combined with the integration by parts formulas in \autoref{eq:bcos:general_estimates:1th}--\autoref{eq:bcos:general_estimates:2th}, in the estimate above yields
	\begin{align}
		\widehat{\gamma}_n^\pi(x)\sigma(t_n, x) = \begin{aligned}[t]
			&-\sigma^2(t_n, x)(1+\partial_x\mu(t_n, x)\Delta t_n)\sideset{}{'}\sum_{k=0}^{K-1} \frac{k\pi}{b-a}\widehat{\mathcal{W}}_{k}(t_{n+1})\Im{\Phi(k\vert x)}\\
			&+\sigma(t_n, x)\partial_x\sigma(t_n, x)\sideset{}{'}\sum_{k=0}^{K-1} \widehat{\mathcal{W}}_k(t_{n+1})\Re{\Phi(k\vert x)}\\
			&-\Delta t_n\sigma^3(t_n, x)\partial_x\sigma(t_n, x)\sideset{}{'}\sum_{k=0}^{K-1} \left(\frac{k\pi}{b-a}\right)^2\widehat{\mathcal{W}}_k(t_{n+1})\Re{\Phi(k\vert x)}\\
			&-\widehat{\gamma}_n^\pi(x)\Delta t_n\sigma^2(t_n, x)\sideset{}{'}\sum_{k=0}^{K-1} \frac{k\pi}{b-a}\widehat{\mathcal{F}}^z_{k}(t_{n+1})\Im{\Phi(k\vert x)}.
		\end{aligned}
	\end{align}
	The approximation $D_n\widehat{Z}_n^\pi=\widehat{\Gamma}_n^\pi\sigma(t_n, X_n^\pi)$ subsequently follows. The coefficients $\{\widehat{\mathcal	{DZ}}_k(t_{n+1})\}_{k=0,\dots, K-1}$ are calculated by DCT and subsequently plugged into the approximations of the $Z$ process, which follows analogously using the formulas in \autoref{eq:bcos:general_estimates:0th}--\autoref{eq:bcos:general_estimates:1th}. The approximation of the $Y$ process in \autoref{scheme:osm:implementable:y} is identical to \cite{ruijter_fourier_2015} and therefore omitted.
	
\end{appendix}

% % % Default option
\iffalse
\bibliographystyle{siam}
\bibliography{references.bib}
%\printbibliography
\fi
% % %

% % % BibLaTeX option
%\iffalse
%\printbibliography
\printbibliography[heading=bibintoc, title={References}]

%\fi
% % %

% % % AMS option
%\bibliography{references.bib}
% % %

\end{document}